\documentclass{amsart}

\usepackage{paralist, amsmath, amsthm, amssymb, color, graphicx, hyperref, tikz, mathrsfs}
\usetikzlibrary{patterns}

\theoremstyle{plain}
\newtheorem{theorem}{Theorem}[section]
\newtheorem{proposition}[theorem]{Proposition}
\newtheorem{lemma}[theorem]{Lemma}
\newtheorem{corollary}[theorem]{Corollary}

\theoremstyle{definition}
\newtheorem{definition}[theorem]{Definition}
\newtheorem{remark}[theorem]{Remark}
\newtheorem{example}[theorem]{Example}

\newcommand{\OB}[1]{ }
\newcommand{\TK}[1]{ }
\newcommand{\AP}[1]{ }

\newcommand{\ds}{\displaystyle}

\newcommand{\al}{\alpha}
\newcommand{\be}{\beta}
\newcommand{\bal}{\boldsymbol{\alpha}}
\newcommand{\bbe}{\boldsymbol{\beta}}

\newcommand{\Om}{\Omega}

\newcommand{\mA}{\mathcal{A}}
\newcommand{\mB}{\mathcal{B}}
\newcommand{\mC}{\mathcal{C}}

\newcommand{\mF}{\mathcal{F}}
\newcommand{\mI}{\mathcal{I}}
\newcommand{\mK}{\mathcal{K}}
\newcommand{\mP}{\mathcal{P}}
\newcommand{\mQ}{\mathcal{Q}}
\newcommand{\mS}{\mathcal{S}}
\newcommand{\mT}{\mathcal{T}}

\newcommand{\mZ}{\mathcal{Z}}

\DeclareMathOperator\lr{lr}
\DeclareMathOperator\rl{rl}
\DeclareMathOperator{\codim}{codim}

\newcommand{\hT}{\widehat{\T}}

\newcommand{\one}{\mathbf{1}}
\newcommand{\wti}{\widetilde}
\newcommand{\m}{\mathbf{m}}

\DeclareMathOperator{\cork}{cork}
\DeclareMathOperator{\nul}{null}
\DeclareMathOperator{\hypbox}{Box}
\newcommand{\weight}{\textrm{weight}}

\def\Z{\mathbb{Z}}
\def\R{\mathbb{R}}

\def\Ext{\mathrm{Ext}}
\def\Int{\mathrm{Int}}
\def\tExt{\widetilde{\mathrm{Ext}}}
\def\tInt{\widetilde{\mathrm{Int}}}

\def\ext{\mathrm{ext}}
\def\int{\mathrm{int}}

\def\mir{T}

\def\oi{\mathrm{oi}}
\def\oe{\mathrm{oe}}
\def\ie{\mathrm{ie}}

\def\a{\mathbf{a}}
\def\b{\mathbf{b}}
\def\c{\mathbf{c}}
\def\e{\mathbf{e}}

\def\x{\mathbf{x}}

\def\scA{\mathscr{A}}
\def\scP{\mathscr{P}}

\def\T{\mathscr{T}}
\def\mR{\mathscr{R}}
\def\QCF{\mathscr{Q}}
\def\V{\mathbf{V}}

\def\conv{\mathrm{conv}}
\def\wt{\mathit{wt}}

\def\level{\mathrm{level}}

\def\Sup{S^\textrm{up}}
\def\tSup{\tilde S^\textrm{up}}
\def\Slow{S^\textrm{low}}
\def\tSlow{\tilde S^\textrm{low}}
\def\Cup{C^\textrm{up}}
\def\Clow{C^\textrm{low}}

\title{Universal Tutte polynomial}
\author{Olivier Bernardi, Tam\'as K\'alm\'an, Alexander Postnikov}
\thanks{OB is partially supported by NSF grant DMS-1800681.}
\date{\today}

\address{Department of Mathematics, Brandeis University, Waltham, MA 02453, USA}
\email{bernardi@brandeis.edu}

\address{Department of Mathematics, Tokyo Institute of Technology, 
Tokyo 152-8551, Japan}
\email{kalman@math.titech.ac.jp}

\address{Department of Mathematics, MIT, Cambridge, MA 02139, USA}
\email{apost@math.mit.edu}

\begin{document}

\begin{abstract}
The Tutte polynomial is a well-studied invariant of graphs and matroids. We first extend the Tutte polynomial from graphs to hypergraphs, and more generally from matroids to polymatroids, as a two-variable polynomial. 
Our definition is related to previous works of Cameron and Fink and of K\'alm\'an and Postnikov.
We then define the 
universal Tutte polynomial 
$\T_n$, which is a polynomial of degree $n$ in $2+(2^n-1)$ variables that specializes to the Tutte polynomials of all polymatroids (hence all matroids) on a ground set with $n$ elements.
The universal polynomial $\T_n$ admits three kinds of symmetries: translation invariance, $S_n$-invariance, and duality.
\end{abstract}

\maketitle




\section{Introduction}

The Tutte polynomial $T_M(x,y)$ is an important invariant of a matroid $M$. For a graphical matroid $M_G$ associated to a graph $G$, the Tutte polynomial $T_G(x,y):=T_{M_G}(x,y)$ specializes to many classical graph invariants, such as the chromatic polynomial, the flow polynomial, the reliability polynomial etc. 
The Tutte polynomial and its various evaluations were studied in the context of statistical physics (partition functions of the Ising and Potts models), knot theory (Jones and Kauffman polynomials), and many other areas of mathematics and physics. There is a vast literature on the Tutte polynomial; see for instance~\cite[Chapter~10]{Bollobas} or~\cite{EM} for an introduction and references.

This paper contains two main contributions to the theory of the Tutte polynomial. The first is an extension of the Tutte polynomial, from the class of matroids to that of polymatroids. Our notion is related to previous works of K\'alm\'an and Postnikov \cite{Pos,Kal,KP} and of Cameron and Fink \cite{CF}. 
The second contribution is the definition of the \emph{universal Tutte polynomial} $\T_n$, which is a polynomial in $x$, $y$, and $2^n-1$ additional variables $z_I$ indexed by the non-empty subsets $I$ of $[n]$. This polynomial specializes to the Tutte polynomials of all polymatroids (hence all matroids) on a ground set with $n$ elements. The polynomial $\T_n$ admits three kinds of symmetries: translation-invariance, $S_n$-invariance, and duality.

\OB{Added:\\}
Let us mention here that our definition of the Tutte polynomial of a polymatroid does not appear to be directly related to other, known extensions of the Tutte polynomial to the polymatroid (or hypergraph) setting, such as the one defined by Helgason \cite{Hel} (see also \cite{Sta3,Whi}). In fact, it is an open question whether there exists a relation between our Tutte polynomial and the chromatic polynomial of a hypergraph as defined in \cite{Ber}, or the `basic invariant' defined in \cite{AA,AKT}). See Section \ref{sec:rel-other-invariants} for a more thorough discussion.
We also note that our universal Tutte polynomial $\T_n$ is not directly related to the ``multivariate'' Tutte polynomials 
of \cite{BR,Zas,ET} or \cite{Sok}, which all satisfy linear deletion-contraction recurrences. Rather, $\T_n$ is a two-variable refinement of the ``Ehrhart-type'' polynomial of \cite{Pos}, counting lattice points in generalized permutahedra. It gives a polynomial parametrization of the Tutte polynomials of all the polymatroids on the ground set $[n]$ in terms of their rank functions.

\smallskip

\noindent \textbf{A Tutte polynomial for polymatroids.} 
Recall that a matroid over a ground set $[n]$ can be defined in terms of its set of bases, which is a subset of $\{0,1\}^n$ satisfying the \emph{exchange axiom}. Similarly, an integer polymatroid can be defined in terms of its set of bases, which is a finite subset of $\Z^n$ satisfying the appropriate version of the exchange axiom; see Section~\ref{sec:polymatroids} for a precise definition\footnote{In this paper we take a \emph{polymatroid} to mean what is usually called the set of bases of an integer polymatroid. That is, our polymatroids are subsets of $\Z^n$, whose elements are called \emph{bases}.}.
Polymatroids are an abstraction of hypergraphs in the same manner as matroids are an abstraction of graphs.
Indeed, for any graph $G$ the spanning trees of $G$ are the bases of a matroid, and similarly the set of \emph{spanning hypertrees} of any hypergraph is the set of bases of a polymatroid (see Section~\ref{sec:hypergraphs} for definitions and~\cite{tuttebook} for a proof). 
Hence any polymatroid invariant automatically gives a hypergraph invariant.

It is natural to attempt to generalize the Tutte polynomial from matroids to polymatroids using the notions of \emph{internal and external activity}. Recall that the original definition of the Tutte polynomial of a matroid $M$~\cite{crapo,Tut} uses the notion of internal activity $\int(B)$ and external activity $\ext(B)$ of a basis $B$ of the matroid $M$, with respect to some total ordering of the ground set. These definitions are recalled in Section~\ref{sec:matroids}. The Tutte polynomial is then given by the sum
$$
T_M(x,y):=\sum_{B\,:\,\text{basis of } M} x^{\int(B)}\, y^{\ext(B)}.
$$
A remarkable feature of this definition is that, although the activities of each basis depend on the ordering of the ground set, the sum $T_M(x,y)$ is invariant under reordering. 


Motivated by the study of the Homfly polynomial, which is an important invariant in knot theory, K\'alm\'an~\cite{Kal} extended the notions of activity to the setting of polymatroids and used them to define two invariants called interior and exterior polynomial.
Precisely, for a polymatroid $P$, the \emph{interior polynomial} 
and \emph{exterior polynomial} of $P$ are equivalent (via a change of variables) to 
the sums 
$$ 
J_P(x) := \sum_{\a\,:\,\text{basis of }P} x^{\int(\a)}
\quad\text{and}\quad
E_P(y) := \sum_{\a\,:\,\text{basis of }P} y^{\ext(\a)},
$$
respectively, where 
$\int(\a)$ and $\ext(\a)$ are the \emph{internal} and \emph{external activities}, respectively, of $\a$ (see Section~\ref{sec:polymat-Tutte} for the definitions). 

It was shown in~\cite{Kal} that, although the activities $\int(\a)$ and $\ext(\a)$ depend on the natural ordering of the ground set, 
the interior and exterior polynomials (in analogy to $T_M$ above) are independent of this ordering. 
 We refer to this property as \emph{$S_n$-invariance}, where $S_n$ is the symmetric group acting on the ground set $[n]$. 
However, somewhat disappointingly, 
the two-variable polynomial 
\begin{equation}
\label{eq:naive}
\sum_{\a\,:\,\text{basis of }P} x^{\int(\a)}\, y^{\ext(\a)}
\end{equation}
is not $S_n$-invariant for a general polymatroid, even though this is the case for matroids.

A second proof of $S_n$-invariance for the polynomials $J_P$ and $E_P$ is given in~\cite{KP} using a ``point-counting'' perspective. 
Let us denote by $\mP=\conv(P)$ the convex hull of $P$, which is classically known as the \emph{base polytope} of the polymatroid $P$. It is shown in~\cite{KP} that $J_P$ and $E_P$ can be obtained via a certain binomial transformation from the
Ehrhart-style functions (in fact, polynomials)
$$
f(s):=|(\mP + s \nabla)\cap \Z^n|
\quad\text{and}\quad
g(t):=|(\mP + t \Delta)\cap \Z^n|,
$$
respectively, that count numbers of lattice points in Minkowski sums of $\mP$ with scaled simplices. 
Here $\Delta=\conv(\e_1,\dots,\e_n)\subset \R^n$ is the standard coordinate $(n-1)$-simplex 
and $\nabla=-\Delta$.
The obvious $S_n$-invariance of the polynomials $f$ and $g$ implies
the $S_n$-invariance of the 
polynomials $J_P$ and $E_P$.

Cameron and Fink~\cite{CF} combined 
$f$ and $g$ and defined a two-variable invariant $\QCF'_P(x,y)$ of a polymatroid $P$
via a binomial transformation of the two-variable Ehrhart-style polynomial 
$$
\QCF_P(s,t) = |(\mP+s\nabla + t\Delta)\cap \Z^n|. 
$$ 
The $S_n$-invariance of $\QCF_P(s,t)$ and $\QCF'_P(x,y)$ 
is clear from the definitions. Moreover, it is shown in~\cite{CF} that the invariant $\QCF'_P$ is equivalent (up to a change of variables) to the 
Tutte polynomial $T_P$ in the case when $P$ is a matroid. 

In the present paper, we take a different route and define
the \emph{polymatroid Tutte polynomial} $\T_P(x,y)$ in terms
of internal and external activities, as the 
following sum over elements (bases) of a polymatroid $P$:
\begin{equation}
\label{eq:polytutte}
\T_P(x,y) := \sum_{\a\,:\,\text{basis of } P} x^{\oi(\a)} y^{\oe(\a)} (x+y-1)^{\ie(\a)},
\end{equation}
where $\oi(\a)$ is the number of ground set elements (to which we will also refer as \emph{indices}) that are ``only internally'' 
active, but not externally active, with respect to a basis $\a\in P$;
the quantity $\oe(\a)$ is the number of indices which are ``only externally''
active, but not internally active with respect to $\a$;
and $\ie(\a)$ is the number of indices which are both 
``internally and externally'' active. That is, what would have been $(xy)^{\ie(\a)}$ in~\eqref{eq:naive} is replaced with $(x+y-1)^{\ie(\a)}$. In Section~\ref{sec:order-invariance} we show that this suffices to achieve the desired $S_n$-invariance. 

We stress that the notion of activity in~\eqref{eq:polytutte} is the same as in~\cite{Kal} and above.
In Section~\ref{sec:relation-classical-Tutte}
we show that when the polymatroid $P$ is a matroid, then 
the polynomial $\T_P(x,y)$ is related to the classical Tutte polynomial $T_P(x,y)$ by a simple change of variables.
It is also clear from the definition that the polymatroid Tutte polynomial satisfies the 
duality relation
$$\T_{-P}(x,y)=\T_P(y,x).$$

\OB{Added:}
We establish several properties of $\T_P$ such as interpretation for various specializations, and information on coefficients (see Propositions \ref{prop:TP_properties} and \ref{prop:Brylawski}). In particular we give a formula for the top coefficients of $\T_P$ and show that it implies and extend the so-called \emph{Brylawski identities} \cite{Bry} (as well as Gordon's identity \cite{Gor}) for the classical Tutte polynomial of matroids.

As we show in Section~\ref{sec:Cameron-Fink}, our polymatroid Tutte polynomial $\T_P(x,y)$ is essentially
the same as Cameron--Fink's polynomial $\QCF'_P(x,y)$.
However, we believe that the construction of $\T_P(x,y)$ is 
more straightforward, since it does not require 
computing Ehrhart-style polynomials and applying binomial transformations.\\

We give a second expression for the polymatroid Tutte polynomial $\T_P(x,y)$ in Section~\ref{sec:order-invariance}. This 
is the polymatroid analogue of the classical definition of the Tutte polynomial as the generating function of subsets of the ground set counted according to their corank and nullity. Namely, for a polymatroid $P\subset\Z^n$, we show that $\T_P(x,y)$ is equivalent (up to a change of variables) to the weighted sum 
$$
\wti \T_P(u,v):=\sum_{\c\in \Z^n} \wt_P(\c).
$$
Here the weight $\wt_P(\c)$ of an arbitrary lattice point $\c\in \Z^n$ is defined in terms of 
the Manhattan distance $d_1(P,\c)$ between the polymatroid $P$ and 
$\c$, and a splitting of this Manhattan distance into two summands
$d_1(P,\c)= d_1^{>}(P,\c) + d_1^{<}(P,\c)$. Then we have
$\wt_P(\c):=u^{d_1^{>}(P,\c)} \, v^{d_1^{<}(P,\c)}$.

This second expression of $\T_P(x,y)$, through $\wti \T_P$, has the advantage of being explicitly $S_n$-invariant. 
Since $\T_P(x,y)$ specializes to the interior and exterior polynomials, our result gives an alternative proof of the $S_n$-invariance of those as well.

The proof of the equivalence between the two expressions of $\T_P$ relies on a decomposition of the lattice $\Z^n$ into cones indexed by the bases of $P$. This cone decomposition 
is the polymatroid analogue of another classical result about 
matroids. Namely, it generalizes the 
decomposition due to Crapo~\cite{crapo} of the Boolean lattice $(2^{[n]},\subseteq)$ into intervals indexed by the bases of a matroid.

We also give two additional formulas for $\T_P(x,y)$ in terms of \emph{shadows} of polymatroids; 
see Section~\ref{sec:polymatroid_shadows}.
For matroids, these recover known expressions due to Gordon and Traldi~\cite{GT} for the Tutte polynomial as sums over either independent sets or spanning sets of the matroid.


\medskip

\noindent \textbf{A universal Tutte polynomial.} 
The second main contribution of the present paper is to define the \emph{universal Tutte polynomial} $\T_n$ for every positive integer $n$. 
This is a polynomial in the variables $x$, $y$, and the $2^n-1$
variables $z_I$ indexed by the nonempty subsets $I$ of $[n]$. It parametrizes the Tutte polynomials of all polymatroids defined on the ground set $[n]$. 
More precisely, to a polymatroid $P$ one can associate its \emph{rank function} $f$ (which is the function $f\colon2^{[n]}\to\Z$ defined by $f(I)=\max\{\,\sum_{i\in I} a_i\mid(a_1,\ldots,a_n)\in P\,\}$), and the Tutte polynomial $\T_P(x,y)$ is obtained from $\T_n$ by setting $z_I=f(I)$ for all nonempty sets $I\subseteq [n]$.
As we recall in Section~\ref{sec:submodular-function}, the rank functions of polymatroids are exactly the \emph{submodular functions} with values in $\Z$.

The universal Tutte polynomial $\T_n$ has total degree $n$, and $(n-1)!\T_n$ has integer coefficients. It admits three kinds of symmetries: translation-invariance, $S_n$-invariance and duality. 



We give an explicit formula for $\T_n$ in Section~\ref{sec:formula_universal_Tutte}. The expression involves the \emph{draconian sequences} introduced in~\cite{Pos} for the enumeration of lattice points of generalized permutohedra. In fact, the proof of the formula is based on the fact that the base polytope $\mP=\conv(P)$ of a polymatroid $P$ (and hence each of its faces) is a generalized permutohedron. 

As we explain in Section~\ref{sec:Tutte-hypergraph}, the explicit formula for $\T_n$ has a nice combinatorial interpretation when applied in the context of (a class of) hypergraphical polymatroids. 
This allows us to
compute the polymatroid Tutte polynomials of graphical zonotopes in Section~\ref{sec:zonotopes}. 
Namely, 
we show that when $P$ is the polymatroid whose bases are the lattice points of the zonotope associated to a graph $G$, then the polymatroid Tutte polynomial $\T_P(x,y)$ is a specialization of the classical Tutte polynomial $T_G(x,y)$ associated to $G$. 

\medskip

\noindent {\bf Outline of the paper.}
In Section~\ref{sec:matroids} we recall some classical definitions about matroids and the Tutte polynomial. In Section~\ref{sec:polymatroids} we recall some background definitions about polymatroids. In Section~\ref{sec:polymat-Tutte} we define, for each polymatroid $P$, the polymatroid Tutte polynomial $\T_P(x,y)$, and state some of its properties. Our definition is a polymatroid analogue of the definition of the classical Tutte polynomial in terms of activities. In Section~\ref{sec:relation-classical-Tutte} we spell out the relation, for a matroid, 
between its polymatroid Tutte polynomial and its classical Tutte polynomial. 
In Section~\ref{sec:universal-Tutte} we define, for each natural number $n$, the universal Tutte polynomial $\T_n$ which is a polynomial parametrizing the Tutte polynomials of every polymatroid in $\Z^n$. In Section~\ref{sec:formula_universal_Tutte} we give an explicit expression for 
$\T_n$. 
In Section~\ref{sec:TP_properties} we gather the proofs of some basic facts about activities and some properties of the polymatroid Tutte polynomial such as Brylawski's identities. 

Up to this point in the article, the focus is on the notions of internal and external activity. The next five sections deal with
some decompositions of the lattice $\Z^n$, and an $S_n$-invariant definition of the polymatroid Tutte polynomial. 
In Section~\ref{sec:cone-decomposition} we define, for a polymatroid in $
\Z^n$, a canonical decomposition of the lattice $\Z^n$ into cones. In Section~\ref{sec:order-invariance} we use the cone decomposition of $\Z^n$ to give 
the polymatroid analogue of the corank-nullity definition of the classical Tutte polynomial.
It comes in the form of 
another characterization of the polymatroid Tutte polynomial as a sum over elements of $\Z^n$. 
In Section~\ref{sec:polymatroid_shadows} we give two other descriptions of the polymatroid Tutte polynomial $\T_P$ which are halfway between the expression in Section~\ref{sec:polymat-Tutte} and the expression in Section~\ref{sec:order-invariance}. These additional characterizations are based on some refined cone decompositions of $\Z^n$ which we study in Section~\ref{sec:shadows}.
In Section~\ref{sec:Cameron-Fink} we explain the relation between the polymatroid Tutte polynomial and the polymatroid invariant defined by Cameron and Fink~\cite{CF}.

For the rest of the paper we turn our attention to polymatroids associated to hypergraphs.
In Section~\ref{sec:hypergraphs} we recall some background on hypergraphs and the associated polymatroids. 
In Section~\ref{sec:Tutte-hypergraph} we use the explicit formula of $\T_n$ to give another combinatorial formula for the Tutte polynomial of a large class of hypergraphical polymatroids. This class does not contain graphical matroids but it does contain the class of graphical zonotopes. In Section~\ref{sec:zonotopes} we compute the polymatroid Tutte polynomials of graphical zonotopes and show a relation with the Tutte polynomials of the associated graphs. In Section~\ref{sec:boxed} we consider a certain truncation of the polymatroid Tutte polynomial and study its properties in the context of hypergraphs. 

We conclude with some open questions in Section~\ref{sec:conclusion}.

\medskip

\noindent {\bf Notation.} Throughout the paper we use the following notation.
For a positive integer $n$, we let $[n]=\{1,\dots,n\}$.
The set of all subsets of $[n]$ is denoted by $2^{[n]}$.
The set of permutations of $[n]$ is denoted by $S_n$.
We use the symbol $\biguplus$ to indicate the disjoint union of sets.
We denote by $\e_1,\ldots,\e_n$ the elements of the canonical basis of $\R^n$.
Lastly, for a point $\x\in \R^n$, we denote by $x_i$ the $i$'th coordinate of $\x$, so that $\x=\sum_{i=1}^n x_i\,\e_i$.


\section{Background: Tutte polynomials of matroids}
\label{sec:matroids}
In this section we recall the notion of a matroid, and the definition of the Tutte polynomial in terms of internal and external activities.
We fix a positive integer $n$ throughout the section.
\begin{definition}
A \emph{matroid} $M$ on the \emph{ground set} $[n]$ is a non-empty subset of $2^{[n]}$ (whose elements are called \emph{bases}) satisfying the following condition:
\smallskip

\noindent {\underline{Basis Exchange Axiom}.}
For any two bases $A, B\in M$ and any $i\in A\setminus B$, there exists
$j\in B\setminus A$ such that both $(A\setminus\{i\})\cup\{j\}$ and 
$(B\setminus\{j\})\cup\{i\}$ are bases of $M$.
\end{definition}

The Basis Exchange Axiom implies that the bases of $M$ all have the same cardinality. This value is called the \emph{rank} of $M$.



\begin{remark}
In the above definition we identify the matroid with its set of bases.
Note also that the Basis Exchange Axiom is given in its ``strong version;''
this is known to be equivalent to the ``weaker version'' which only states
that for all $A,B\in M$ and all $i\in A\setminus B$, there exists $j\in
B\setminus A$ such that $(A\setminus \{i\})\cup \{j\}\in M$.
\end{remark}

Recall that any graph $G$ has an associated matroid $M_G$. Namely, for a connected graph $G$ with edge set $E=\{e_1,\ldots,e_n\}$ indexed by $[n]$, the bases of the matroid $M_G$ correspond to the spanning trees of $G$. Precisely, $A\subseteq [n]$ is a basis of $M_G$ if and only if $\{\,e_i\mid i\in A\,\}$ is the edge set of a spanning tree of $G$. For a non-connected graph $G$, the matroid $M_G$ is defined similarly but considering the maximal spanning forests of $G$ (consisting of one spanning tree in each connected component of $G$) instead of the spanning trees. A matroid obtained in this manner is called \emph{graphical}.

\begin{definition}\label{def:activity-matroid}
Let $M$ be a matroid on the ground set $[n]$, and let $A\in M$ be a basis of $M$.
An element $i\in A$ is called
\emph{internally active} if there is no $j<i$ such that
$(A\setminus\{i\})\cup\{j\}$ is a basis of $M$.
Let $\Int_M(A)\subseteq A$ denote the set of internally active elements with respect to the basis $A$.

An element $i\in[n]\setminus A$ is called \emph{externally active} if there is no $j<i$ such that
$(A\cup\{i\})\setminus\{j\}$ is a basis of $M$.
Let $\Ext_M(A)\subseteq [n]\setminus A$ denote the set of externally active elements with respect to the basis $A$.
\end{definition}

\begin{remark}\label{rk:fundamental-cycle}
The definitions of activity are sometimes presented in terms of the \emph{fundamental cycle} or \emph{cocycle} of an element $i\in[n]$, with respect to a basis $A$. Namely, an element $i\notin A$ (resp., $i\in A$) is called active if $i$ is the smallest element in the unique cycle of $A\cup\{i\}$ (resp.\ unique cocycle of $([n]\setminus A)\cup \{i\})$. It is easy to see that this is equivalent to Definition~\ref{def:activity-matroid}.
\end{remark}

\begin{definition}
\label{def:classical_Tutte}
The \emph{Tutte polynomial} $T_M(x,y)$ of the matroid $M$ is the following sum over all bases $A$ of $M$:
$$
T_M(x,y) = \sum_{A\in M} x^{|\Int_M(A)|}\, y^{|\Ext_M(A)|}.
$$
\end{definition}

\begin{example}
\label{ex:graph}
\begin{figure}[h]
\centering
\begin{tikzpicture}[scale=.20]

\draw [line width=5,opacity=.3] (8,0) -- (8,6);
\draw [line width=5,opacity=.3] (2,3) to [out=-90,in=150] (3.8,-.4) to [out=-30,in=-150] (8,0);
\draw [line width=5,opacity=.3] (8,6) to [out=30,in=140] (15.5,8) to [out=-40,in=110] (17,6);
\draw [line width=5,opacity=.3] (8,0) to [out=-30,in=-140] (15.5,-2) to [out=40,in=-110] (17,0);

\draw [thick] (8,0) -- (8,6);
\draw [thick] (2,3) to [out=90,in=-150] (3.8,6.4) to [out=30,in=150] (8,6);
\draw [thick] (2,3) to [out=-90,in=150] (3.8,-.4) to [out=-30,in=-150] (8,0);
\draw [thick] (8,6) to [out=30,in=140] (15.5,8) to [out=-40,in=110] (17,6);
\draw [thick] (8,0) to [out=-30,in=-140] (15.5,-2) to [out=40,in=-110] (17,0);
\draw [thick] (17,6) to [out=-70,in=90] (17.7,3) to [out=-90,in=70] (17,0);

\draw [fill] (2,3) circle [radius=.4];
\draw [fill] (8,0) circle [radius=.4];
\draw [fill] (8,6) circle [radius=.4];
\draw [fill] (17,0) circle [radius=.4];
\draw [fill] (17,6) circle [radius=.4];

\node at (2,6) {\small $1$};
\node at (2,-1) {\small $2$};
\node at (7,3) {\small $3$};
\node at (13.5,7.5) {\small $4$};
\node at (13.5,-1.5) {\small $5$};
\node at (18.6,3) {\small $6$};

\end{tikzpicture}
\caption{A graph with a spanning tree.}
\label{fig:graph}
\end{figure}
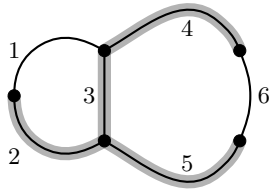

Consider the graph $G$ of Figure~\ref{fig:graph} and identify the set of its edges with the set $[6]$ as shown. The bases of the graphical matroid $M=M_G$ associated to $G$ correspond to the $11$ spanning trees of $G$. With respect to the spanning tree $A$ indicated in the figure, the internally active edges are $4$ and $5$, and the unique externally active edge is $1$. 
Thus, 
$A$ contributes $x^2y$ to the Tutte polynomial $T_M(x,y)$. After computing and adding the other $10$ contributions, one finds
\[T_M(x,y)=\left(
\begin{array}{rrrrr}
y^2&&&&\\
+y&+2xy&+x^2y&&\\
&+x&+2x^2&+2x^3&+x^4
\end{array}\right).\]
\end{example}

Clearly, the notions of internally and externally active elements depend on the ordering of the ground set $[n]$ of the matroid $M$.
However, as we now recall, the Tutte polynomial $T_M(x,y)$ is invariant under reordering the ground set. 
More precisely, let us consider the reordering induced by a permutation $w\in S_n$. The permutation $w$ acts on the power set $2^{[n]}$ as well, 
and gives a new matroid $w(M)$. Tutte 
established the following classical result for graphs, which was later extended to matroids by Crapo. 

\begin{theorem}[\cite{Tut,crapo}] 
For any matroid $M\subseteq 2^{[n]}$ and any permutation 
$w\in S_n$, we have
$$
T_M(x,y) = T_{w(M)}(x,y).
$$
\end{theorem}

Lastly, we recall the other (equivalent) characterization of the Tutte polynomial of a matroid $M\subseteq 2^{[n]}$:
\begin{equation}\label{eq:corank-nullity-mat}
T_M(x,y) = \sum_{S\subseteq [n]}(x-1)^{\cork(S)}(y-1)^{\nul(S)},
\end{equation}
where the \emph{corank} $\cork(S)$ of a subset $S$ is the minimal number of elements one needs to add to $S$ in order to obtain a set $S^+$ containing a basis of $M$, and the \emph{nullity} $\nul(S)$ is the minimal number of elements one needs to delete from $S$ in order to obtain a set $S^-$ contained in a basis of $M$.



\section{Background: polymatroids}
\label{sec:polymatroids}

Polymatroids, introduced by Edmonds~\cite{Edmonds}, are a generalization of matroids.
In this paper we will restrict our attention to so-called \emph{integer polymatroids}, and simply call them \emph{polymatroids}.
There are several alternative ways to define and view polymatroids: 
as $M$-convex sets $P\subset \Z^n$, as submodular functions
$f\colon2^{[n]}\to\Z$, and as the class of polytopes $\mP\subset \R^n$ called 
(integer) generalized permutohedra. We now review these three perspectives in order, and then we introduce polymatroid duality and the generalized Ehrhart polynomial.


\subsection{M-convexity}

Fix a positive integer $n$, and recall that $\e_1,\dots,\e_n$ denotes the canonical basis of $\R^n$.

\begin{definition}
\label{def:Mconvex_polymatroids}
An \emph{M-convex set} is any subset $P\subset \Z^n$
that satisfies:
\smallskip 

\noindent {\underline{Exchange Axiom}.}
For any $\a=(a_1,\ldots, a_n),\b=(b_1,\ldots, b_n)\in P$, and any $i\in [n]$ such that $a_i>b_i$, there exists $j\in[n]$ such that $a_j<b_j$ and both 
$\a-\e_i+\e_j$ and $\b+\e_i-\e_j$ belong to $P$.

\smallskip

A 
\emph{polymatroid} 
is a finite nonempty M-convex set.
The elements of such a set are called the \emph{bases} of the polymatroid.
\end{definition}

The Exchange Axiom implies that any M-convex set belongs to an
affine hyperplane $\{\,x_1+\cdots+x_n=c\,\}$, for some constant
$c\in\Z$. The constant $c$ is called the \emph{level} of $P$ and is denoted by $\level(P)$.
It is also clear that adding the same vector to each element of an M-convex set yields another M-convex set.


\begin{remark}
Everywhere in this paper the term ``polymatroid'' stands for a subset of $\Z^n$ (something that may be called ``integer polymatroid'' elsewhere).
On the other hand, our polymatroids are slightly more general than Edmonds' 
original notion~\cite{Edmonds} in that Edmonds' polymatroids 
belong to the positive orthant $\R_{\geq 0}^n$. We do not need this
condition, so we omitted it in order to simplify the notation.
Since our constructions are all translation-invariant, the reader who prefers polymatroids to belong to the positive orthant may assume so without loss of generality.
\end{remark}

\begin{remark}\label{rk:mat-are-polymat}
(Matroids as polymatroids)
Let $M\subseteq 2^{[n]}$ be a matroid (viewed as the collection of its bases as in Section~\ref{sec:matroids}). By identifying each subset $B\subseteq [n]$ with the vector 
$\b=\sum_{i\in B} \e_i \in \{0,1\}^n$, it
is clear that any matroid $M$ yields a polymatroid $P(M) :=
\{\,\b \mid B\in M\,\}$. This shows that matroids with
ground set $[n]$ are in bijection with polymatroids contained in
$\{0,1\}^n$. 
\end{remark}


\subsection{Submodular functions}\label{sec:submodular-function}


We use~\cite{Edmonds} and~\cite[chapter 44]{sch} as general references for results mentioned in this subsection.

\begin{definition}
A function $f\colon2^{[n]}\to\R$ is \emph{submodular} if 
$f(\varnothing)=0$ and, 
for any subsets $I,J\subseteq [n]$, we have
\begin{equation}
\label{eq:submod}
f(I)+f(J)\geq f(I\cup J)+f(I\cap J).
\end{equation}
\end{definition}

Polymatroids $P\subset \Z^n$ bijective with integer submodular functions $f\colon2^{[n]}\to \Z$.
Indeed, for a polymatroid $P\subset \Z^n$, the \emph{rank function} $f = f_P$ of $P$, 
defined by 
$$
f(I)=\max_{\a\in P}\sum_{i\in I} a_i
$$
for all $I\subseteq [n]$, is a submodular function. Conversely, for an integer submodular function 
$f\colon2^{[n]}\to \Z$, the associated polymatroid $P= P_f$ is given by 
\[
P_f:=\left\{\x\in \Z^n\,\biggm|\,\sum_{i\in I} x_i \leq f(I) \text{ for any subset }I\subseteq [n], \text{ and }
\sum_{i\in[n]} x_i = f([n])\right\}.
\]

The same scheme applies to infinite M-convex sets, which correspond to submodular functions that may take $\infty$ as a value.
Note that if we translate the M-convex set $P$
by a vector $\mathbf{v}\in\Z^n$, then its rank function $f=f_P$ shifts by the 
function $I\mapsto\sum_{i\in I}v_i$ (which satisfies~\eqref{eq:submod} with equality). 




\begin{remark}[Rank functions of matroids]
The rank function $f\colon2^{[n]}\to \Z$ of a matroid $M\subseteq 2^{[n]}$ 
is usually defined to associate, to a subset $I\subseteq[n]$, the maximal cardinality of an intersection of $I$ with a basis of $M$. This 
is equal to the rank function of the associated polymatroid $P(M)\subset\Z^n$. 
In fact, from the point of view of rank
functions, matroids correspond to polymatroids such that the rank function $f$
satisfies $0\leq f(I)\leq f(I\cup\{i\})\leq
f(I)+1$ for any $I\subseteq [n]$ and any $i\in[n]$.
\end{remark}

\subsection{Generalized permutohedra}\label{sec:permutohedra}

We mention two definitions below; for the proof of their equivalence, and several other alternative characterizations 
of generalized permutohedra, see \cite[Theorem~15.3]{PRW}.



\begin{definition}~\cite[Section 6]{Pos} 
\label{def:gen_perm}
A \emph{generalized permutohedron} is a convex polytope $\mP$ in $\R^n$ 
such that every edge of $\mP$ is parallel to a vector of the form $\e_i
- \e_j$ for some $i,j\in[n]$ (where, as usual, $\e_1,\dots,\e_n$ denotes the canonical basis of $\R^n$).

Equivalently, a generalized permutohedron 
is a polytope in $\R^n$ 
whose normal fan 
is a coarsening of the Coxeter fan, that is, the normal fan of the 
\emph{standard permutohedron}
$\Pi_n:=\conv\{\,(w(1),\dots,w(n))\mid w\in S_n\,\}\subset \R^n$.
\end{definition}

To any matroid $M\subseteq 2^{[n]}$ one can associate the \emph{matroid polytope} (also known as \emph{matroid base polytope})
$$
\mP(M) := \conv\left\{\sum_{i\in B} \e_i\,\biggm|\,B\in M\right\}\subset \R^n. 
$$
In~\cite{GGMS}, Gelfand, Goresky, MacPherson and Serganova give the following 
characterization of matroid polytopes.
\begin{theorem}~\cite{GGMS} 
\label{th:GGMS}
A polytope $\mP\subset\R^n$, whose vertices all belong to the set $\{0,1\}^n$, 
is a matroid polytope $\mP(M)$ of some matroid $M$ if and only if every edge of $\mP$ is parallel to a vector of the form $\e_i - \e_j$.
\end{theorem}

Clearly, Definition~\ref{def:gen_perm} of generalized permutohedra is an
extension of this characterization of matroid polytopes. Thus every matroid
polytope is a generalized permutohedron. Moreover, matroid polytopes are
exactly the generalized permutohedra whose vertices belong to the set
$\{0,1\}^n$.

Generalized permutohedra $\mP\subset \R^n$ can also be described by inequalities, namely
$$
\mP = \mP_f :=\left\{\x\in \R^n \,\biggm|\, \sum_{i\in I} x_i \leq f(I)\text{ for all }I \subseteq [n],
\text{ and } \sum_{i\in[n]}x_i = f([n])\right\},
$$
where $f\colon2^{[n]}\to \R$ is a submodular function.
This gives a one-to-one correspondence between generalized permutohedra 
and $\R$-valued submodular functions. 

A generalized permutohedron $\mP$ is \emph{integer} if all its vertices
lie in $\Z^n$, or, equivalently, if the corresponding
submodular function $f$ is integer-valued.
Integer generalized permutohedra $\mP\subset \R^n$ 
are in a one-to-one correspondence with polymatroids $P\subset\Z^n$. The correspondence is simply given by 
$$
P=\mP\cap \Z^n\quad\text{and, conversely,}\quad \mP = \conv(P).
$$

\begin{remark}[Polymatroids and Minkowski sums]\label{rk:polymat-Minkowski}
Recall that the \emph{Minkowski sum} of two subsets $\mA,\mB$ of $\R^n$ is defined as
$$\mA+\mB=\{\,\a+\b\mid \a\in \mathcal \mA,\, \b\in \mathcal \mB\,\}.$$
It is easy to see that the class of generalized permutohedra in $\R^n$ is closed under taking Minkowski sums.
Indeed, for two submodular functions $f,g\colon2^{[n]}\to\Z$,
the function $f+g$ is clearly submodular, and hence
the Minkowski sum $\mP_f+\mP_g$ of the associated 
generalized permutohedra is the 
generalized permutohedron $\mP_{f+g}$.
Although it is less trivial, the identity also holds at the level of polymatroids\OB{Maybe we should add a reference here to justify this fact}: the Minkowski sum $P_f + P_g$ of the polymatroids associated to $f$ and $g$ is the polymatroid $P_{f+g}$. Hence, the class of polymatroids in $\Z^n$ is closed under Minkowski sums.
\end{remark}

\begin{figure}[h]
\centering
\begin{tikzpicture}[scale=.35]

\draw [->] (0,-1) -- (0,7.2);
\draw [->] (0,-1) -- (-7,-4.5);
\draw [->] (0,-1) -- (7.2,-2.2);
\node at (-7.6,-4.8) {\small $x_1$};
\node at (8.1,-2.35) {\small $x_2$};
\node at (0,8) {\small $x_3$};
\draw [help lines] (0,6) -- (-6,-4) -- (6,-2) -- cycle;
\draw [fill=lightgray,lightgray] (0,-3) circle [radius=.3];
\draw [fill=lightgray,lightgray] (3,-2.5) circle [radius=.3];
\draw [fill=lightgray,lightgray] (-1.5,-1) circle [radius=.3];
\draw [fill=lightgray,lightgray] (1.5,-.5) circle [radius=.3];
\draw [fill=lightgray,lightgray] (4.5,0) circle [radius=.3];
\draw [fill=lightgray,lightgray] (-3,1) circle [radius=.3];
\draw [fill=white,white] (0,1.5) circle [radius=.4];
\draw [fill=lightgray,lightgray] (0,1.5) circle [radius=.3];
\draw [fill=lightgray,lightgray] (3,2) circle [radius=.3];
\draw [fill=lightgray,lightgray] (-1.5,3.5) circle [radius=.3];
\draw [fill=lightgray,lightgray] (1.5,4) circle [radius=.3];
\draw [fill=lightgray,lightgray] (0,6) circle [radius=.3];

\node [right] at (.3,6) {\small $(0,0,4)$};

\end{tikzpicture}
\caption{A polymatroid 
(with bases indicated by large dots) in $\Z^3$.}
\label{fig:polymatroid}
\end{figure}
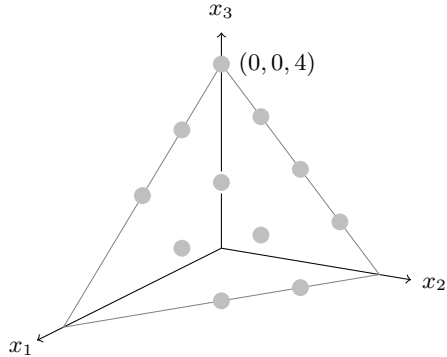

\begin{example}
\label{ex:polymatroid}
Figure~\ref{fig:polymatroid} shows a polymatroid $P\subset\Z^3$ which has 11 bases. The rank function $f\colon2^{[3]}\to\Z$ of $P$
is given by $f(123)= f(12)=f(13)=f(23)=4$, $f(1)=2$, $f(2)=3$, $f(3)=4$, and $f(\varnothing)=0$. 
(Here $123$ is an abbreviation for $\{1,2,3\}$, etc.)
The polymatroid $P_H$ lies in the plane $x_1+x_2+x_3 = 4$, and is given  by the inequalities $0\leq x_1\leq 2$, $0\leq x_2\leq 3$, $0\leq x_3\leq 4$. 
Note that the three non-negativity constraints are equivalent to upper bounds, by values of the rank function given above,
on the three possible sums of two coordinates.

The convex hull $\mP =\conv(P)\subset\R^3$ is a generalized permutohedron.
Indeed, every edge of $\mP$ is parallel to a vector of the form $\e_i-\e_j$.
Notice that 
$\mP$ can be expressed as a Minkowski sum 
of coordinate simplices $\Delta_I:=\conv\{\,\e_i \mid i\in I\,\}$
(which are themselves generalized permutohedra) as follows:
$$
\mP = \Delta_{13} + 2 \Delta_{23} + \Delta_{123}.
$$
\end{example}

\subsection{Polymatroid duality}
\label{sec:duality}

Given a matroid $M$ with ground set $[n]$, the \emph{dual matroid} is 
$$M^*=\{\,[n]\setminus A\mid A\in M\,\}.$$
In other words, the bases of $M^*$ are the complements of the bases of $M$. 
It is clear from the definitions that for any polymatroid $P\subset\Z^n$, the set 
$$-P:=\{\,(-x_1,\ldots,-x_n)\mid(x_1,\ldots,x_n)\in P\,\}$$ 
is a polymatroid. 
We call $-P$ the \emph{dual} of $P$.
Note that there is a slight discrepancy between the notions of duality for matroids and polymatroids. Indeed, the polymatroid $P(M^*)$ is obtained from the polymatroid $-P(M)$ by translating it by the vector $(1,1,\ldots,1)$. 
We will ignore this issue because, 
as we will see below, there is very little difference for us between a polymatroid and its translates.

\subsection{Generalized Ehrhart polynomial}\label{sec:gen-Ehrhart-poly}
It turns out that the number of bases $|P|$ of a polymatroid $P=P_f$ or,
equivalently, the number of integer lattice points $|\mP\cap
\Z^n|$ of an integer generalized permutohedron $\mP=\mP_f$, is a
polynomial function of the values $(f(I))_{I\subseteq[n]}$ of its rank function $f$. This fact
follows from a general result of Brion~\cite{Brion} about lattice points of 
polytopes, as explained in~\cite[Sections~11 and 19]{Pos}. 

\begin{proposition}[\cite{Pos}] \label{prop:generalized-Ehrhart}
Let $n$ be a positive integer, and let $(z_I)_{\varnothing \neq I\subseteq [n]}$ be a set of variables indexed by the $2^{n}-1$ non-empty subsets of $[n]$.
\begin{compactenum}
\item 
There exists a unique polynomial $E_n((z_I))$ such that for any 
 integer-valued submodular function $f\colon2^{[n]}\to\Z$, the number of bases in the polymatroid $P_f=\mP_f\cap \Z^n$ of rank function $f$ is equal to the evaluation $E_n((f(I)))$. 
\item \label{ketto}
The polynomial $E_n$ has degree $n-1$, and $(n-1)!\,E_n$ has integer coefficients.
\end{compactenum}
\end{proposition}

The polynomial $E_n$ is called the \emph{generalized Ehrhart polynomial} of permutohedra.
An explicit formula for $E_n$ is given in~\cite[Theorem~11.3]{Pos}. This formula, which is related to \emph{hypertrees} (or, equivalently, \emph{draconian sequences}) will be recalled in Section~\ref{sec:formula_universal_Tutte}.
We will in fact use the following easy corollary about the number of (relative) interior points of polymatroids.

\begin{corollary}\label{cor:Dragon-exists}
Let $P\subset \Z^n$ be a polymatroid, and let $f$ be its rank function. Let $\mP=\conv(P)$ be the associated permutahedron, of dimension $\dim(\mP)$, and let $\mP^\circ$ be its \emph{relative interior} (the interior of $P$ in the affine subspace it spans). 
The number of interior lattice points
$|\mP^\circ\cap \Z^n|$ is given by the evaluation $(-1)^{\dim(\mP)}E_n((-f(I)))$ of the generalized Ehrhart polynomial. 

In particular, for the class of polymatroids of fixed dimension $d$, the number $|\mP^\circ\cap \Z^n|$ of interior points is a polynomial function of the rank values $(f(I))_{I\subseteq [n]}$ with coefficients in $\frac{1}{(n-1)!}\,\Z$. 
\end{corollary}

\begin{proof}
From Proposition~\ref{prop:generalized-Ehrhart} we see that the classical Ehrhart polynomial $\mathfrak{i}_{\mP}(t)$ of $\mP$ (defined by $\mathfrak{i}_{\mP}(k)=|k\mP\cup \Z^n|$ for all $k\in\Z_{>0}$) is given by 
$$\mathfrak{i}_{\mP}(t)=E_n((t\,f(I))).$$
Hence the Ehrhart--Macdonald reciprocity theorem \cite{Mac}
gives 
$$\ds |\mP^\circ\cap \Z^n|=(-1)^{\dim(\mP)}\mathfrak{i}_{\mP}(-1)=(-1)^{\dim(\mP)}E_n((-f(I))).$$
The other assertion is immediate from part \eqref{ketto} of Proposition~\ref{prop:generalized-Ehrhart}.
\end{proof}



\section{Polymatroid Tutte polynomial}
\label{sec:polymat-Tutte}

In this section we define the main object of this paper, the polymatroid Tutte polynomial, and start
discussing its properties. We begin by defining activities for bases of a polymatroid.

\begin{definition}\cite[section~5.2]{Kal}. \label{def:activity-polymatroid}
Let $P\subset\Z^n$ be a polymatroid (or more generally an M-convex set, not necessarily finite). 
For $\a\in P$, an index $i\in [n]$ is called
\emph{internally active} if there is no $j<i$ such that $\a-\e_i+\e_j\in P$.
Let $\Int(\a)=\Int_P(\a)\subseteq [n]$ denote the set of internally active indices with respect to $\a$.

For $\a\in P$, an index $i\in [n]$ is called
\emph{externally active} if there is no $j<i$ such that $\a+\e_i-\e_j\in P$.
Let $\Ext(\a)=\Ext_P(\a)\subseteq [n]$ denote the set of externally active indices with respect to $\a$.
\end{definition}

Note that the smallest index $i=1$ is always simultaneously internally active and externally active for every element $\a$ of $P$. In Figure~\ref{fig:polymatroid_activities} we indicated the activities of the bases of the polymatroid of Example~\ref{ex:polymatroid}.


We mention the following analogue of the characterization of activities, 
in terms of ``minimality in the fundamental cycle or cocycle'' that we mentioned in Remark~\ref{rk:fundamental-cycle}.

\begin{lemma}\label{lem:activity-fundamentalcocycle}
Let $P\subset \Z^n$ be a polymatroid, with rank function $f$. Let $\a\in P$ be a basis, and let $i\in[n]$.
\begin{compactenum} 
\item The index $i$ is externally active for $\a$ if and only if there exists a subset $I\subseteq [n]$ such that $i=\min(I)$ and $\sum_{i\in I}a_i=f(I)$.
\item The index $i$ is internally active for $\a$ if and only if there exists a subset $I\subseteq [n]$ such that $i=\min(I)$ and $\sum_{i\in [n]\setminus I}a_i=f([n]\setminus I)$.
\end{compactenum}
\end{lemma}
The proof of Lemma~\ref{lem:activity-fundamentalcocycle} is delayed to Section~\ref{sec:TP_properties}. We point out that in the context of matroids, the minimal subset $I$ containing $i$ and such that $\sum_{i\in I}a_i=f(I)$ would be the fundamental cycle of the element $i$ with respect to the basis $\a$ (assuming $i$ is an external element).


We can now define the main character of this story.

\begin{definition}
\label{def:polymatroid_Tutte_poly}
Let $P\subset\Z^n$ be a polymatroid. 
The \emph{polymatroid Tutte polynomial} $\T_P(x,y)$ is defined as
$$
\T_P(x,y):= \sum_{\a\in P} x^{\oi(\a)}\,
y^{\oe(\a)}\,
(x+y-1)^{\ie(\a)},
$$
where
$\oi(\a):=|\Int(\a)\setminus\Ext(\a)|$,
$\oe(\a):=|\Ext(\a)\setminus\Int(\a)|$,
and $\ie(\a):= |\Int(\a)\cap\Ext(\a)|$.
\end{definition}

In 
words, the first exponent $\oi(\a)$ in the formula of $\T_P(x,y)$ is the number of
indices $i\in[n]$ which are ``only internally'' active, but not externally
active. The second exponent $\oe(\a)$ is the number of indices which are ``only
externally'' active, but not internally active. The third exponent $\ie(\a)$ is
the number of indices which are both ``internally and externally'' active.
Because $1$ is such an index for any $\a\in P$, the polymatroid Tutte polynomial
is always divisible by $x+y-1$.

This polymatroid invariant is a natural generalization of the classical Tutte polynomial of matroids. More precisely, as we show in Section~\ref{sec:relation-classical-Tutte}, for any matroid $M$ the polynomials $T_M(x,y)$ and $\T_{P(M)}(x,y)$ are equivalent up to a change of variables.

\begin{example}
\label{ex:polynomial}
Let us compute the Tutte polynomial of the polymatroid $P\subset\Z^3$ from Example~\ref{ex:polymatroid}, which is represented in Figures~\ref{fig:polymatroid} and~\ref{fig:polymatroid_activities}.
The first index $1\in[3]$ is automatically active, both
internally and externally, for any basis of $P$. To decide whether $2$ is
internally (resp., externally) active for a basis $\a\in P$, we check whether
the point $\a-(\e_2-\e_1)$ (resp., $\a+(\e_2-\e_1)$) is a basis of $P$ or not.
From Figure~\ref{fig:polymatroid_activities} we see that there are five bases $\a\in P$ 
such that $\a+(\e_1-\e_2)\not\in P$, so that 2 is internally active for those bases.
Similarly, the index 2 is externally active for five bases. 
Regarding the internal (resp., external) activity of the index $3$ for a basis $\a\in P$, 
the criterion is whether $\a-(\e_3-\e_1)\not\in P$ and $\a-(\e_3-\e_2)\not\in P$ both hold 
(resp., $\a+(\e_3-\e_1), \, \a+(\e_3-\e_2)\not\in P$).
The index $3$ is internally active for two bases and externally active for one base.

At the right of Figure~\ref{fig:polymatroid_activities} we listed the contributions of each basis to the polymatroid Tutte polynomial. More precisely, we wrote the part of each contribution which comes from the indices $2$ and $3$, in that order, and omitted the factor $(x+y-1)$ corresponding to the index 1. 
Summing all the contributions gives
\[\T_{P}(x,y)=(x+y-1)\left(
\begin{array}{rrr}
y^2&&\\
+2y&+2xy&\\
+2&+3x&+x^2
\end{array}
\right).\]
\begin{figure}[h]
\centering
\begin{tikzpicture}[scale=.3]

\begin{scope}[shift={(-14,0)}]
\node at (14,10) {\small internally};
\node [above right] at (10,-3.6) {\small $2$ is};
\node [above right] at (10,-4.6) {\small active};
\node [left] at (25,-3.5) {\small $3$ is active};

\draw [fill=lightgray,lightgray] (15.5,7.5) circle [radius=.3];
\draw [fill=lightgray,lightgray] (14,4.75) circle [radius=.3];
\draw [fill=lightgray,lightgray] (17,4.75) circle [radius=.3];
\draw [fill=lightgray,lightgray] (12.5,2) circle [radius=.3];
\draw [fill=lightgray,lightgray] (15.5,2) circle [radius=.3];
\draw [fill=lightgray,lightgray] (18.5,2) circle [radius=.3];
\draw [fill=lightgray,lightgray] (14,-.75) circle [radius=.3];
\draw [fill=lightgray,lightgray] (17,-.75) circle [radius=.3];
\draw [fill=lightgray,lightgray] (20,-.75) circle [radius=.3];
\draw [fill=lightgray,lightgray] (15.5,-3.5) circle [radius=.3];
\draw [fill=lightgray,lightgray] (18.5,-3.5) circle [radius=.3];

\draw (17.6,-4.6) -- (14,2) -- (17.6,8.6) -- (10,8.6) -- (10,-4.6) -- cycle;
\draw (14.75,-4.25) rectangle (25,-2.75);
\end{scope}


\begin{scope}[shift={(2,0)}]
\node at (16,10) {\small externally};
\node [below left] at (22,8.6) {\small $2$ is active};
\node [right] at (8,7.5) {\small $3$ is active};

\draw [fill=lightgray,lightgray] (15.5,7.5) circle [radius=.3];
\draw [fill=lightgray,lightgray] (14,4.75) circle [radius=.3];
\draw [fill=lightgray,lightgray] (17,4.75) circle [radius=.3];
\draw [fill=lightgray,lightgray] (12.5,2) circle [radius=.3];
\draw [fill=lightgray,lightgray] (15.5,2) circle [radius=.3];
\draw [fill=lightgray,lightgray] (18.5,2) circle [radius=.3];
\draw [fill=lightgray,lightgray] (14,-.75) circle [radius=.3];
\draw [fill=lightgray,lightgray] (17,-.75) circle [radius=.3];
\draw [fill=lightgray,lightgray] (20,-.75) circle [radius=.3];
\draw [fill=lightgray,lightgray] (15.5,-3.5) circle [radius=.3];
\draw [fill=lightgray,lightgray] (18.5,-3.5) circle [radius=.3];

\draw (16.4,-4.6) -- (18.5,-.75) -- (13.4,8.6) -- (22,8.6) -- (22,-4.6) -- cycle;
\draw (8,6.75) rectangle (16.25,8.25);
\end{scope}

\begin{scope}[shift={(16,0)}]
\node at (15.5,7.5) {\small $(x+y-1)\cdot y$};
\node at (14,4.75) {\small $x\cdot1$};
\node at (17,4.75) {\small $y\cdot1$};
\node at (12.5,2) {\small $x\cdot1$};
\node at (15.5,2) {\small $1\cdot1$};
\node at (18.5,2) {\small $y\cdot1$};
\node at (14,-.75) {\small $x\cdot1$};
\node at (17,-.75) {\small $1\cdot1$};
\node at (20,-.75) {\small $y\cdot1$};
\node at (15.5,-3.5) {\small $x\cdot x$};
\node at (18.5,-3.5) {\small $y\cdot x$};
\end{scope}

\end{tikzpicture}
\caption{
Internal and external activities of the elements of the polymatroid $P$ from Figure~\ref{fig:polymatroid}. On the left and in the middle we indicate the elements of $P$ for which the indices $2$ and $3$ are internally/externally active.
On the right we show the contribution of each element to the polynomial $\T_P(x,y)/(x+y-1)$. 
}
\label{fig:polymatroid_activities}
\end{figure}
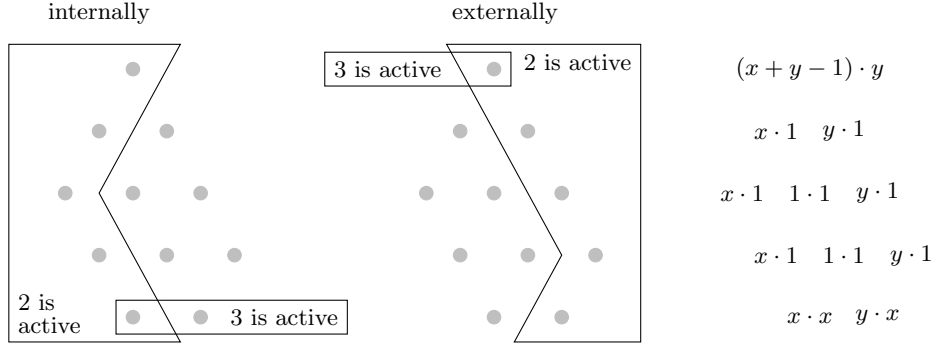
\end{example}

It is immediately clear from the definitions that the polymatroid Tutte polynomial is invariant under affine translations of polymatroids.

\begin{lemma}
\label{lem:shifts}
For any polymatroid $P\subset\Z^n$ and 
any $\c\in \Z^n$, we have 
$$
\T_P(x,y) = \T_{\c+P}(x,y).
$$
\end{lemma}

Another easy property is the relation between the Tutte polynomial of a polymatroid $P$ and its dual $-P$.

\begin{lemma}
\label{lem:duality}
For any polymatroid $P$ we have
$$
\T_P(x,y) = \T_{-P}(y,x).
$$ 
\end{lemma}

\begin{proof}
Clearly, $i\in \Int_P(\a)$ if and only if $i\in \Ext_{-P}(-\a)$.
Likewise, $i\in \Ext_P(\a)$ if and only if $i\in \Int_{-P}(-\a)$.
\end{proof}

Lemma~\ref{lem:duality} is reminiscent of the relation between the Tutte
polynomial of a matroid $M$ and its dual $M^*$. 
Indeed, Lemma~\ref{lem:duality}
is an extension of this known duality relation (via the relation between the
Tutte polynomials of matroids and polymatroids discussed in Section~\ref{sec:relation-classical-Tutte}).

A far less obvious property of $\T_P$ is its \emph{$S_n$-invariance}, that is to say, its invariance with respect to permutations of the coordinates of $\Z^n$. 
Here we consider the natural action of the symmetric group $S_n$ on $\Z^n$ by permutations of coordinates. 
That is, for a permutation $w\in S_n$ and a point $\a=(a_1,\ldots,a_n) \in \Z^n$ we define $w(\a)=(a_{w(1)},\ldots,a_{w(n)})$, and for $P\subset \Z^n$ we define
$$w(P):=\{\,w(\a)\mid\a\in P\,\}.$$
We will prove the following theorem in Section~\ref{sec:order-invariance}.
 
\begin{theorem}
\label{th:order_independence}
Let $P\subset\Z^n$ be a polymatroid.
Then, for any permutation $w\in S_n$,
$$
\T_P(x,y) = \T_{w(P)}(x,y).
$$
In other words, the Tutte polynomial $\T_P(x,y)$ does not depend on 
the choice of ordering of the coordinates.
\end{theorem}

Theorem~\ref{th:order_independence} 
is a subtle claim. It 
does not imply that the double statistics
$\a\mapsto(|\Int(\a)|, |\Ext(\a)|)$, on elements $\a$ of $P$, 
are equidistributed with each other for different orderings of the coordinates. The following example illustrates this point.

\begin{example}
For the polymatroid represented in Figure~\ref{fig:polymatroid_activities}, the ``naive'' polynomial $N_P=\sum_{\a\in P}x^{|\Int_P(\a)|}y^{|\Ext_P(\a)|}$ is not $S_n$-invariant. Indeed $N_P=xy(x^2+xy^2+xy+3x+3y+2)$, while for $w=132$ one gets $N_{w(P)}=xy(x^2+2xy+y^2+3x+2y+2)$.
\end{example}

However, as we now explain, these pathologies do not occur for polymatroids which are ``generic enough''.
We say that a polymatroid is of \emph{generic type}, if the generalized permutahedron $\mP=\conv(P)$ has normal fan equal to the braid arrangement (as is the case for the classical permutahedron). 

\begin{corollary}\label{cor:order-independence-statistics}
If $P\subseteq\Z^n$ is a polymatroid of generic type, then the double statistics
$\a\mapsto(|\Int(\a)|, |\Ext(\a)|)$, on elements $\a$ of $P$, is $S_n$-invariant.
\end{corollary}

We delay the proof of Corollary~\ref{cor:order-independence-statistics} to Section~\ref{sec:TP_properties}. Interestingly, the $S_n$-invariance of the double statistics also occurs for matroids, which are far from generic.
A more immediate consequence of Theorem~\ref{th:order_independence} is that the 
single statistics $\a\mapsto |\Int(\a)|$ and $\a\mapsto |\Ext(\a)|$ 
on $P$ are both $S_n$-invariant:


\begin{corollary} 
Let $P\subset\Z^n$ be a polymatroid. 
The polynomials 
$$
\T_P(x,1)=J_P(x)=\sum_{\a\in P} x^{|\Int(\a)|}
\quad\text{and}\quad
\T_P(1,y)=E_P(y)=\sum_{\a\in P} y^{|\Ext(\a)|}
$$
are $S_n$-invariant. That is to say, $\T_{w(P)}(x,1)=\T_P(x,1)$ and $\T_{w(P)}(1,y)=\T_P(1,y)$ for any $w\in S_n$.
\end{corollary}

The above corollary was originally proved in~\cite{Kal}, where 
$I_P(x)= x^n\, \T_P(x^{-1},1)$ and 
$X_P(y)=y^n\, \T_P(1,y^{-1})$
are called the \emph{interior} and \emph{exterior} polynomials of $P$, respectively.

We now list some additional properties of the polymatroid Tutte polynomial. For polymatroids $P\in \Z^n$ and $Q\in \Z^m$, we define their \emph{direct sum} as
$$P\oplus Q:=\{\,(a_1,a_2,\ldots,a_n,b_1,b_2,\ldots,b_m)\mid\a\in P,\,\b\in Q\,\}\subset \Z^{n+m},$$ 
and their \emph{aggregated sum} as
$$P\boxplus Q:= \{\,(a_1+b_1,a_2,\ldots,a_n,b_2,\ldots,b_m)\mid\a\in P,\,\b\in Q\,\}\subset \Z^{n+m-1}.$$
These are easily seen to be polymatroids. 

\begin{proposition} 
\label{prop:TP_properties}
Let $n$ be a positive integer, let $P\subset\Z^n$ be a polymatroid, and let $\mP=\conv(P)$ be the associated generalized permutahedron.
\begin{compactitem}
\item[(a)] The polymatroid Tutte polynomial $\T_P(x,y)$ has degree $n$. 
\item[(b)] 
The coefficients of $\T_P(x,y+1)$ and $\T_P(x+1,y)$ are non-negative integers, but $\T_P(x,y)$ always has both positive and negative coefficients. 
\item[(c)] The polynomial $\T_P(x,y)$ is divisible by $(x+y-1)^{\codim(P)}$, where $\codim(P)=n-\dim(\mP)$ is the codimension of $\mP$. 
\item[(d)] The top homogeneous component of $\T_P(x,y)$ equals $(x+y)^n$. Equivalently, for all $k\in [n]$, the coefficient of $x^k y^{n-k}$ in $\T_P(x,y)$ is ${n\choose k}$.
\item[(e)] For any polymatroid $Q$, we have
$$\T_{P\oplus Q}(x,y)=\T_{P}(x,y)\T_{Q}(x,y)\quad\text{and}\quad \T_{P\boxplus Q}(x,y)=\frac{\T_{P}(x,y)\T_{Q}(x,y)}{x+y-1}.$$
\item[(f)] Let $i\in[n]$. For $\a\in\Z^n$ we put $\widehat{\a}^i=(a_1,\ldots,a_{i-1},a_{i+1},\ldots,a_n)\in\Z^{n-1}$. 
Let $f$ be the rank function of $P$, and let $r_i=f(\{i\})+f([n]\setminus\{i\})-f([n])$. 
If $r_i=0$, then 
\begin{equation}\label{eq:deletion-coloop}
\T_P(x,y)=(x+y-1)\T_{\widehat{P}^{i}}(x,y),
\end{equation} 
where $\widehat{P}^{i}:=\{\,\widehat{\a}^i\mid\a\in P\,\}$.
If $r_i=1$, then 
\begin{equation}\label{eq:deletion-contraction}
\T_P(x,y)=x\T_{P\setminus i}(x,y)+y\T_{P/ i}(x,y), 
\end{equation}
where $P\setminus i:=\{\,\widehat{\a}^{i}\mid\a\in P,~ a_i=f(\{i\}-1)\,\}$ and $P/i:=\{\,\widehat{\a}^{i}\mid\a\in P,~ a_i=f(\{i\})\,\}$.
\item[(g)] $\T_P(1,1)=|P|=|\mP\cap \Z^n|$ and $\T_P(0,0)=(-1)^{\codim(P)}|\mP^\circ\cap \Z^n|$, where $\mP^\circ$ is the \emph{relative interior} of $\mP$ (its interior in the affine subspace it spans).
\item[(h)] 
Consider the \emph{Minkowski differences} of $\mP$ with the standard simplex $\Delta=\conv(\e_i,~i\in[n])$ and the reversed simplex $\nabla=\conv(-\e_i,~i\in[n])$: 
$$\mP-\Delta:=\{\,\a\in\R^n \mid\a+\Delta\in \mP\,\}~\text{ and }~\mP-\nabla:=\{\,\a\in\R^n\mid\a+\nabla\in \mP\,\}.$$
The numbers of lattice points in these polytopes are given by
$$\frac{\T_P(x,y)}{(x+y-1)}\bigg|_{x=1,y=0}=|(\mP-\Delta)\cap \Z^n|\text{ and }\frac{\T_P(x,y)}{(x+y-1)}\bigg|_{x=0,y=1}=|(\mP-\nabla)\cap \Z^n|,$$
respectively. 
\end{compactitem}
\end{proposition}

We postpone the proof of Proposition~\ref{prop:TP_properties} to Section~\ref{sec:TP_properties}.
We mention that Property (f) is a (partial) generalization to the polymatroid setting of the deletion-contraction identity satisfied by the classical Tutte polynomial; see Section~\ref{subsec:deletion-contraction} for further discussion of this point. We also mention, with respect to Property (h), that it is shown in Section~\ref{sec:Cameron-Fink} that for all positive integers $k$, the number of lattice points in the Minkowski sums $\mP+k\Delta$ and $\mP+k\nabla$ for $k\in \Z\geq 0$ are also captured by $\T_P(x,y)$. 


\section{Classical Tutte polynomial vs polymatroid Tutte polynomial}\label{sec:relation-classical-Tutte}

Recall from Remark~\ref{rk:mat-are-polymat} that any matroid $M$ (viewed as a set of bases) can be identified with a polymatroid
$$
P(M):=\left\{\,\sum_{i\in A} \e_i\biggm| A\in M\,\right\}. 
$$

Thus, for a matroid $M$, we have the classical Tutte polynomial 
$T_M(x,y)$ and the polymatroid Tutte polynomial $\T_{P(M)}(x,y)$.
We claim that these polynomials, although not literally equal, are related by a simple change of variables.

Let us point out a few superficial differences between $T_M(x,y)$ and $\T_{P(M)}(x,y)$. 
\begin{enumerate}
\item
 The polymatroid Tutte polynomial $\T_{P}$ is invariant under affine translations of 
 $P$, whereas $T_M(x,y)$ is not invariant under translation (not even for translations of $P(M)$ which stay within $\{0,1\}^n$).
For instance, for a matroid $M$ with a single basis the classical Tutte polynomial is $T_{M}(x,y) = x^{d}\, y^{n-d}$, where $d$ is the rank. 
By contrast, for a polymatroid $P$ with a single basis the polymatroid Tutte polynomial is $\T_{P}(x,y)=(x+y-1)^n$, independently of the level.
\item
The polymatroid Tutte polynomial $\T_P(x,y)$ is always divisible by $(x+y-1)$. 
In contrast, the classical Tutte polynomial $T_M(x,y)$ is never divisible by $(x+y-1)$. 
To see this, recall that a \emph{factor} of a matroid is a subset $S$ of the ground set $[n]$ so that the rank function $f$ satisfies $f(S)+f([n]\setminus S)=f([n])$. Matroids can be restricted to any subset of their ground set and when we restrict to a pair of factors $S$ and $[n]\setminus S$, then we have $T_{M|_S}\cdot T_{M|_{[n]\setminus S}}=T_M$. Finally when a matroid only has trivial factors, then its matroid Tutte polynomial cannot be divisible by $x+y-1$ because substituting $x=0$, $y=1$ in it yields a positive quantity (bounded below by the so called $\beta$ invariant, see~\cite{crapobeta}).
\end{enumerate}

One reason for the discrepancy between $T_M(x,y)$ and $\T_{P(M)}(x,y)$ is that the definitions of activities for matroids and polymatroids
are not exactly equivalent. The following trivial lemma expresses
the relationship between the two notions.

\begin{lemma}
\label{lem:activity_correction}
Let $M\subset 2^{[n]}$ be a matroid, and let $P=P(M)\subset \Z^n$ 
be the corresponding polymatroid.
Let $A\in M$ be a basis of $M$, and let $\a=\sum_{i\in A}\e_i$ be the 
corresponding point of $P$.
Then the sets $\Int_M(A)$ and $\Ext_M(A)$ 
of internally/externally active elements with respect to the basis $A$ of the matroid $M$
are related to the sets $\Int_P(\a)$ and $\Ext_P(\a)$ 
of internally/externally active indices, with respect to the point $\a$ 
of the polymatroid $P$, as follows:
$$
\Int_P(\a) = \Int_M(A) \cup ([n]\setminus A)
\quad\text{and}\quad
\Ext_P(\a) = \Ext_M(A) \cup A.
$$
\end{lemma}

Since each basis $A$ of $M$ has the same cardinality (the rank $d$ of $M$), the discrepancy noted in Lemma~\ref{lem:activity_correction} would only change the Tutte polynomial by a monomial factor. A more substantial effect comes from the presence of $(x+y-1)^{\ie(\a)}$ in Definition~\ref{def:polymatroid_Tutte_poly}, which replaces what would be $(xy)^{\ie(\a)}$ if we followed Definition~\ref{def:classical_Tutte}. 
Indeed, it is this change that leads to order-independence and some other favorable properties. The exact relation between the classical and polymatroid Tutte polynomials is as follows. 

\begin{theorem}
\label{thm:oldandnew}
Let $M\subset 2^{[n]}$ be a matroid of rank $d$ on the ground set $[n]$, 
and let $P=P(M)\subset\{0,1\}^n$ be the corresponding polymatroid as defined in Remark~\ref{rk:mat-are-polymat}.
Then we have
\begin{equation}\label{eq:classical-to-poly}
\T_P(x,y) = x^{n-d}\, y^{d}\, T_M\left({x+y-1\over y}, {x+y-1\over x}\right).
\end{equation}
\end{theorem}

\begin{proof}
Let $A$ be a basis of $M$. Let
$\iota:=|\Int_M(A)|$ and
$\epsilon:=|\Ext_M(A)|$
be the usual internal and external 
activities of $A$ (in the matroidal sense). 
For the corresponding point $\a$ of the polymatroid
$P=P(M)$, all of these $\iota+\epsilon$ internally or externally active indices 
become both internally and externally active at the same time
(in the polymatroidal sense). In addition, by Lemma~\ref{lem:activity_correction}, there will be another 
$n-d-\epsilon$ internally active indices 
and
$d-\iota$ externally active indices 
with respect to the point $\a\in P$.
Thus if the basis $A\in M$ makes the contribution $x^\iota\,y^\epsilon$ 
to the classical
Tutte polynomial $T_M(x,y)$, then the corresponding point $\a\in P$
makes the contribution 
$$
(x+y-1)^{\iota+\epsilon}\, x^{n-d-\epsilon}\, y^{d-\iota} =
x^{n-d}\,y^{d}\, \left({x+y-1\over y}\right)^\iota \, 
\left({x+y-1\over x}\right)^\epsilon
$$
to the polymatroid Tutte polynomial $\T_P(x,y)$.
This implies the stated relation between the two polynomials.
\end{proof}

\begin{example}
Let $M$ be the graphical matroid of Example~\ref{ex:graph}.
If we compute the polymatroid Tutte polynomial of $P=P(M)$, we obtain 
\[\T_P(x,y)=(x+y-1)\left(
\begin{array}{rrrrrr}
y^5&&&&&\\
-y^4&+5xy^4&&&&\\
&-4xy^3&+10x^2y^3&&&\\
&+xy^2&-11x^2y^2&+10x^3y^2&&\\
&&+5x^2y&-10x^3y&+5x^4y&\\
&&-x^2&+3x^3&-3x^4&+x^5
\end{array}
\right).\]
The polynomials $T_M$ and $\T_P$ are related by the substitution indicated in Theorem~\ref{thm:oldandnew}. Observe that, when viewed in triangular form, the two polynomials produce the same row-sums and column-sums of coefficients, owing to the fact that $xy$ and $x+y-1$ agree when $x=1$ or $y=1$.
\end{example}

We note that the inverse of the conversion formula \eqref{eq:classical-to-poly} for a matroid $M$ on $[n]$ of rank $d$ is
\begin{equation*}
T_M(x,y)=\frac{(x+y-xy)^n}{x^{n-d}y^d}\,\T_{P(M)}\left(\frac x{x+y-xy},\frac y{x+y-xy}\right).
\end{equation*}

\medskip

As we now explain, the Brylawski identities for the Tutte polynomial $T_M$ of a matroid $M$ are a direct consequence of Proposition \ref{prop:TP_properties}(d) about the top coefficients of $\mT_{P(M)}$. Recall that \emph{Brylawski identities} are linear relations between the coefficients of $T_M$. These identities can be written as follows. Consider the power series
$$B(x,y)=\sum_{i,j\geq 0}{i\choose j}x^i y^j = \sum_{i\geq 0}x^i(1+y)^i=\frac{1}{1-x(1+y)}.$$
It was shown by Brylawski \cite{Bry} that for all integers $p,q$ such that $q\geq 0$ and $p< n+q$,
\begin{equation}\label{eq:Brylawski}
[x^p y^q]B(x,y)\,T_M(x,-1/y)=0.
\end{equation}
Equivalently, denoting $t_{i,j}=[x^iy^j]T_M(x,y)$, and using the fact that $t_{i,j}=0$ for $i<0$ this identity can be written as
$$\forall q\geq 0,~\forall p< n+q,~~\sum_{i=0}^p\sum_{j=0}^i{i\choose j}(-1)^j \, t_{p-i,-q+j}=0.$$
For instance, the case $(p,q)=(0,0)$ gives $t_{0,0}=0$ for all $n>0$.
The case $(p,q)=(1,0)$ gives $t_{1,0}+t_{0,0}-t_{0,1}=0$, hence $t_{1,0}=t_{0,1}$, for all $n>1$.
 
Additionally it was proved by Gordon \cite{Gor} that for all  $q\geq 0$, 
\begin{equation}\label{eq:Gordon}
[x^{n+q}y^q]B(x,y)\,T_M(x,-1/y)=(-1)^{n-d}.
\end{equation}
Of course all the relations \eqref{eq:Brylawski} and \eqref{eq:Gordon} are not all independent, and the case $q=0$ actually implies the general case $q\geq 0$. 
We will show in Section \ref{sec:TP_properties} that Proposition \ref{prop:TP_properties}(d) implies the following generalization of \eqref{eq:Brylawski} and \eqref{eq:Gordon} to polymatroids.
\begin{proposition}[Brylawski identities]\label{prop:Brylawski}
Let $P\subset\Z^n$ be a polymatroid of level $d$, and let $\ds T_P(x,y):=\frac{(x+y-xy)^n}{x^{n-d}y^d}\,\T_{P}\left(\frac x{x+y-xy},\frac y{x+y-xy}\right)$. 
Then $T_P(x,y)$ is a Laurent polynomial in $x,y$. Moreover, denoting $\ell$ the minimal $y$-degree of $T_P(x,y)$ one has for all  $q\geq -\ell$:
$[x^p y^q]B(x,y)\,T_P(x,-1/y)=0$  for all $p< n+q$, and $[x^{n+q}y^q]B(x,y)\,T_P(x,-1/y)=(-1)^{n-d}$.
\end{proposition}

\section{Universal Tutte polynomial}
\label{sec:universal-Tutte}

In this section we define, for any positive integer $n$, the universal Tutte polynomial $\T_n$. This is 
a polynomial (of degree $n$) in the variables $x,y$ and $\{z_I\}$, 
where the index $I$ runs over all the nonempty subsets of $[n]$. 
The polynomial $\T_n$ contains the polymatroid Tutte polynomials of every polymatroid $P\subset\Z^n$, in the sense that $\T_P(x,y)$ is obtained from $\T_n$ by specializing the variables $z_I$.
More precisely, we have the following.


\begin{theorem} Let $n$ be a positive integer.
\label{th:universal_Tutte}
\begin{compactenum}
\item
There exists a unique 
polynomial $\T_n(x,y,(z_I))$ in $x,y$ and the $2^n-1$ variables $z_I$ indexed by the nonempty subsets $I$ of $[n]$, such that,
for any submodular function $f\colon 2^{[n]}\to\Z$, the specialization
of $\T_n$ at $z_I=f(I)$, for all $I\subseteq[n]$, is 
the polymatroid Tutte polynomial of the polymatroid $P_f$:
$$
\T_{P_f}(x,y)=\T_n(x,y,(z_I))\big|_{z_I=f(I),~\forall I\subseteq [n]}.
$$
\item 
The polynomial $\T_n$ has total degree $n$, and 
$(n-1)!\,\T_n$ has integer coefficients.
\end{compactenum}
\end{theorem}

\begin{definition}
The \emph{universal Tutte polynomial} $\T_n(x,y,(z_I))$ is the 
polynomial in the variables $x,y$, and $(z_I)_{\varnothing \neq I\subseteq [n]}$,
whose existence and uniqueness are provided by part (1) of Theorem~\ref{th:universal_Tutte}.
\end{definition}



An explicit expression of $\T_n$ will be given in Section~\ref{sec:formula_universal_Tutte} using results from~\cite{Pos}.

\begin{example}
In the case $n=1$, all polymatroids are singletons, and their Tutte polynomial is $x+y-1$.
Hence the universal Tutte polynomial is $\T_1=x+y-1$. (It does not depend on its third variable $z_{\{1\}}$.) 

To illustrate Theorem~\ref{th:universal_Tutte} and the challenges in its proof, we now examine the case $n=2$.
For conciseness, we will write the nonempty subsets of $[2]=\{1,2\}$ as $1$, $2$, and $12$.

\begin{figure}[h]
\centering
\begin{tikzpicture}[scale=.3]

\draw (-2,10) -- (17,10);
\draw (15,-2) -- (15,11);
\draw (19,-1) -- (7,11);

\node [below] at (1,10) {\small $a_2=z_2$};
\node [left] at (15,-1) {\small $a_1=z_1$};
\node [above right] at (18,0) {\small $a_1+a_2=z_{12}$};

\draw [gray,fill=gray] (8,10) circle [radius=.3];
\draw [gray,fill=gray] (9,9) circle [radius=.3];
\draw [ultra thick,fill] (11,7) circle [radius=.1];
\draw [ultra thick,fill] (11.5,6.5) circle [radius=.1];
\draw [ultra thick,fill] (12,6) circle [radius=.1];
\draw [gray,fill=gray] (14,4) circle [radius=.3];
\draw [gray,fill=gray] (15,3) circle [radius=.3];

\end{tikzpicture}
\caption{An arbitrary polymatroid in $\Z^2$.}
\label{fig:n=2}
\end{figure}
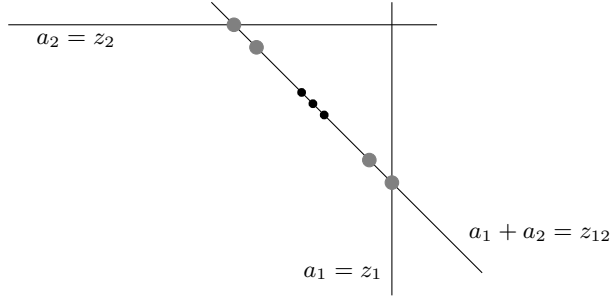

The three constraints for the points $\a=(a_1,a_2)$ of a polymatroid $P\subset \Z^2$ are indicated in Figure~\ref{fig:n=2}. 
(Submodularity in this case means $z_1+z_2\ge z_{12}$.) As usual the index $1$ is internally and externally active for every point $\a\in P$. 
The index 2 is internally active only for the bottom point and it is externally active only for the top point. 
Considering the case $z_1+z_2>z_{12}$ leads to the following tentative formula for the universal Tutte polynomial $\T_2$:
\begin{equation}
\label{eq:T2}
\begin{aligned}
\T_2&=(x+y-1)\,x+(z_1+z_2-z_{12}-1)(x+y-1)+(x+y-1)\,y\\
&=(x+y-1)(x+y-1+z_1+z_2-z_{12}).
\end{aligned}
\end{equation}
Now, when $z_1+z_2=z_{12}$, the polymatroid $P$ degenerates to a point and both indices are internally and externally active.
The important observation is that the expression~\eqref{eq:T2} for $\T_2$ specializes to the correct value $\T_P(x,y)=(x+y-1)^2$ for this case, despite the fact that the logic we used to write the formula does not apply.
\end{example}

\begin{example}
\label{ex:T3}
The third universal Tutte polynomial is
\begin{multline*}
\T_3=(x+y-1)(x^2+2xy+y^2\\
+(z_1+z_2+z_3-z_{123}-2)\,x+(z_{12}+z_{13}+z_{23}-2z_{123}-2)\,y\\
+\frac12(z_{123}^2-z_{12}^2-z_{13}^2-z_{23}^2-z_1^2-z_2^2-z_3^2-2z_{123}(z_1+z_2+z_3)\\
+2(z_1z_{12}+z_1z_{13}+z_2z_{12}+z_2z_{23}+z_3z_{13}+z_3z_{23})+3z_{123}-z_{12}-z_{13}-z_{23}-z_1-z_2-z_3+2)).
\end{multline*}
At $z_1=2$, $z_2=3$, $z_3=z_{12}=z_{13}=z_{23}=z_{123}=4$, it does specialize to the polymatroid Tutte polynomial that we computed in Example~\ref{ex:polynomial}.
\end{example}



In the proof of Theorem~\ref{th:universal_Tutte}, and in the next section, it will be fruitful to consider an alternative to the set of variables $(z_I)_{\varnothing \neq I\subseteq [n]}$. 
Namely, we consider the indeterminates $(t_I)_{\varnothing \neq I\subseteq [n]}$, and the following change of variables:
\begin{equation}
\label{eq:from_t_to_z}
z_I=\sum_{\varnothing\neq J\subseteq [n]\text{ with }J\cap I\ne\varnothing}t_J.
\end{equation}
By applying the inclusion-exclusion formula between $t_I$ and $y_I=M-z_{[n]\setminus I}=\sum_{J\subseteq I}t_J$, where $M=\sum_{\varnothing\neq J\subseteq [n]}t_J$, we obtain the inverse
\begin{equation}\label{eq:inclusion-exclusion}
t_I=\sum_{\varnothing\neq J\subseteq [n]\text{ with }J\cup I=[n]}(-1)^{|I\cap J|+1}z_J.
\end{equation}
In this paper, we use the variables $(z_I)$ to parametrize the rank function of a generalized permutohedron $\mP$ (and thus the positions of its facets). As we now explain, the variables $(t_I)$ should be thought of as parametrizing the coefficients in the decomposition of $\mP$ into the Minkowski sum of simplices.

\begin{remark}\label{rk:zI-to-tI-meaning}
For a non-empty subset $I\subseteq [n]$, we denote by $\Delta_I=\conv(\e_i\mid i\in I)$ the coordinate simplex associated to $I$. For a (non-zero) tuple $\m=(m_{I})_{\varnothing \neq I\subseteq [n]}$ of non-negative integers, we consider the Minkowski sum
\begin{equation*}
\mQ_\m:=\sum_{\varnothing \neq I\subseteq [n]}m_{I}\,\Delta_I.
\end{equation*}
Since each simplex $\Delta_I$ is a generalized permutohedron, Remark~\ref{rk:polymat-Minkowski} lets us conclude that $\mQ_\m$ is a generalized permutohedron, and that $Q_\m:=\mQ_\m\cap \Z^n$ is a polymatroid. By computing the rank functions of the simplices $\Delta_I$, we deduce that the rank function of $Q_\m$ is the function $f_\m\colon2^{[n]}\to \Z$ given by 
\begin{eqnarray}\label{eq:rank-Minkowski-simplex}
f_\m(I)=\sum_{\varnothing \neq J\subseteq[n]~:~J\cap I\ne\varnothing}m_J,
\end{eqnarray}
for all $I\subseteq [n]$. In other words, $Q_\m=P_{f_{\m}}$. Comparing~\eqref{eq:from_t_to_z} to~\eqref{eq:rank-Minkowski-simplex} justifies our interpretation of the variables $t_I$ as parametrizing the coefficients in the decomposition of generalized permutahedra into the Minkowski sum of simplices (at least when evaluating the variables $t_I$ at non-negative integers). 
\OB{I added the following reference to \cite{DK1,DK2}. Not sure if this is the best place nor if we should add more. I am not clear about the uniqueness of the expression of $P$ as a Minkowski sum and difference of simplex.}
We mention that it was proved in \cite{DK1,DK2} that any polymatroid can actually be written as a Minkowski sum and difference of simplices of the form  $\Delta_I$.
\end{remark}

Before embarking on the proof of Theorem~\ref{th:universal_Tutte} let us state a classical fact (see, e.g., \cite[Theorem 3.15]{fuji}), of which we include a proof for the reader's convenience.
\begin{lemma}\label{lem:codim=sum}
If a polymatroid $P\subset \Z^n$ has codimension $c$, then it is a direct sum of $c$ polymatroids. That is to say, there exists a permutation $w\in S_n$ and some polymatroids $P_1,\ldots,P_c$ such that $w(P)=P_1\oplus \cdots\oplus P_c$.
\end{lemma}

\begin{proof}
Assume $c>1$. Then there exists a subset $\varnothing \neq I\subsetneq [n]$ such that $P\subset \{\,\x\in \R^n\mid\sum_{i\in I}x_i=f(I)\,\}$, where $f$ is the rank function of $P$. 
Up to permutation of coordinates, we can assume that $I=[k]$ for some $0<k<n$. Let $Q=\{\,(a_1,\ldots,a_k)\mid\a\in P\,\}$ and $Q'=\{\,(a_{k+1},\ldots,a_n)\mid\a\in P\,\}$. We claim that $Q\oplus Q'=\{\,(a_1,\ldots,a_n)\mid(a_1,\ldots,a_k)\in Q,~(a_{k+1},\ldots,a_n)\in Q'\,\}$ is equal to $P$. Indeed the inclusion $P\subseteq Q\oplus Q'$ is obvious, and the exchange axiom easily implies that $Q\oplus Q'\subseteq P$. Thus, if $c>1$, then $P$ is the direct sum of two polymatroids. The claimed property follows by induction on~$c$. 
\end{proof}


\begin{proof}[Proof of Theorem~\ref{th:universal_Tutte}]
Let us start by proving the uniqueness of $\T_n$. Suppose that some polynomials $\T_n(x,y,(z_I))$ and $\T_n'(x,y,(z_I))$ both satisfy Property (1). Let $(t_{I})_{\varnothing \neq I\subseteq [n]}$ be a new set of variables, and let $\hT_n(x,y,(t_I))$ and $\hT_n'(x,y,(t_I))$ be the polynomials obtained from $\T_n$ and $\T_n'$, respectively, via 
the substitution~\eqref{eq:from_t_to_z}.
It suffices to show that $\hT_n=\hT_n'$.
For any tuple $\m=(m_{I})_{\varnothing \neq I\subseteq [n]}$ of non-negative integers, the function $f_\m\colon2^{[n]}\to \Z$ given by \eqref{eq:rank-Minkowski-simplex} is the rank function of the polymatroid $Q_\m$. By assumption, $\T_n(x,y,(f_\m(I)))=\T_{Q_\m}(x,y)=\T_n'(x,y,(f_\m(I)))$, hence $\hT_n(x,y,(m_I))=\hT_n'(x,y,(m_I))$ for all tuples $(m_{I})$ of non-negative integers.
This implies $\hT_n=\hT_n'$, and proves the uniqueness of $\T_n$.

We now embark on the proof of the existence of $\T_n$.
Let us consider the set of polymatroids $P\subset \Z^n$ such that the generalized permutohedron $\mP=\conv(P)$ has a given fixed ``combinatorial type'' (that is, a given normal fan). 
As we now explain, the number of lattice points in the relative interior $\mF^\circ$ of any face $\mF$ of $\mP$ is a polynomial in the variables $(z_I)$ (which record the values of the rank function of $\mP$). In fact, the face $\mF$ is a generalized permutahedron, and by Corollary~\ref{cor:Dragon-exists}, the number of lattice points in $\mF^\circ$ is a polynomial function of the rank function of $\mF$. So it is sufficient to show that the values of the rank function $g$ of $\mF$ are polynomial functions of 
the values of the rank function $f$ of $\mP$. We will actually prove that each value $g(I)$ is a linear function of $(f(J))_{J\subseteq [n]}$. 
By induction on the codimension of $\mF$, it suffices to show that this is true for each facet $\mF$ of $\mP$. So let us fix a non-empty subset $K\subsetneq [n]$ and consider the facet $\mF=\mP\cap H_K$, where $H_K$ is the hyperplane $\{\,\x\in\R^n\mid\sum_{i\in K}x_i=f(K)\,\}$. We contend that the rank function $g$ of $\mF$ is given in terms of the rank function $f$ of $\mP$ by 
$$g(I)=f(I\cup K)+f(I\cap K)-f(K).$$
Indeed it is not hard to prove that this function $g$ is submodular, and that $P_g=P_f\cap H_K$ (the inclusion $P_g\subseteq P_f$ is by submodularity of $f$, the inclusion $P_g\subseteq H_K$ is because $g(K)=f(K)$ and $g([n]\setminus K)=f([n])-f(K)$, and the inclusion $P_f\cap H_K\subseteq P_g$ is because any $\a\in P_f\cap H_K$ satisfies $\sum_{i\in I}a_i=\sum_{i\in I\cup K}a_i+\sum_{i\in I\cap K}a_i-\sum_{i\in K}a_i\leq g(I)$).
Hence we have indeed proved that the number of lattice points in $\mF^\circ$ is a polynomial in the $(z_I)$. 

Moreover, for a basis $\a\in P$ the sets $\Int_P(\a)$ and $\Ext_P(\a)$ of internally and externally active indices
depend only on the unique face of the polytope $\mP$ whose relative interior contains the point $\a$. 
This shows that, for the set of polymatroids of a given combinatorial type, the polymatroid Tutte polynomial $\T_P(x,y)$ can be expressed as a certain polynomial in $x$, $y$ and the $(z_I)$ (by writing it as a sum over the faces of $\mP$).

As before, we say that a polymatroid $P$ has \emph{generic type} if the normal fan of $\mP$ is the braid arrangement (i.e., $\mP$ is combinatorially equivalent to the classical permutohedron). Let $\T_n$ be the polynomial corresponding to the Tutte polynomials of polymatroids of generic type.
 We need to show that $\T_n$ specializes to the correct polynomials
for all \emph{degenerate} (i.e., non-generic) polymatroids as well.
Since any degenerate polymatroid can be transformed into a generic one 
by repeatedly taking Minkowski sums with some line
segments $[\e_i,\e_j]$, it is enough to prove the following claim.

Let $P$ be any polymatroid, and let $\mP=\conv(P)$. For $t\in\Z_{\geq 0}$, let $\mP(t):=\mP+t[\e_i,\e_j]$ and $P(t)=\mP(t)\cap\Z^n$. 
The above argument shows that the polymatroid Tutte polynomial $\T_{P(t)}(x,y)$ is a polynomial function of $t$ (and $x,y$) for $t\in\Z_{>0}$.

\smallskip

\noindent \textbf{Claim.} The specialization of this polynomial at $t=0$ is exactly $\T_P(x,y)$.\\ 

\smallskip

We now prove 
this claim. In the limit $t\to 0$, some faces $\mF(t)$ of $\mP(t)$ collapse to a face $\mF$ of $P$ of dimension one less. 
We want nonetheless to establish a correspondence between the faces of $\mP(t)$ and the faces of $\mP$. For this, recall that a face of a polytope $\mQ\subset\R^n$ is the set of points $\mF$ in $\mQ$ maximizing a linear functional $\al\in (\R^n)^*$. From this, it follows that the faces of the Minkowski sum of two polytopes are the Minkowski sums of the appropriate faces of these polytopes. 
Hence the faces of the Minkowski sum $\mP(t)=\mP+t[e_i,e_j]$ can be partitioned according to the face $\mF$ of $\mP$ that they correspond to. There are three possible situations, which are illustrated in Figure~\ref{fig:type-faces-universality}:
\begin{compactitem}
\item[(a)] The face $\mF$ is not the maximizer of any functional $\al$ such that $\al(\e_i)=\al(\e_j)$. In this case, we associate to $\mF$ the face $\mF(t)$ of $\mP(t)$ which is the translation of $\mF$ by either $t\e_i$ (if $\al(\e_i)>\al(\e_j)$) or $t\e_j$ (if $\al(\e_i)<\al(\e_j)$).
\item[(b)] The face $\mF$ is the maximizer of a functional $\al$ such that $\al(\e_i)=\al(\e_j)$, and moreover the affine span of $\mF$ contains (lines parallel to) the vector $\e_i-\e_j$. In this case, we associate to $\mF$ the face $\mF(t)$ of $\mP(t)$ which is the Minkowski sum $\mF+t[\e_i,\e_j]$ (which has dimension $\dim(\mF)$).
\item[(c)] The face $\mF$ is the maximizer of a functional $\al$ such that $\al(\e_i)=\al(\e_j)$, but the affine span of $\mF$ does not contain the vector $\e_i-\e_j$. In this case, we associate to $\mF$ three faces of $\mP(t)$, namely $\mF_1(t)=\mF+t\e_i$, $\mF_2(t)=\mF+t\e_j$, and $\mF_{1,2}(t)=\mF+t[\e_i,\e_j]$ (the last of which has dimension $\dim(\mF)+1$).
\end{compactitem}
We now show that in all three cases the contribution, to the Tutte polynomial, of the bases in the interior of the face (or faces) of $\mP(t)$ is a polynomial in $t$ which specializes correctly at $t=0$.

\begin{figure}[!ht]
\begin{center}\includegraphics[width=\linewidth]{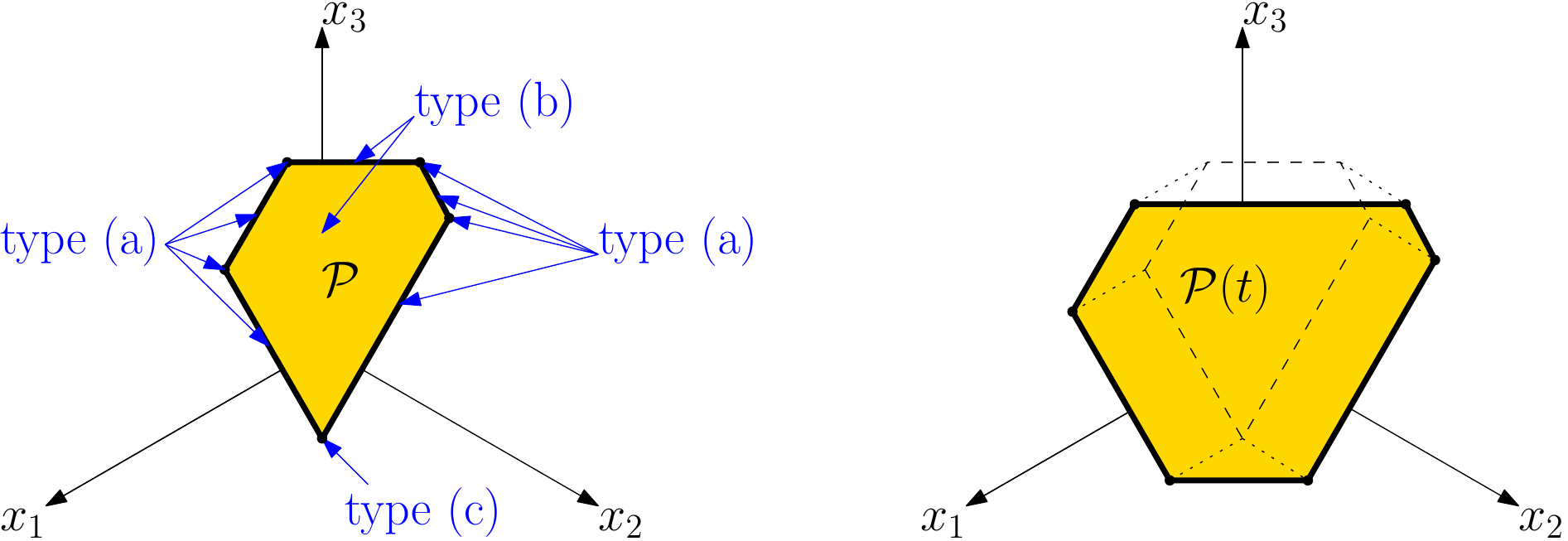}\end{center}
\caption{The generalized permutahedra $\mP$ (left) and $\mP(t)=\mP+t[\e_1,\e_2]$ (right). For each face of $\mP$ we indicate if this face corresponds to the situation described in case (a), (b), or (c).}
\label{fig:type-faces-universality}
\end{figure}

In the cases (a) and (b) the number of lattice points in the interior of $\mF(t)$ is a polynomial in $t$ (valid also at $t=0$), and any interior lattice points $\a$ of $\mF(t)$ have the same activity sets $\Int_{P(t)}(\a)$ and $\Ext_{P(t)}(\a)$ as the interior lattice points of $\mF$. Hence the contribution of those interior lattice points is indeed a polynomial function of $t$ valid at $t=0$.
It remains to examine the case (c). Let $\mF$ be a face of $\mP$ of type (c). Let $F^\circ$, $F^\circ_1(t)$, $F^\circ_2(t)$, and $F^\circ_{1,2}(t)$ be the sets of interior lattice points of $\mF$, $\mF_1(t)$, $\mF_2(t)$, and $\mF_{1,2}(t)$, respectively. Then for all $t\geq 0$ we have the partition 
$$F^\circ_1(t)\cup F^\circ_2(t) \cup F^\circ_{1,2}(t)=\biguplus_{\a\in F^\circ} L_\a(t),$$ 
where $L_\a(t)=(\a+t[\e_i,\e_j])\cap \Z^n=\{\a+t\e_i,\a+(t-1)\e_i+\e_j,\ldots,\a+t\e_j\}$.
Let us now fix $\a\in F^\circ$, and see how the contribution of the bases of $P(t)$ along $L_\a(t)$ depends on $t$.
For a polymatroid $Q\subset \Z^n$ and $\c\in Q$, let us put 
\[A_Q(\c)=\{\,(k,\ell)\in [n]^2\mid\a-\e_k+\e_\ell\in Q\,\}.\] 
Note that for all $k\in[n]$, the pair $(k,k)$ is in $A_Q(\c)$. Note also that 
$\Int_Q(\c)=\{\,k\in[n]\mid\forall \ell<k,~(k,\ell)\notin A_Q(\c)\,\}$, and $\Ext_Q(\c)=\{\,k\in[n]\mid\forall \ell<k,~(\ell,k)\notin A_Q(\c)\,\}$.
We now use three facts.

\medskip

\noindent Fact 1: If $(k,\ell)$ and $(\ell,m)$ belong to $A_Q(\c)$, then so does $(k,m)$. 

\medskip

In order to prove Fact 1, let us consider the rank function $g$ of $Q$. Assuming $\c-\e_k+\e_\ell\in Q$ and $\c-\e_\ell+\e_m\in Q$, it follows that $\c-\frac{\e_k}{2}+\frac{\e_m}{2}\in \conv(Q)$. This, in turn, implies that for all subsets $I\subset [n]$ containing $m$ but not $k$, we have $\sum_{i\in I}c_i< g(I)$. Hence $\b:=\c-\e_k+\e_m$ satisfies $\sum_{i\in I}b_i\leq g(I)$, which shows that $\b$ is in $Q$, as wanted.

\medskip

Let $A=A_P(\a)$. For $t>0$ let $B_1=A_{P(t)}(\a+t\e_i)$, and $B_2=A_{P(t)}(\a+t\e_j)$. We observe that for $s\in[t-1]$, the set $A_{P(t)}(\a+s\e_i+(t-s)\e_j)$ does not depend on $s$ (these points all belong to the interior of $\mF(t)$), and we denote it by $B_{1,2}$. 
Let $C_1=\{\,(k,\ell)\mid(k,i),(j,\ell)\in A\,\}$ and $C_2= \{\,(k,\ell)\mid(k,j),(i,\ell)\in A\,\}$.

\medskip

\noindent Fact 2: We have $B_1=A\cup C_1$, $B_2=A\cup C_2$, and $B_{1,2}=A\cup C_1\cup C_2$.

\medskip

Let us show $B_1=A\cup C_1$. Using Fact 1, we get $A\cup C_1\subseteq B_1$ (since $(i,j)\in B_1$). We now prove the other inclusion. Let $(k,\ell)\in B_1$. In this case, the vector $-\e_k+\e_\ell$ belongs to the cone generated by $\{\,-\e_{k'}+\e_{\ell'}\mid(k',\ell')\in A_1\,\}$, where $A_1=A\cup \{(i,j)\}$. That is to say, $-\e_k+\e_\ell=\sum_{(k',\ell')\in A_1}c_{k',\ell'}\,(-\e_{k'}+\e_{\ell'})$, with all the coefficients $c_{k',\ell'}$ non-negative. We now fix such an expression of $-\e_k+\e_\ell$ so that it minimizes the number of non-zero coefficients. Consider the directed graph $D$ with vertex set $[n]$ and arc set $\{\,(k',\ell')\mid c_{k',\ell'}>0\,\}$. The minimality condition ensures that $D$ has no directed cycles. It is easy to see that there exists a directed path from $k$ to $\ell$ in $D$. If the directed path does not use the arc $(i,j)$, then Fact 1 implies that $(k,\ell)\in A$. If the directed path uses the arc $(i,j)$, then Fact 1 implies that both $(k,i)$ and $(j,\ell)$ belong to $A$, whence $(k,\ell)\in C_1$. Hence $(k,\ell)\in A\cup C_1$, which completes the proof of $B_1=A\cup C_1$. The proof of $B_2=A\cup C_2$ (resp., $B_{1,2}=A\cup C_1\cup C_2$) is similar with $A_1$ replaced with $A_2=A\cup\{(j,i)\}$ (resp., $A_{1,2}=A\cup\{(i,j),(j,i)\}$). This proves Fact 2.

\medskip

Let $i'=\min\{\,\ell\in[n]\mid(i,\ell)\in A\,\}$, $i''=\min\{\,\ell\in[n]\mid(\ell,i)\in A\,\}$, $j'=\min\{\,\ell\in[n]\mid(j,\ell)\in A\,\}$, and $j''=\min\{\,\ell\in[n]\mid(\ell,j)\in A\,\}$.
Let also $K_1'=\{\,k\in[n]\mid i'<k\leq j'\text{ and } (k,j)\in A\,\}$ and $K_2'=\{\,k\in[n]\mid j'<k\leq i'\text{ and } (k,i)\in A\,\}$, as well as $K_1''=\{\,k\in[n]\mid i''<k\leq j''\text{ and } (j,k)\in A\,\}$ and $K_2''=\{\,k\in[n]\mid j''<k\leq i''\text{ and } (i,k)\in A\,\}$.

\medskip

\noindent Fact 3: 
\begin{compactitem} 
\item $\Int(\a+t\e_i)=\Int(\a)\setminus K_2'$ and $\Ext(\a+t\e_i)=\Ext(\a)\setminus K_1''$;
\item $\Int(\a+t\e_j)=\Int(\a)\setminus K_1'$ and $\Ext(\a+t\e_i)=\Ext(\a)\setminus K_2''$;
\item For $\b\in L_\a(t)\setminus \{\a+t\e_i,\a+t\e_j\}$, we have $\Int(\b)=\Int(\a)\setminus (K_1'\cup K_2')$ and $\Ext(\b)=\Ext(\a)\setminus (K_1''\cup K_2'')$.
\end{compactitem}

\medskip

Fact 3 follows directly from Fact 2 and we leave the details to the reader. By way of illustration, let us observe that if $(k,\ell)$ is in $A$ then $(k,i')\in A$ and $(k,j')\in C_1$ (in fact $j'=\min\{\,\ell\mid(k,\ell)\in C_1\,\}$).

\medskip

We can now finish the proof of the Claim. We assume, without loss of generality, that $i'\leq j'$. Note that this implies $K_2'=\varnothing$.
We now consider two cases and show that the aggregate contribution $w(t)$ of $L_\a(t)$ to $\T_{P(t)}$, seen as a polynomial in $t$, specializes correctly at $t=0$ to the contribution $w$ of $\a$ to $\T_P$. 
\begin{compactitem} 
\item[Case 1: $i''\geq j''$.] In this case, $K_2'=\varnothing$ and $K_1''=\varnothing$, hence the contribution of $\a+t\e_i$ to $\T_{P(t)}$ is $w$. Moreover, the contributions of the other elements of $L_\a(t)$ are all equal; let us denote this common value by $w'$. This gives $w(t)=w+t w'$, and $w(0)=w$, as wanted.
\item[Case 2: $i''< j''$.] In this case $K_2'=\varnothing$ and $K_2''=\varnothing$. 
First observe that if $K_1''\cap \Ext_P(\a)=\varnothing$, then we can reason as in Case 1. 
Similarly, if $K_1'\cap \Int_P(\a)=\varnothing$, then the contribution of $\a+t\e_j$ to $\T_{P(t)}$ is $w$, while all the other elements of $L_\a(t)$ contribute the same value $w'$. Hence in this case we get $w(t)=w+t w'$ and $w(0)=w$, as wanted. It remains to consider the case when both $K_1'\cap \Int_P(\a)$ and $K_1''\cap \Ext_P(\a)$ are non-empty. Let $k'\in K_1'\cap \Int_P(\a)$ and let $k''\in K_1''\cap \Ext_P(\a)$. Since $k'\in K_1'$ we have $k'\leq j'$, and moreover $(k',j)\in A$ which implies $j''\leq k'$. Similarly, $k''\in K_1''$ implies $j'\leq k''\leq j''$. Together this gives $j'=j''=k'=k''$. Hence $j'\in\Int(\a)\cap\Ext(\a)$ and $j'$ is the only index which has different activity status for $\a$ and for $\b\in L_\a(t)$. Precisely, the contribution of $\a+t\e_i$ to $\T_{P(t)}$ is $\frac{x}{x+y-1}w$, the contribution of $\a+t\e_j$ is $\frac{y}{x+y-1}w$, and the contribution of any other element of $L_\a(t)$ is $\frac{1}{x+y-1}w$. Thus $w(t)=\frac{x+y+(t-1)}{x+y-1}w$, and $w(0)=w$, as wanted.
\end{compactitem}
This completes the proof of the claim, and hence that of the existence of $\T_n$. \\

It only remains to prove that $(n-1)!\,\T_n$ has integer coefficients and degree~$n$. 
The integrality of the coefficients follows from Corollary~\ref{cor:Dragon-exists}, 
and the linear relation explained above between the rank function of a polymatroid and that of its faces.
Lastly we show that the degree of $\T_n$ is $n$.
Recall that each $d$-dimensional face $\mF$ of a generic permutohedron $\mP$ corresponds to an ordered partition $B$ of $[n]$ into $n-d$ blocks. The correspondence is such that the vector $\e_i-\e_j$ is contained in the affine span of $\mF$ if and only if the indices $i,j$ are in the same block of $B$. Consequently, only minimal elements in the blocks can potentially be internally or externally active with respect to a basis $\a$ of $P$ in the interior of $\mF$ (since $\a+\e_i-\e_j\in \mF$ for all $i,j$ in the same block of $B$). Hence the contribution of each basis of $P$ in the interior of $\mF$ is a polynomial in $x,y$ of degree at most $d$.

Next we show that the number of bases in the interior of $\mF$ is given by a polynomial in the $(z_I)$ of degree at most~$d$. By Lemma~\ref{lem:codim=sum}, the polymatroid $F=\mF\cap \Z^n$ is a direct sum of $n-d$ polymatroids: $w(F)=P_1\oplus \ldots\oplus P_{n-d}$ for some permutation $w$ and polymatroids $P_1,\ldots, P_{n-d}$. The number of bases in the interior of $\mF$ is $\prod_{i=1}^{n-d}|P_i^\circ|$, where $P_i^\circ$ is the set of lattice points in the relative interior of $\mP_i=\conv(P_i)$.
We have shown above that the values of the rank function $g$ of $w(F)$ are linear functions of the variables $(z_I)$ (which record the values of the rank function of $P$). Hence for each $i\in[n-d]$ the values of the rank function $f_i$ of $P_i$ are linear functions of the $(z_I)$. Finally, by Corollary~\ref{cor:Dragon-exists}, $|P_i^\circ|$ is a polynomial, in the values of $f_i$, of degree $\dim(\mP_i)$. Hence the number of bases in the interior of $\mF$ is indeed a polynomial in the $(z_I)$ of degree $\prod_{i=1}^{n-d}\dim(\mP_i)=\dim(\mP)=d$.
So the contribution of the interior points of the face $\mF$ to the polynomial $\T_n$ is a polynomial of degree at most 
$(n-d)+d=n$. The top-dimensional face of the permutohedron contributes a 
polynomial of degree $n$ to $T_n$. So $\T_n$ has degree~$n$.

This completes the proof of Theorem~\ref{th:universal_Tutte}.
\end{proof}

Let us reformulate Lemmas~\ref{lem:shifts} and~\ref{lem:duality}, as well as
Theorem~\ref{th:order_independence},
in terms of the universal Tutte polynomial.
It is instructive to check these claims against Example~\ref{ex:T3}.

\begin{corollary}
The universal Tutte polynomial $\T_n(x,y,(z_I))$ has the following symmetries:
\begin{enumerate}
\item {\bf affine translations:}
$\T_n(x,y,(z_I))= \T_n(x,y,(z_I+\sum_{i\in I} t_i))$, for formal variables $t_1,\dots,t_n$.
\item {\bf duality:} $\T_n(x,y,(z_I)) = \T_n(y,x,(z_{[n]\setminus I}-z_{[n]}))$, where we use the convention $z_\varnothing = 0$.
\item {\bf $S_n$-symmetry:}
$\T_n(x,y,(z_I))= \T_n(x,y,(z_{w(I)}))$, for any $w\in S_n$.
\end{enumerate}
\end{corollary}

Let us mention a couple of additional properties of the universal Tutte polynomial
$\T_n$ that easily follow from its definition.

\begin{proposition}\label{prop:easy-prop-Tn}
{\rm (a)} The polynomial $\T_n$ is divisible by $(x+y-1)$.\\
\noindent
{\rm (b)} The specialization of $\T_n(x,y,(z_I))$ obtained by setting all the variables $z_I$ equal to $0$, is $(x+y-1)^n$.
\end{proposition}

\begin{proof}
(a) We derive this property from the fact (cf.\ part (c) of Proposition~\ref{prop:TP_properties}) that any polymatroid Tutte polynomial is divisible by $(x+y-1)$. 
Consider the polynomial $\hT_n(1-y,y,(t_I))$ defined as in the proof of Theorem~\ref{th:universal_Tutte}. For any tuple of positive integers $(t_I)$ the evaluation $\hT_n(1-y,y,(t_I))$ is equal to the evaluation of a polymatroid Tutte polynomial at $(x,y)=(1-y,y)$, which is 0. Hence $\hT_n(1-y,y,(t_I))=0$. Thus $T_n(1-y,y,(z_I))=0$, which shows that $T_n(x,y,(z_I))$ is divisible by $(x+y-1)$.

(b) Let $P$ be the polymatroid with unique basis $(0,0,\ldots,0)\in\Z^n$. It has an identically $0$ rank function and satisfies $\T_{P}(x,y)=(x+y-1)^n$. Hence $\T_n(x,y,(0))=T_P(x,y)=(x+y-1)^n$.
\end{proof}



In the next section we will give an explicit formula for the universal Tutte polynomial.
In Sections~\ref{sec:cone-decomposition} and~\ref{sec:order-invariance}, we will prove the $S_n$-invariance of the 
polymatroid Tutte polynomial by giving another construction 
for it which is manifestly $S_n$-invariant.


\section{Formula for the universal Tutte polynomial} \label{sec:formula_universal_Tutte}
In this section, we give an explicit formula for the universal Tutte polynomial.
We first recall a definition from~\cite{Pos} which will play a role in the formula.

\begin{definition} \label{def:dragon-poly}
A function $g\colon2^{[n]}\setminus \{\varnothing\}\to \Z^{\geq 0}$ is \emph{$n$-draconian} if $\sum_{I\subseteq [n]}g(I)=n-1$ and for any $k>0$ and any collection $I_1,\ldots,I_k$ of distinct non-empty subsets of $[n]$, we have
$$\sum_{j=1}^k g(I_j)\leq |\bigcup_{j=1}^k I_j|-1.$$
The $n$'th \emph{dragon polynomial} $D_n$ is the polynomial in the $2^{n}-1$ variables $(t_I)_{\varnothing \neq I\subseteq [n]}$ defined by
$$D_n((t_I))=\sum_{g} {t_{[n]}-1\choose g([n])}\prod_{\varnothing \neq I\subsetneq [n]} {t_{I}\choose g(I)},$$
where the sum is over all $n$-draconian functions, and ${t \choose k}$ is the polynomial $\frac{1}{k!}t(t-1)\cdots(t-k+1)$ (with the usual convention ${t \choose 0}:=1$).
\end{definition}


\begin{example} For $n=1,2,3$ we get
$D_1=1$, $D_2=t_{\{1,2\}}-1$, and 
$$D_3=\frac{1}{2}\,t_{\{1,2,3\}}^2+\left(e_1-\frac{3}{2}\right)\,t_{\{1,2,3\}}+e_2-e_1+1,$$
where $e_1=t_{\{1,2\}}+t_{\{1,3\}}+t_{\{2,3\}}$ and $e_2=t_{\{1,2\}}t_{\{1,3\}}+t_{\{1,2\}}t_{\{2,3\}}+t_{\{1,3\}}t_{\{2,3\}}$.
\end{example}

The dragon polynomial is related to the generalized Ehrhart polynomial of permutohedra $E_n((z_I))$ defined in Section~\ref{sec:gen-Ehrhart-poly}. Namely, 
$D_n((t_I))$ is obtained from $E_n((z_I))$ by: 
\begin{equation}\label{eq:Dn-to-En}
D_n((t_I))=(-1)^{n-1}E_n(-(z_I))\big|_{z_I=\sum_{J\subseteq[n],J\cap I\neq \varnothing}t_J}.
\end{equation}
Hence, according to Corollary~\ref{cor:Dragon-exists}, the dragon polynomial counts interior points of polymatroids. See Lemma~\ref{lem:lattice-count} below for a precise statement. As we will explain in Remark~\ref{rmk:complete_dragon}, the $n$-draconian functions are in bijection with the spanning hypertrees of the hypergraph whose vertex set is $[n]$ and in which each non-empty subset of $[n]$ is a hyperedge.
The numbers of $n$-draconian functions for $n=1,2,3,4,5$ are $1$, $1$, $7$, $188$, and $16626$, respectively.

Let $\mB_n$ be the set of all \emph{ordered set partitions} of $[n]$, that is, tuples $(B_1,\ldots,B_\ell)$ of non-empty disjoint subsets of $[n]$ such that $\biguplus_i B_i=[n]$. A \emph{left-to-right minimum} (resp., \emph{right-to-left minimum}) for $B=(B_1,\ldots,B_\ell)\in\mB_n$ is an integer $m\in[n]$ such that $m=\min\left(\bigcup_{i\leq k} B_i\right)$ (resp., $m=\min\left(\bigcup_{i\geq k} B_i\right)$) for some $k\in [\ell]$. 

\begin{theorem}\label{thm:formulaTn}
Let $n$ be a positive integer. Let $\T_n(x,y,(z_I)_{\varnothing\neq I\subseteq [n]})$ be the universal Tutte polynomial and let $\hT_n(x,y,(t_I)_{\varnothing\neq I\subseteq [n]})$ be the polynomial obtained from $\T_n$ by the substitution 
\eqref{eq:from_t_to_z} of Section~\ref{sec:universal-Tutte}.
Then, 
\begin{equation}\label{eq:formulaTn}
\hT_n(x,y,(t_I))=(x+y-1)\sum_{B\in\mB_n}(-1)^{\ell(B)-1}\,D_n((t^B_I))\,x^{\lr(B)-1}\,y^{\rl(B)-1},
\end{equation}
where 
\begin{compactitem}
\item $\mB_n$ is the set of ordered set partitions of $[n]$ and $D_n$ is the $n$'th dragon polynomial, 
\item 
$\ell(B)$ is the number of blocks of $B\in\mB_n$, and $\lr(B)$ and $\rl(B)$ are the numbers of left-to-right and right-to-left minima of~$B$, respectively; finally,
\item for any $B=(B_1,\ldots,B_\ell)\in \mB_n$ and any non-empty set $I\subseteq [n]$, 
$$t^B_I=\sum_{J\subseteq \bigcup_{i<k}B_i}t_{I\cup J}$$
if $I\subseteq B_k$ for some $k\in[\ell]$, and $t^B_I=0$ if $I$ is contained in none of the subsets $B_k$ for $k\in[\ell]$. 
\end{compactitem}
The polynomial $\T_n(x,y,(z_I))$ is obtained from $\hT_n(x,y,(t_I))$ by the substitution~\eqref{eq:inclusion-exclusion}.
\end{theorem}

\OB{Open question: Could the following remark be extended to allow some negative coefficients? }

\begin{remark}
\label{cor:formulaTn}
Recall from Remark~\ref{rk:zI-to-tI-meaning} the combinatorial interpretation of the variables $(t_I)$. It follows from Theorem~\ref{thm:formulaTn} that, for any (non-zero) tuple $\m=(m_I)_{\varnothing\neq I\subset [n]}$ of non-negative integers, the polymatroid 
\begin{equation}\label{eq:Qm}
Q_\m=\Z^n\cap \sum_{\varnothing\neq I\subset [n]}m_I\Delta_I 
\end{equation}
has Tutte polynomial
\begin{equation}\label{eq:TQm}
\T_{Q_{\m}}(x,y)=(x+y-1)\sum_{B\in\mB_n}(-1)^{\ell(B)-1}\,D_n((m^B_I))\,x^{\lr(B)-1}\,y^{\rl(B)-1},
\end{equation}
where, for all nonempty subsets $I\subseteq [n]$,
\begin{equation}\label{eq:mB}
m^B_I=\sum_{J\subseteq \bigcup_{i<k}B_i}m_{I\cup J}
\end{equation}
if $I\subseteq B_k$ for some $k\in[\ell]$, and $m^B_I=0$ if $I$ is contained in none of the subsets $B_k$ for $k\in[\ell]$. 
\end{remark}

\begin{example} For $n=1,2,3$ we get $\frac{\hT_1}{x+y-1}=1$, $\frac{\hT_2}{(x+y-1)}=(x+y+t_{\{1,2\}}-1)$, and 
\begin{eqnarray*}
\frac{\hT_3}{(x+y-1)}
&=&
{x}^{2}+2\,xy+{y}^{2} + \left(2\,t_{\{1,2,3\}}+e_1 -2\right) x+ \left(t_{\{1,2,3\}} +e_1 -2\right)y\\
&&+\frac{1}{2}{t_{{123}}}^{2}+ \left(e_1
 -\frac{3}{2}\right) t_{\{1,2,3\}}
+e_2-e_1+1,
\end{eqnarray*}
where $e_1=t_{\{1,2\}}+t_{\{1,3\}}+t_{\{2,3\}}$, and $e_2=t_{\{1,2\}}t_{\{1,3\}}+t_{\{1,2\}}t_{\{2,3\}}+t_{\{1,3\}}t_{\{2,3\}}$.
\end{example}

\begin{remark}
Note that the polynomial $\hT_n(x,y,(t_I))$ only depends on $2^n-n+1$ variables: the variables $x,y$ and the variables $t_I$ for $I\subseteq [n]$ of cardinality at least~$2$. Indeed, the dragon polynomial $D_n$ does not depend on $t_I$ for $|I|\leq 1$. The fact that $\hT_n$ is independent of $\{\,t_{\{i\}}\mid i\in[n]\,\}$ simply reflects the invariance of the Tutte polynomial with respect to translations of polymatroids. 
\end{remark}

The rest of this section is dedicated to the proof of Theorem~\ref{thm:formulaTn}. 
We want to prove that the polynomial $\hT_n(x,y,(t_I))$ is given by the formula~\eqref{eq:formulaTn}.
As can be seen from the proof of Theorem~\ref{th:universal_Tutte} (uniqueness of $\T_n$), it suffices to show that for all tuples $\m=(m_{I})_{\varnothing \neq I\subseteq [n]}$ of positive integers, the polymatroid $Q_\m$ defined by~\eqref{eq:Qm} satisfies~\eqref{eq:TQm}.

We now fix a tuple $\m=(m_{I})_{\varnothing \neq I\subseteq [n]}$ of positive integers and proceed to prove~\eqref{eq:TQm}.
We consider the partition of 
$Q_\m$ according to the faces of $\mQ_\m$: 
$$Q_\m=\biguplus_{\mF\,:\,\text{face of }\mQ_\m}F^\circ,$$
where $F^\circ:=\mF^\circ \cap \Z^n$ and $\mF^\circ$ denotes the relative interior of the face $\mF$.
It remains to count the points in $F^\circ$, and to determine their activities.



First, we observe that the polytope $\mQ_\m$ is a generalized permutohedron with the same combinatorial type as the classical permutohedron $\Pi_n$. In other words, the normal fan of $\mQ_\m$ is the braid arrangement, and its faces are indexed by the ordered set partitions of $[n]$. More precisely, for $B=(B_1,\ldots,B_\ell)\in \mB_n$, we denote by $\mF_B$ the face of $\mQ_\m$ maximizing the linear form $\varphi_B\in (\R^n)^*$ defined by setting $\varphi_B(\e_i)=k$ for all $k\in[\ell]$ and all $i\in B_k$. Now, we again use the well-known fact that the faces of a Minkowski sum of polytopes correspond to appropriate Minkowski sums of the faces of these polytopes. In our case we get 
$$\mF_B=\sum_{\varnothing\neq I\subseteq [n]}m_I \, \Delta_{I\cap B_{k(I)}},$$
where $k(I)=\max\{\,k\in[\ell]\mid I\cap B_k\neq\varnothing\,\}$. Reordering the terms gives 
$$\mF_B=\sum_{\varnothing\neq I\subseteq [n]}m^B_I\Delta_{I}=\mQ_{\m^B},$$
where $\m^B=(m^B_I)$ is given by~\eqref{eq:mB}.

We now use the following key Lemma, which is a consequence of~\cite[Thm 11.3]{Pos}. 
More precisely, combining \eqref{eq:Dn-to-En} with Remark~\ref{rk:zI-to-tI-meaning}, we can reformulate Corollary~\ref{cor:Dragon-exists} as follows.


\begin{lemma} \label{lem:lattice-count}
Let $\m'=(m'_I)_{\varnothing \neq I\subseteq [n]}$ be a tuple of non-negative integers. The number of lattice points in the relative interior of $\mQ_{\m'}$ is 
$$ |\mQ_{\m'}^\circ\cap \Z^n|\,=\,(-1)^{\codim(\mQ_{\m'})-1}\, D_n((m'_I)),$$
where $D_n$ is the $n$'th dragon polynomial and $\codim(Q_{\m'})$ is the codimension of $Q_{\m'}$.
\end{lemma}

Observe that the codimension of the face $\mF_B=\mQ_{\m^B}$ is $\ell(B)$ (since the set of linear forms that are maximized on $\mF_B$ has dimension $\ell$).
Hence Lemma~\ref{lem:lattice-count} gives $|F_B^\circ|=(-1)^{\ell(B)-1}D_n(m^B_I)$. Next, we determine the activities of the points in $F_B^\circ$. From Lemma~\ref{lem:activity-fundamentalcocycle} we obtain the following.
\begin{lemma} \label{lem:activities}
With the above notation, for all $\a\in F_B^\circ\subset \mQ_\m$, 
the internally active indices of $\a$ are the left-to-right minima of~$B$, while the externally active indices of $\a$ are the right-to-left minima of $B$.
\end{lemma}

\begin{proof}
Recall from Section~\ref{sec:permutohedra} that the polytope $\mQ_\m$ is the intersection of the half-spaces given by the rank function $f_\m$. Precisely, defining for each nonempty subset $I\subseteq [n]$ the sets
$$\ds H_I=\left\{\x\in\R^n\biggm|\sum_{i\in I}x_i=\sum_{J\cap I\ne\varnothing}m_J\right\}\text{ and } \ds H_I^{-}=\left\{\x\in\R^n\biggm|\sum_{i\in I}x_i\leq \sum_{J\cap I\ne\varnothing}m_J\right\},$$ 
we have
$$\mQ_\m=H_{[n]}\cap \,\bigcap_{\varnothing \neq I\subset [n]}H_I^-.$$
Moreover, the face $\mF_B$ is the intersection of $\mQ_\m$ with the hyperplanes $H_{B_\ell}$, $H_{B_{\ell-1}\cup B_\ell}$, $\ldots$, $H_{B_1\cup B_2\cup \cdots \cup B_\ell}$, but is not contained in any other hyperplane $H_I$ (since it has codimension $\ell$). 

By Lemma~\ref{lem:activities}, for any $\a\in Q_\m$, the index $i\in [n]$ is externally (resp., internally) active if and only if there exists a subset $I\subseteq [n]$ such that $i=\min(I)$ and $\a\in H_{I}$ (resp., $a\in H_{[n]\setminus I}$). Hence for $\a\in F_B^\circ$ the externally (resp., internally) active indices are the right-to-left (resp., left-to-right) minima of~$B$.
\end{proof}

To summarize, each element $\a\in F_B^\circ$ contributes $(x+y-1)\,x^{\lr(B)-1}\,y^{\rl(B)-1}$ to the Tutte polynomial of $Q_\m$ (since 1 is the only index which is both internally and externally active), so that their total contribution is 
$$(-1)^{\ell(B)-1}\,D_n((m^B_I))\,(x+y-1)\,x^{\lr(B)-1}\,y^{\rl(B)-1}.$$ 
Summing over all ordered set partitions of $[n]$ gives~\eqref{eq:Qm}, which concludes the proof of Theorem~\ref{thm:formulaTn}.


\section{Proofs of Lemma~\ref{lem:activity-fundamentalcocycle}, Corollary~\ref{cor:order-independence-statistics}, and Propositions~\ref{prop:TP_properties} and \ref{prop:Brylawski}}
\label{sec:TP_properties}

In this section we prove Lemma~\ref{lem:activity-fundamentalcocycle}, Corollary~\ref{cor:order-independence-statistics}, and Propositions~\ref{prop:TP_properties} and \ref{prop:Brylawski}. 
Before we embark on these proofs, let us set some notation and recall an easy fact. 
For a basis $\a$ of a polymatroid $P\subset \Z^n$ with rank function $f$, we denote
\begin{equation}
\mI(\a)\equiv \mI_P(\a)=\left\{\,I\subset [n]\biggm|\sum_{i\in I}a_i=f(I)\,\right\}.
\end{equation}
(Elements of $\mI(\a)$ were called `tight at $\a$' in \cite{Kal}.)
Note that by convention $\varnothing,[n]\in \mI(\a)$. We now observe that for all $I,J\in \mI(\a)$, both $I\cup J$ and $J\cap J$ are also in $\mI(\a)$. This is because $f$ is submodular while 
$\sum_{i\in I}a_i+\sum_{i\in J}a_i=\sum_{i\in I\cup J}a_i+\sum_{i\in I\cap J}a_i$.

\begin{proof}[Proof of Lemma~\ref{lem:activity-fundamentalcocycle}] 
We first prove the statement about external activities. Let $\a\in P$. We want to prove that $\Ext(\a)=\{\,\min(I)\mid I\in\mI(\a)\,\}$.

Suppose first that there exists $I\in\mI(\a)$ such that $i=\min(I)$. 
Then for all $j<i$ the point $\b:=\a-\e_j+\e_i$ is not in $P$ because $\sum_{i\in I}b_i>f(I)$. Hence $i$ is externally active for $\a$. 

Conversely, suppose that there does not exist a subset $I\in\mI(\a)$ such that $i=\min(I)$. Consider the set $\mI'=\{\,I\in \mI(\a)\mid i\in I\,\}$.
There is a smallest element $I_0$ of $\mI'$ for inclusion (indeed if $I,J\in \mI'$, then $I\cap J\in\mI'$).
By our assumption, there exists $j<i$ in $I_0$. Hence $j$ is contained in every subset $I\in\mI'$.
This implies that $\a-\e_j+\e_i\in P$, in particular that $i$ is not externally active.

This proves the statement about external activities. The one about internal activities follows by duality upon observing that $\Int_{P}(\a)=\Ext_{-P}(-\a)$ and that the rank function of $-P$ is $I\mapsto f([n]\setminus I)-f([n])$ (so that $\mI_{-P}(-\a)=\{\,[n]\setminus I\mid I\in \mI_P(\a)\,\}$).
\end{proof} 

\smallskip

\begin{proof}[Proof of Corollary~\ref{cor:order-independence-statistics}]
Let $P\subset \Z^n$ be a polymatroid of generic type, and let $f$ be its rank function. 
We claim that for all $\a\in P$, one has $\Int(\a)\cap\Ext(\a)=\{1\}$. 
Suppose for contradiction that $m>1$ is in $\Int(\a)\cap\Ext(\a)$. By Lemma~\ref{lem:activity-fundamentalcocycle}, there exist $I,J\in \mI(\a)$ such that $m=\min(I)=\min([n]\setminus J)$. 
However, this implies that $f(I)+f(J)=f(I\cup J)+f(I\cap J)$, while neither $I\subseteq J$ (for $m\in I\setminus J$) nor $J\subseteq I$ (as $1\in J\setminus I$). This is impossible for a polymatroid of generic type.
\end{proof}

\smallskip

\begin{proof}[Proof of Proposition~\ref{prop:TP_properties}]
Throughout we denote the rank function of $P$ by $f$.

\noindent (a) It is clear from Definition~\ref{def:polymatroid_Tutte_poly} that $\T_P(x,y)$ is
a polynomial of degree at most~$n$.
The fact that the degree of $\T_P(x,y)$ is exactly $n$ 
can be deduced by picking the lexicographically minimal basis $\a_{\min}$ of $P$ and observing that 
any index $i\in[n]$ is externally active with respect to $\a_{\min}$.

\medskip

\noindent (b) It is clear from the definition that both $\T_P(x,y+1)$ and $\T_P(x+1,y)$ have non-negative integer coefficients. We now show that $\T_P(x,y)$ has both positive and negative coefficients. 
Suppose for contradiction that all the coefficients have the same sign. In this case the formal power series 
$\ds \sum_{k=0}^\infty \T_P(x,y)(x+y)^k$ has infinite degree. 
However this is equal to the formal power series $(1-x-y)^{-1} \T_P(x,y)$, which is a polynomial since $\T_P(x,y)$ is divisible by $(x+y-1)$. A contradiction.

\medskip

\noindent (c) 
Recall that by Lemma~\ref{lem:codim=sum}, a polymatroid $P$ of codimension $c$ is a direct sum of $c$ polymatroids. 
Hence by property (e) (proved below) and the $S_n$-invariance of $T_P$ (proved later), 
$\T_P(x,y)$ is a product of $c$ polymatroid Tutte polynomials. This proves the property since every polymatroid Tutte polynomial is divisible by $(x+y-1)$.


\medskip

\noindent (d) 
Since the universal polynomial $\T_n(x,y,(z_I))$ has total degree $n$, we have 
\begin{multline*}
[x^k y^{n-k}]\T_P(x,y)=[x^k y^{n-k}]\T_n(x,y,(z_I))\\
=[x^k y^{n-k}]\T_n(x,y,(0))=[x^k y^{n-k}](x+y-1)^n,
\end{multline*}
where the last equality comes from Proposition~\ref{prop:easy-prop-Tn}(b). 

\medskip
\noindent (e) Let $Q\subset \Z^n$ be a polymatroid. 
It is clear that for any $\c\in P\oplus Q$ there is a unique pair $(\a,\b)\in P\times Q$ such that $\c=(a_1,\ldots,a_n,b_1,\ldots,b_m)$.
Moreover for $i,j\in[n+m]$, the point $\c+\e_i-\e_j$ is in $P\oplus Q$ if and only if either ($i,j\leq n$ and $\a+\e_i-\e_j\in P$) or ($i,j>n$ and $\b+\e_{i-n}-\e_{j-n}\in Q$).
This implies $\Int_{P\oplus Q}(\c)=\Int_P(\a)\cup (n+\Int_Q(\b))$ and $\Ext_{P\oplus Q}(\c)=\Ext_P(\a)\cup (n+\Ext_Q(\b))$ 
Hence, 
\begin{eqnarray*}
\T_{P\oplus Q}(x,y)&=& \sum_{\c\in P\oplus Q} x^{\oi(\c)}\,y^{\oe(\c)}\,(x+y-1)^{\ie(\c)}\\
&=&\sum_{(\a,\b)\in P\times Q} x^{\oi(\a)+\oi(\b)}\,y^{\oe(\a)+\oe(\b)}\,(x+y-1)^{\ie(\a)+\ie(\b)}\\
&=&\T_{P}(x,y)\,\T_{Q}(x,y).
\end{eqnarray*}

Similarly for any $\c\in P\boxplus Q$ there is a unique pair $(\a,\b)\in P\times Q$ such that $\c=(a_1+b_1,a_2,\ldots,a_n,b_2,\ldots,b_m)$. 
(The decomposition of $c_1$ is determined by the levels of $P$ and $Q$.) 
Moreover for $1\leq i<j\leq n+m-1$, the point $\c+\e_i-\e_j$ is in $P\boxplus Q$ if and only if either ($i,j\leq n$ and $\a+\e_i-\e_j\in P$), 
or ($i,j>n$ and $\b+\e_{i+1-n}-\e_{j+1-n}\in Q$), 
or ($i\leq n<j$ and $\a-\e_1+\e_i\in P$ and $\a+\e_1-\e_{j+1-n}\in Q$).
This implies 
$\Int_{P\boxplus Q}(\c)=\Int_P(\a)\cup \{\,n+i-1\mid i\in\Int_Q(\b)\setminus\{1\}\,\}$. 
Similarly, $\Ext_{P\boxplus Q}(\c)=\Ext_P(\a)\cup \{\,n+i-1\mid i\in\Ext_Q(\b)\setminus\{1\}\,\}$.
Consequently, 
\begin{eqnarray*}
\T_{P\boxplus Q}(x,y)&=&\sum_{(\a,\b)\in P\times Q} x^{\oi(\a)+\oi(\b)}\,y^{\oe(\a)+\oe(\b)}\,(x+y-1)^{\ie(\a)+\ie(\b)-1}\\
&=&\frac{\T_{P}(x,y)\,\T_{Q}(x,y)}{x+y-1}.
\end{eqnarray*}
 
\medskip
\noindent (f) 
It is easy to verify that $\widehat{P}^i$ in the $r_i=0$ case, as well as $P\setminus i$ and $P/i$ in the $r_i=1$ case, are polymatroids. See Section~\ref{subsec:deletion-contraction} for more on the underlying geometry.

By the $S_n$-invariance of $\T_P$ we may assume $i=n$. In the case $r_n=0$ we have $P=\widehat{P}^n\oplus Q$, where $Q$ is the matroid in $\Z$ with unique basis $\a=f(\{n\})$. Hence~\eqref{eq:deletion-coloop} follows from property (e). 

In the case $r_n=1$, any basis $\a\in P$ satisfies $a_n\in \{f(\{n\})-1, f(\{n\})\}$. Hence every basis of $P$ corresponds either to a basis of $P\setminus n$ (if $a_n=f(\{n\})-1$) or to a basis of $P/n$ (if $a_n=f(\{n\})$).
If $a_n=f(\{n\})-1$, then it is clear that $n\in \Int_P(\a)$, and the Exchange axiom easily implies that $n\not\in \Ext_P(\a)$. As to the other indices, the internal or external activity with respect to $\a\in P$ is the same as with respect to $\widehat{\a}^n\in P\setminus n$. 

To summarize, if $a_n=f(\{n\})-1$, then $\Int_P(\a)=\Int_{P\setminus n}(\widehat{\a}^n)\cup\{n\}$ and $\Ext_P(\a)=\Ext_{P\setminus n}(\widehat{\a}^n)$. 
Similarly, if $a_n=f(\{n\})$, then $\Int_P(\a)=\Int_{P/ n}(\widehat{\a}^n)$ and $\Ext_P(\a)=\Ext_{P/ n}(\widehat{\a}^n)\cup\{n\}$. This gives \eqref{eq:deletion-contraction}.

\medskip
\noindent (g) The property $\T_P(1,1)=|P|$ is clear from the definition. We now prove $\T_P(0,0)=(-1)^{\codim(P)}|P^\circ|$, where $P^\circ=\mP^\circ\cap \Z^n$. 
By property (e), it suffices to do so when $\codim(P)=1$, so we assume this condition from now on. 
Also, by definition 
$$\T_P(0,0)=\sum_{\a\in P\,:\,\Int(\a)\triangle \Ext(\a)=\varnothing}(-1)^{\Int(\a)\cap \Ext(\a)},$$
where $A\triangle B:=(A\setminus B)\cup (B\setminus A)$ denotes the symmetric difference of the sets $A$ and $B$. Hence it suffices to show that for all $\a\in P$, we have $\a\in P^\circ$ if and only if $\Int(\a)\triangle \Ext(\a)=\varnothing$; furthermore in this case $\Int(\a)=\Ext(\a)=\{1\}$. 

Let us prove this property. For $I\subseteq [n]$, we consider the hyperplane $H_I=\{\,\x\in\Z^n\mid\sum_{i\in I}x_i=f(I)\,\}$. 
Let us first assume $\a\in P^\circ$. This implies $\a\notin \bigcup_{\varnothing \neq I\subsetneq [n]}H_I$, in particular
for all $i,j\in [n]$, we have $\a+\e_i-\e_j\in P$. Thus $\Int(\a)=\Ext(\a)=\{1\}$, and $\Int(\a)\triangle \Ext(\a)=\varnothing$. 

Conversely, we now assume that $\a\in P\setminus P^\circ$, and want to prove $\Int(\a)\triangle \Ext(\a)\neq\varnothing$. 
By Lemma~\ref{lem:activity-fundamentalcocycle}, we have $\Ext(\a)=\{\,\min(I)\mid\varnothing\neq I\in\mI(\a)\,\}$ and $\Int(\a)=\{\,\min([n]\setminus I)\mid[n]\neq I\in\mI(\a)\,\}$. 
Since $\a\notin P^\circ$, there exists a set $I_0\in \mI(\a)$ which is neither empty nor equal to $[n]$. 
Hence $\max(\Int(\a)\cup \Ext(\a))>1$. Therefore, if $m:=\max(\Int(\a)\cap \Ext(\a))$ is equal to $1$, then $\Int(\a)\triangle \Ext(\a)\neq\varnothing$. 
Let us now assume $m>1$, and consider $I,J\in\mI(\a)$ such that $m=\min(I)=\min([n]\setminus J)$. We have $I\neq [n]\setminus J$ (otherwise $P\subset H_I$ and $\codim(P)>1$), whence we get $I\cap J\neq \varnothing$ or $I\cup J\neq [n]$. Moreover, both $I\cap J$ and $I\cup J$ are in $\mI(\a)$.
If $I\cap J\neq \varnothing$, then $k:=\min(I\cap J)>m$ and thus $k\in \Ext(\a)\triangle \Int(\a)$. 
Similarly, if $I\cup J\neq [n]$, then $k:=\min([n]\setminus(I\cup J))>m$, implying $k\in \Int(\a)\triangle \Ext(\a)$.
In both cases we obtain $\Int(\a)\triangle \Ext(\a)\neq \varnothing$, as wanted. 

\medskip
\noindent (h) It suffices to prove $\ds \frac{\T_P(x,y)}{(x+y-1)}\bigg|_{x=1,y=0}=|(\mP-\Delta)\cap \Z^n|$, as the other identity follows by duality. 
Let $\mA=\{\,\a\in P\mid\Ext(\a)=\{1\}\,\}$ and let $\mB=\{\,\a\in \Z^n\mid\a-\e_1+\Delta\subseteq \mP\,\}$. Clearly, $\ds \frac{\T_P(x,y)}{(x+y-1)}\bigg|_{x=1,y=0}=|\mA|$, and $|(\mP-\Delta)\cap \Z^n|=|\mB|$, so it suffices to prove $\mA=\mB$.

For $i\in \{2,3,\ldots,n\}$, let $\mA_i=\{\,a\in P\mid\exists j<i,~\a+\e_i-\e_j\in P\,\}$ and $\mB_i=\{\,\a\in P\mid\a+\e_i-\e_1\in P\,\}$. It is easy to see that $\mB=\bigcap_{i=2}^n\mB_i\subseteq \bigcap_{i=2}^n\mA_i=\mA$. 

We now prove, by induction on $k$, that for all $k\in \{2,3,\ldots,n\}$, we have $\bigcap_{i=2}^k\mA_i\subseteq \bigcap_{i=2}^k\mB_i$.
The base case $k=2$ is trivial, and we now assume $\bigcap_{i=2}^{k-1}\mA_i\subseteq \bigcap_{i=2}^{k-1}\mB_i$. Let $\a\in \bigcap_{i=2}^{k}\mA_i$. We need to prove that $\b=\a+\e_k-\e_1$ is in $P$. 
Since $\a\in\mA_k$, there exists $j<k$ such that $\a+\e_k-\e_j\in P$. As $\a\in\bigcap_{i=2}^{k-1}\mA_i\subseteq \bigcap_{i=2}^{k-1}\mB_i$, we get $\a+\e_j-\e_1\in P$. Hence, by convexity, $\a+\frac{\e_k-\e_1}{2}\in \mP$. This, in turn, implies that for all subsets $I\subset [n]$ containing $k$ but not $1$, we have $\sum_{i\in I}a_i< f(I)$. Thus we also have $\sum_{i\in I}b_i\leq f(I)$, which implies that $\b$ is in $P$, as wanted.
\end{proof}

\smallskip

\begin{proof}[Proof of Proposition \ref{prop:Brylawski}]
It is clear that $T_P(x,y)$ is a Laurent polynomial since the degree of $\T_P(x,y)$ is $n$. 
Let us denote $b_{p,q}=[x^p y^q]B(x,y)T_P(x,-1/y)$.
Observe that for all integers $p,q$,
\begin{eqnarray*}
b_{p,q}-b_{p-1,q-1}-b_{p-1,q}
&=& [x^p y^q](1-x(1+y))B(x,y)T_P(x,-1/y)\\
&=& [x^p y^q]T_P(x,-1/y).
\end{eqnarray*}
The right-hand side of the above equation is equal to 0 whenever $q>-\ell$, since $\ell$ is the minimum $y$-degree of $T_P(x,y)$.
Thus, if for a given $k\in \Z$ one has $b_{k+q,q}=0$ 
for all $q\geq -\ell$, then $b_{k+1+q,q}=b_{k+q,q-1}$ for all $q>-\ell$.  Hence in this case the coefficients $b_{k+1+q,q}$
for $q\geq -\ell$ are all equal.
Since for $k$ sufficiently small one has $b_{k+q,q}=0$ 
for all $q$, there exists an integer $k_0\in \Z\cup\{+\infty\}$ and a constant $c\neq 0$ such that  
for all $q\geq -\ell$ one has
$b_{p,q}=0$ 
for all  $p< k_0+q$ and $b_{k_0+q,q}=c$. 

It remains to show that $k_0=n$ and $c=(-1)^{n-d}$. Since $\ell\geq -d$, it suffices to show
\begin{equation}\label{eq:Bryld}
\forall p< n,~b_{p+d,d}=0, \textrm{ and } b_{n+d,d}=(-1)^{n-d}.
\end{equation}
Let us denote $\tau_{i,j}=[x^i y^j]\T_P(x,y)$, so that
$$T_P(x,y)=\sum_{i,j\geq 0}\tau_{i,j}\,x^{i+d-n}y^{j-d}(x+y-xy)^{n-i-j}.$$
One gets
$$B(x,y)T_P(x,-1/y)=\sum_{i,j\geq 0}\tau_{i,j}\,(-xy)^{i+d-n}(1-x-xy)^{n-i-j-1}.$$
Hence
$$ 
b_{p+d,d}=(-1)^d\sum_{i,j\geq 0}\tau_{i,j}(-1)^{n-i}[x^{p+n-i}y^{n-i}](1-x-xy)^{n-i-j-1}.
$$
Note that for $i+j<n$, $[x^{p+n-i}y^{n-i}](1-x-xy)^{n-i-j-1}=0$, and that for $i+j>n$, $\tau_{i,j}=0$. 
Moreover, by Proposition \ref{prop:TP_properties}(d), $\tau_{n-j,j}={n\choose j}$. Thus,
$$
b_{p+d,d}
=(-1)^d\sum_{j=0}^n{n\choose j}(-1)^{j}[x^{p+j}y^{j}](1-x-xy)^{-1}.
$$
For $p<0$, $[x^{p+j}y^{j}](1-x-xy)^{-1}=0$, hence $b_{p+d,d}=0$. 
For $p\geq 0$,  $[x^{p+j}y^{j}](1-x-xy)^{-1}={p+j \choose j}$, hence
\begin{eqnarray*}
b_{p+d,d}
&=&(-1)^d\sum_{j=0}^n{n\choose j}(-1)^{j}{p+j \choose j}\\
&=&(-1)^d\sum_{j=0}^n[x^{n-j}](1+x)^n\cdot  [x^j](1+x)^{-p-1}\\
&=&(-1)^{d}[x^n](1+x)^{n-p-1}.
\end{eqnarray*}
This readily gives \eqref{eq:Bryld}.
\end{proof}


\section{Decomposition of the integer lattice into cones}
\label{sec:cone-decomposition}

In this section we describe, for a (not necessarily finite) M-convex set $P\subset \Z^n$, a decomposition of the integer lattice $\Z^n$ into 
a disjoint union of 
cones indexed by the elements of $P$. 
We will use this cone decomposition in the next section to give another characterization of the polymatroid Tutte polynomial, and prove the $S_n$-invariance stated in Theorem~\ref{th:order_independence}.



\begin{definition}
\label{def:C_cone}
Let $P\subset \Z^n$ be an M-convex set.
For $\a\in P$, define the \emph{cone} $C(\a)=C_P(\a)\subset \Z^n$ to consist
of all points $\a+\x\in\Z^n$ such that, for any $i\in[n]$,
\begin{enumerate}
\item
if $i\not\in\Int(\a)$ and $i\not\in\Ext(\a)$, then $x_i=0$;
\item
if $i\in\Int(\a)$ and $i\not\in\Ext(\a)$, then $x_i\leq 0$;
\item
if $i\not\in\Int(\a)$ and $i\in\Ext(\a)$, then $x_i\geq 0$;
\item
if $i\in\Int(\a)$ and $i\in \Ext(\a)$, then $x_i$ can be any integer.
\end{enumerate}
\end{definition}

We now define the Manhattan distance on $\Z^n$ before stating the main result.

\begin{definition}
The \emph{Manhattan distance} $d_1(\a,\b)$ between two points $\a,\b\in\R^n$ is
$$
d_1(\a,\b)=||\a-\b||_1 := \sum_{i=1}^n |a_i - b_i|.
$$
For two subsets $A$ and $B$ of $\Z^n$, define
$$
d_1(A,B) :=\min_{\a\in A,\, \b\in B} d_1(\a,\b).
$$
\end{definition}

\begin{theorem}
\label{th:decomposition_into_cones}
Let $P\subset\Z^n$ be a nonempty $M$-convex set.
\begin{compactenum}
\item The collection of cones $C(\a)=C_P(\a)$, for $\a$ in $P$, gives a disjoint decomposition of $\Z^n$:
$$
\biguplus_{\a\in P} C_P(\a) = \Z^n.
$$
\item If $\a\in P$ and $\c\in C(\a)$,
then $d_1(\a,\c)=d_1(P,\c)$.
In other words, the point $\a$ is one of the points of $P$ 
closest to the point $\c$ with respect to the Manhattan distance.
\end{compactenum}
\end{theorem}

\begin{remark}
Theorem~\ref{th:decomposition_into_cones} is the analogue, for M-convex sets, of a classical result about the Tutte polynomial of matroids due to Crapo~\cite{crapo}. Indeed, for a matroid $M$ defined on the ground set $[n]$, one can consider the restriction of the cone-decomposition given by Theorem~\ref{th:decomposition_into_cones} to the subset $\{0,1\}^n\subset \Z^n$. This translates into a partition of the Boolean lattice $(2^{[n]},\subseteq)$ into intervals indexed by the bases of $M$, which is exactly the partition exhibited in~\cite{crapo}.
\end{remark}

\begin{proof}
Let us first show that the cones $C(\a)$ and $C(\b)$, for two different
elements $\a, \b \in P$, are disjoint. Suppose that this is not true,
and there is $\c\in C(\a)\cap C(\b)$.
Let $i$ be the maximal index such that $a_i\ne b_i$. 
Without loss of generality, assume that $a_i<b_i$.
According to the Exchange Axiom, there exists $j$ such that $a_j>b_j$
and $\a+\e_i-\e_j\in P$ and $\b-\e_i+\e_j\in P$.
Because of our choice of $i$, we have $j<i$.
Thus $i\not\in \Ext(\a)$ and $i\not\in\Int(\b)$. Hence $c_i\leq a_i$ and $c_i\geq b_i$. 
We reach a contradiction since $a_i<b_i$.

Next, let us show that, for any $\c\in \Z^n$, there exists $\a\in P$ such that
$\c\in C(\a)$.
Pick any element $\a'\in P$. 
We say that an index $i$ is \emph{bad} 
for $\a'$ (with respect to $\c$) 
if we have either ($a'_i<c_i$ and $i\not\in\Ext(\a')$),
or ($a'_i>c_i$ and $i\not\in\Int(\a')$).

If there is no bad index for $\a'$, then $\c\in C(\a')$ and we are done.
Otherwise, let~$i$ be the maximal bad index. If $a'_i<c_i$, then $i\not\in\Ext(\a')$, hence there exists $j<i$ such that $\a'':=\a'+\e_i-\e_j\in P$. If $a'_i>c_i$, then $i\not\in\Int(\a')$, hence there exists $j<i$ such that $\a'':=\a'-\e_i+\e_j\in P$. 
 
We now show that $\a''$ has no bad index greater than $i$ (with respect to~$\c$). 
For $k>i$ we know that $a_k''=a_k'$ and want to show that $k$ is not a bad index for $\a''$. It suffices to show that if $k\notin \Ext(\a'')$ then 
$k\notin \Ext(\a')$, and if $k\notin \Int(\a'')$ then $k\notin \Int(\a')$. Suppose that $k\notin \Ext(\a'')$. Then there is $\ell<k$ such that $\a''+\e_k-\e_\ell\in P$. Applying the Exchange Axiom to the pair $\a'$ and $\a''+\e_k-\e_\ell$, we deduce that there is $r < k\,$ ($r\in\{\ell,i,j\}$) such that $\a'+\e_k-\e_r\in P$. Thus $k\notin \Ext(\a')$. Similarly, if $k\notin \Int(\a'')$, then there exists $r<k$ such that $\a'-\e_k+\e_r\in P$, hence $k\notin \Int(\a')$.

Thus either $\a''$ has no bad index at all, or the maximal bad index for $\a''$ 
is less than $i$, or the maximal bad index for $\a''$ is $i$ and 
$|a_i''-c_i| < |a_i'-c_i|$.

If $\a''$ has a bad index, then we can apply the same operation 
to the element $\a''$ and find a new element $\a'''\in P$, etc. 
The above argument shows that, after finitely many steps, we will 
find an element $\a\in P$ with no bad index.
This means exactly that $\c\in C(\a)$.


Finally, observe that, for the sequence of points $\a',\a'', \a''', \dots\in P$
that we constructed above we have
$d_1(\a',\c)\geq d_1(\a'',\c) \geq d_1(\a''',\c)\geq \cdots$.
Indeed, $|a'_i-c_i| = |a''_i-c_i| - 1$
and $|a'_j-c_j| \leq |a''_j-c_j| + 1$, and all other coordinates
of~$\a'$ and $\a''$ are the same.
Since $\a'$ can be any point of $P$ and the sequence $\a',\a'',\a''',\dots$
converges, after finitely many steps, to $\a\in P$ such that $\c\in C(\a)$,
we deduce that $d_1(\a,\c)\leq d_1(\a',\c)$ for any $\a'\in P$,
which proves part (2) of the theorem.
\end{proof}

\begin{example}
In Figure~\ref{fig:cones} we show the cone decomposition for a polymatroid $P\subset \Z^3$. The polymatroid $P$ is the same as the one represented in Figures~\ref{fig:polymatroid} and~\ref{fig:polymatroid_activities}. Note that every cone of the decomposition contains the line $\Z\, \e_1$ as a Minkowski summand (because the index $1$ is always both internally and externally active). Hence one may project the cone decomposition onto the space spanned by the coordinate vectors $\e_2,\e_3$ without loss of information (one could equivalently consider the induced decomposition of the hyperplane $\{\,\x\in\Z^n\mid\sum x_i=\level(P)\,\}$ containing $P$). This is represented on the right of Figure~\ref{fig:cones}. 

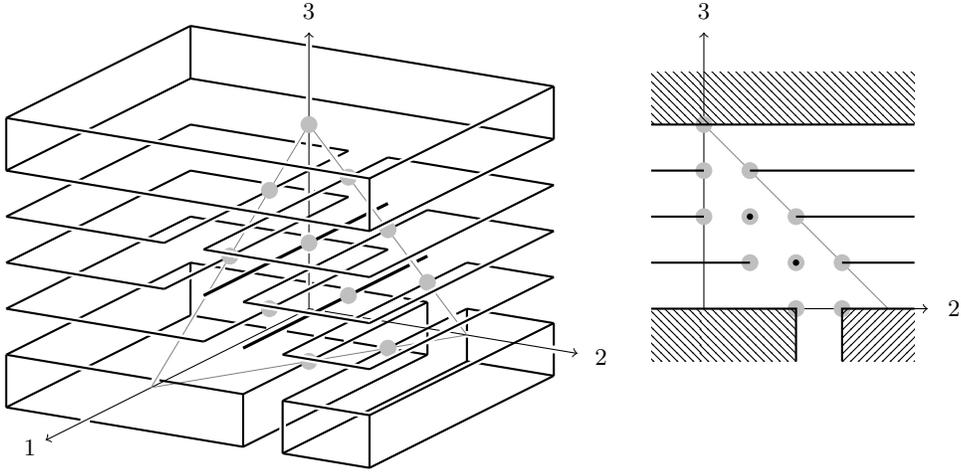
\begin{figure}[h]
\centering
\begin{tikzpicture}[scale=.35]

\draw (0,0) -- (0,2.1);
\draw [->] (0,2.9) -- (0,10.5);
\draw [->] (0,0) -- (-10,-5);
\draw (0,0) -- (.18,-.03);
\draw (.6,-.1) -- (1.68,-.28);
\draw [->] (2.04,-.34) -- (10.2,-1.7);
\node at (-10.6,-5.3) {\small $1$};
\node at (11.1,-1.85) {\small $2$};
\node at (0,11.3) {\small $3$};
\draw [help lines] (6,-1) -- (0,7) -- (-6,-3) -- cycle;

\draw [thick] (-2.5,-3.25) -- (4.5,.25);
\draw [thick] (-2.5,-3.25) -- (-2.5,-5.25);
\draw [thick] (-4.5,1.75) -- (-3.36,1.56);
\draw [thick] (-3.18,1.53) -- (-.12,1.02);
\draw [thick] (.12,.98) -- (1.2,.8);
\draw [thick] (2.1,.65) -- (3.18,.47);
\draw [thick] (3.6,.4) -- (4.5,.25);
\draw [thick] (4.5,-1.75) -- (4.5,-1.35);
\draw [thick] (4.5,-1.15) -- (4.5,-.9);
\draw [thick] (4.5,-.6) -- (4.5,.25);
\draw [thick] (4.5,-1.75) -- (-2.5,-5.25);
\draw [thick] (-11.5,-3.75) -- (-8.68,-4.22);
\draw [thick] (-8.32,-4.28) -- (-2.5,-5.25);
\draw [thick] (-11.5,-3.75) -- (-11.5,-1.75);
\draw [thick] (-4.5,1.75) -- (-11.5,-1.75);
\draw [thick] (-4.5,1.75) -- (-4.5,-.25);
\draw [thick] (-11.5,-3.75) -- (-4.5,-.25);
\draw [thick] (4.5,-1.75) -- (3.42,-1.57);
\draw [thick] (2.58,-1.43) -- (1.68,-1.28);
\draw [thick] (1.32,-1.22) -- (.6,-1.1);
\draw [thick] (.18,-1.03) -- (-.84,-.86);
\draw [thick] (-1.62,-.73) -- (-4.26,-.29);
\draw [line width=2.5,white] (-7,-2.5) -- (-9.7,-2.05);
\draw [thick] (-2.5,-3.25) -- (-5.68,-2.72);
\draw [thick] (-5.98,-2.67) -- (-11.5,-1.75);

\draw [thick] (-1,-3.5) -- (5.3,-.35);
\draw [thick] (5.6,-.2) -- (6,0);
\draw [line width=2.5,white] (-1,-4) -- (-1,-5);
\draw [thick] (-1,-3.5) -- (-1,-5.5);
\draw [thick] (6,-2) -- (6,-1.1);
\draw [thick] (6,-.85) -- (6,0);
\draw [thick] (9.3,-.55) -- (6,0);
\draw [thick] (2.3,-6.05) -- (9.3,-2.55);
\draw [thick] (-1,-5.5) -- (2.3,-6.05);
\draw [thick] (6,-2) -- (-1,-5.5);
\draw [thick] (6,-2) -- (9.3,-2.55);
\draw [thick] (9.3,-.55) -- (9.3,-1.4);
\draw [thick] (9.3,-1.7) -- (9.3,-2.55);
\draw [line width=2.5,white] (-.4,-3.6) -- (2.3,-4.05);
\draw [thick] (-1,-3.5) -- (2.3,-4.05);
\draw [thick] (2.3,-6.05) -- (2.3,-4.05);
\draw [line width=2.5,white] (7.3,-1.55) -- (5.3,-2.55);
\draw [thick] (9.3,-.55) -- (2.3,-4.05);

\draw [line width=2.5,white] (-3,-.75) -- (1,1.25);
\draw [thick] (-4,-1.25) -- (3,2.25);
\draw [line width=2.5,white] (-4.3,-1.2) -- (-10,-.25);
\draw [thick] (-4,-1.25) -- (-11.5,0);
\draw [thick] (-4.5,3.5) -- (-11.5,0);
\draw [thick] (-4.5,3.5) -- (-2.4,3.15);
\draw [thick] (-2.22,3.12) -- (-.3,2.8);
\draw [thick] (.3,2.7) -- (3,2.25);

\draw [very thick] (-2.5,-1.5) -- (3.82,1.66);
\draw [very thick] (4.1,1.8) -- (4.5,2);

\draw [fill=lightgray,lightgray] (0,-2) circle [radius=.3];
\draw [fill=lightgray,lightgray] (3,-1.5) circle [radius=.3];

\draw [thick] (-1,-1.75) -- (6,1.75);
\draw [line width=2.5,white] (-.7,-1.8) -- (2,-2.25);
\draw [thick] (-1,-1.75) -- (2.3,-2.3);
\draw [thick] (6,1.75) -- (9.3,1.2);
\draw [line width=2.5,white] (2.5,-2.2) -- (8.5,.8);
\draw [thick] (2.3,-2.3) -- (9.3,1.2);

\draw [line width=2.5,white] (-5,1) -- (-.6,3.2);
\draw [thick] (-5.5,.75) -- (1.5,4.25);
\draw [line width=2.5,white] (-5.8,.8) -- (-10.3,1.55);
\draw [thick] (-5.5,.75) -- (-11.5,1.75);
\draw [thick] (-4.5,5.25) -- (-11.5,1.75);
\draw [thick] (-4.5,5.25) -- (-1.8,4.8);
\draw [thick] (-1.08,4.68) -- (-.18,4.53);
\draw [thick] (.18,4.47) -- (1.5,4.25);

\draw [line width=2.5,white] (-4.4,.3) -- (1,3);
\draw [very thick] (-4,.5) -- (2.3,3.65);
\draw [very thick] (2.58,3.79) -- (3,4);

\draw [fill=lightgray,lightgray] (-1.5,0) circle [radius=.3];
\draw [fill=lightgray,lightgray] (4.5,1) circle [radius=.3];

\draw [line width=2.5,white] (-1.5,.75) -- (2.5,2.75);
\draw [thick] (-2.5,.25) -- (4.5,3.75);
\draw [thick] (9.3,2.95) -- (4.5,3.75);
\draw [line width=2.5,white] (-2.2,.2) -- (2.3,-.55);
\draw [thick] (-2.5,.25) -- (2.3,-.55);
\draw [line width=2.5,white] (8.3,2.45) -- (2.3,-.55);
\draw [thick] (9.3,2.95) -- (2.3,-.55);

\draw [line width=2.5,white] (-4.5,3) -- (-.5,5);
\draw [thick] (-5.5,2.5) -- (.8,5.65);
\draw [thick] (1.1,5.8) -- (1.5,6);
\draw [line width=2.5,white] (-5.8,2.55) -- (-9.1,3.1);
\draw [thick] (-5.5,2.5) -- (-11.5,3.5);
\draw [thick] (-4.5,7) -- (-11.5,3.5);
\draw [thick] (-4.5,7) -- (-.54,6.34);
\draw [thick] (-.3,6.3) -- (-.12,6.27);
\draw [thick] (.12,6.23) -- (.48,6.17);
\draw [thick] (.78,6.12) -- (1.5,6);

\draw [fill=lightgray,lightgray] (-3,2) circle [radius=.3];

\draw [line width=2.5,white] (-3,2.75) -- (1,4.75);
\draw [thick] (-4,2.25) -- (3,5.75);
\draw [thick] (9.3,4.7) -- (3,5.75);
\draw [line width=2.5,white] (-3.7,2.2) -- (2,1.25);
\draw [thick] (-4,2.25) -- (2.3,1.2); 
\draw [line width=2.5,white] (7.3,3.7) -- (3.3,1.7);
\draw [thick] (9.3,4.7) -- (2.3,1.2);

\draw [fill=lightgray,lightgray] (3,3) circle [radius=.3];
\draw [fill=lightgray,lightgray] (1.5,5) circle [radius=.3];

\draw [line width=2.5,white] (2,3) -- (-10.9,5.15);
\draw [thick] (2.3,2.95) -- (-11.5,5.25);
\draw [thick] (-4.5,8.75) -- (-11.5,5.25);
\draw [thick] (-11.5,7.25) -- (-11.5,5.25);
\draw [thick] (-4.5,8.75) -- (-4.5,10.75);
\draw [thick] (-11.5,7.25) -- (-4.5,10.75);
\draw [thick] (-4.5,8.75) -- (-.12,8.02);
\draw [thick] (.12,7.98) -- (9.3,6.45);
\draw [line width=2.5,white] (2.3,4.95) -- (-10.9,7.15);
\draw [thick] (2.3,4.95) -- (-11.5,7.25);
\draw [line width=2.5,white] (2.3,4.8) -- (2.3,3.1);
\draw [thick] (2.3,4.95) -- (2.3,2.95);
\draw [line width=2.5,white] (3,5.3) -- (9,8.3);
\draw [thick] (2.3,4.95) -- (9.3,8.45);
\draw [line width=2.5,white] (8.4,6) -- (2.6,3.1);
\draw [thick] (9.3,6.45) -- (2.3,2.95);
\draw [thick] (9.3,6.45) -- (9.3,8.45);
\draw [thick] (-4.5,10.75) -- (-.12,10.02);
\draw [thick] (.12,9.98) -- (9.3,8.45);

\draw [fill=lightgray,lightgray] (1.5,.5) circle [radius=.3];
\draw [fill=lightgray,lightgray] (0,2.5) circle [radius=.3];
\draw [fill=lightgray,lightgray] (-1.5,4.5) circle [radius=.3];
\draw [fill=lightgray,lightgray] (0,7) circle [radius=.3];

\begin{scope}[shift={(15,.0)}]
\draw [->] (0,0) -- (0,10.5);
\node at (0,11.3) {\small $3$};
\draw [->] (0,0) -- (8.5,0);
\node at (9.5,0) {\small $2$};
\draw [help lines] (7,0) -- (0,7);

\draw [fill=lightgray,lightgray] (3.5,0) circle [radius=.3];
\draw [fill=lightgray,lightgray] (5.25,0) circle [radius=.3];
\draw [fill=lightgray,lightgray] (1.75,1.75) circle [radius=.3];
\draw [fill=lightgray,lightgray] (3.5,1.75) circle [radius=.3];
\draw [fill=lightgray,lightgray] (5.25,1.75) circle [radius=.3];
\draw [fill=lightgray,lightgray] (0,3.5) circle [radius=.3];
\draw [fill=lightgray,lightgray] (1.75,3.5) circle [radius=.3];
\draw [fill=lightgray,lightgray] (3.5,3.5) circle [radius=.3];
\draw [fill=lightgray,lightgray] (0,5.25) circle [radius=.3];
\draw [fill=lightgray,lightgray] (1.75,5.25) circle [radius=.3];
\draw [fill=lightgray,lightgray] (0,7) circle [radius=.3];

\path [pattern=north west lines] (-2,-2) rectangle (3.5,0);
\draw [thick] (-2,0) -- (3.5,0) -- (3.5,-2);
\path [pattern=north east lines] (5.25,-2) rectangle (8,0);
\draw [thick] (5.25,-2) -- (5.25,0) -- (8,0);
\draw [thick] (-2,1.75) -- (1.75,1.75);
\draw [fill] (3.5,1.75) circle [radius=.1];
\draw [thick] (5.25,1.75) -- (8,1.75);
\draw [thick] (-2,3.5) -- (0,3.5);
\draw [fill] (1.75,3.5) circle [radius=.1];
\draw [thick] (3.5,3.5) -- (8,3.5);
\draw [thick] (-2,5.25) -- (0,5.25);
\draw [thick] (1.75,5.25) -- (8,5.25);
\path [pattern=north west lines] (-2,7) rectangle (8,9);
\draw [thick] (-2,7) -- (8,7);
\end{scope}

\end{tikzpicture}
\caption{Two views of the same cone decomposition of $\Z^3$. The picture on the right is the projection in the subspace $\textrm{span}(\e_2,\e_3)$ of the cone decomposition of $\Z^3$ shown on the left.}
\label{fig:cones}
\end{figure}
\end{example}


\section{Corank-nullity construction of the polymatroid Tutte polynomial}
\label{sec:order-invariance}

In this section we provide another construction for the polymatroid Tutte polynomial. This alternative construction is clearly invariant under permutations of coordinates, hence it provides a proof of Theorem~\ref{th:order_independence}. In fact, this construction is the analogue of the corank-nullity definition of the Tutte polynomial of matroids given by~\eqref{eq:corank-nullity-mat}.


First, we show that the Manhattan distance $d_1(P,\c)$ between a 
nonempty M-convex set $P\subset\Z^n$ and a lattice point $\c\in\Z^n$ splits canonically into two summands.

\begin{definition}
For $\a,\b\in\R^n$, define
$$
d_1^{>}(\a,\b) := \sum_{i\,:\,a_i>b_i} (a_i-b_i)
\quad \text{and} \quad
d_1^{<}(\a,\b) := \sum_{i\,:\,a_i<b_i} (b_i-a_i).
$$
\end{definition}

The functions $d_1^>$ and $d_1^<$ are non-symmetric (in fact $d_1^{<}(\a,\b) = d_1^{>}(\b,\a)$), but also non-negative and, clearly, $d_1(\a,\b) = d_1^{>}(\a,\b)+d_1^{<}(\a,\b)$.

\begin{definition}
For a nonempty M-convex set $P\subset\Z^n$, and a lattice point 
$\c\in\Z^n$, define 
$$
d_1^{>}(P,\c) := \min_{\a\in P} d_1^{>}(\a,\c)
\quad\text{and}\quad
d_1^{<}(P,\c) := \min_{\a\in P} d_1^{<}(\a,\c).
$$
\end{definition}

\begin{lemma} 
\label{lem:d_geq_leq}
Let $P\subset \Z^n$ be a nonempty M-convex set, and let $\c\in \Z^n$. 
For any point $\a\in P$ such that $d_1(\a,\c)=d_1(P,\c)$, we have
$d_1^{>}(\a,\c) = d_1^{>}(P,\c)$ and $d_1^{<}(\a,\c) = d_1^{<}(P,\c)$. 
Thus
$$
d_1(P,\c)=d_1^{>}(P,\c)+d_1^{<}(P,\c).
$$
\end{lemma}


\begin{proof}
Let us fix a point $\a\in P$ such that $d_1(\a,\c)=d_1(P,\c)$. 
Among all points $\b\in P$ such that 
$d_1^{>}(\b,\c) = d_1^{>}(P,\c)$, let us pick a point $\b$ with minimal possible Manhattan distance $d_1(\a,\b)$. We will show that $\a=\b$. 

Suppose for contradiction that $\a\ne \b$.
Since $\a,\b\in P$
we have $\sum a_i =\level(P)= \sum b_i$. Hence, there is an index $i$ such that $a_i>b_i$.
According to the Exchange Axiom, there is an index $j$ such that 
$a_j<b_j$ and $\a'=\a-\e_i+\e_j$ and $\b'=\b+\e_i-\e_j$ both belong to $P$.
Notice that $d_1(\a,\b')=d_1(\a,\b)-2$.
At this point we split into several cases and get contradictions in each case:
\begin{enumerate}
\item In the case $c_i\geq a_i$, we get $c_i>b_i$ hence $d_1^{>}(\b',\c)\leq d_1^{>}(\b,\c)$. Since $d_1(\a,\b')<d_1(\a,\b)$, this contradicts our choice of~$\b$.
\item In the case $c_j\leq a_j$, we get $c_j<b_j$ hence $d_1^{>}(\b',\c)\leq d_1^{>}(\b,\c)$. Since $d_1(\a,\b')<d_1(\a,\b)$, this contradicts our choice of~$\b$.
\item In the case ($c_i<a_i$ and $c_j>a_j$), we get $d_1(\a',\c)<d_1(\a,\c)$, which contradicts our choice of~$\a$. 
\end{enumerate}
This shows that $\a=\b$. So $d_1^{>}(\a,\c) = d_1^{>}(P,\c)$. Similarly, 
$d_1^{<}(\a,\c) = d_1^{<}(P,\c)$. 

Finally, we get
$d_1(P,\c) = d_1(\a,\c) = d_1^{>}(\a,\c) + d_1^{<}(\a,\c) = 
d_1^{>}(P,\c) + d_1^{<}(P,\c)$.
\end{proof}

\begin{remark}\label{rk:corank-null}
The quantities $d_1^{>}(P,\c)$ and $d_1^{<}(P,\c)$ extend the notions of \emph{corank} and \emph{nullity} from matroid theory. 
Indeed, let $M$ be a matroid with ground set $[n]$, and let $P(M)\subset \{0,1\}^n$ be the corresponding polymatroid (as defined in Remark~\ref{rk:mat-are-polymat}). Let $S\subseteq [n]$, and let $\c=(c_1,\ldots,c_n)\in\{\,0,1\,\}^n$ be the vector corresponding to $S$ (that is, $c_i=1$ if $i\in S$ and $c_i=0$ otherwise). By definition, the corank of $S$ with respect to the matroid $M$ is $\min\{\,|A\setminus S|\mid A\text{ is a basis of }M\,\}$, which is clearly equal to $d_1^{>}(P(M),\c)$. 
Dually, the nullity of $S$ with respect to $M$ is $\min\{\,|S\setminus A|\mid A\text{ is a basis of }M\,\}=d_1^{<}(P(M),\c)$. 
\end{remark}

\begin{definition}
\label{def:Tutte_uv}
Let $P\subset \Z^n$ be a polymatroid. Define the formal power
series $\wti \T_P(u,v)$ in two variables $u$ and $v$ by 
$$
\wti \T_P(u,v):=\sum_{\c\in\Z^n} \wt_P(\c),
$$
where the \emph{weight} $\wt_P(\c)$ of a lattice point $\c\in\Z^n$ is given by 
$$
\wt_P(\c) := u^{d_1^{>}(P,\c)}\, v^{d_1^{<}(P,\c)}.
$$
\end{definition}


As we now prove, the formal power series $\wti \T_P(u,v)$ is 
related to the polymatroid 
Tutte polynomial $\T_P(x,y)$ by a simple change of variables.

\begin{theorem}
\label{th:tilde_Tutte}
For all polymatroids $P\subset \Z^n$, we have
$$
\wti \T_P(u,v) = \T_P\left({1\over 1-u}, {1\over 1-v}\right).
$$
\end{theorem}

Theorem~\ref{th:tilde_Tutte} is an identity for formal power series in $u,v$, and the expression ${1\over 1-u}$ stands for $\sum_{n=1}^\infty u^n$ (and similarly ${1\over 1-v}=\sum_{n=1}^\infty v^n$).

\begin{proof} Using the decomposition 
$\Z^n=\biguplus_{\a\in P} C(\a)$
from Theorem~\ref{th:decomposition_into_cones}(1), 
and applying Theorem~\ref{th:decomposition_into_cones}(2)
together with Lemma~\ref{lem:d_geq_leq},
we get
$$
\wti \T_P(u,v) = \sum_{\a\in P} \sum_{\c\in C(\a)} \wt_P(\c) =
\sum_{\a\in P} \prod_{i=1}^n m_i(\a),
$$
where
$$
m_i(\a) = 
\left\{
\begin{array}{cl}
1 & \text{if } i\not\in\Int(\a) \text{ and } i\not\in\Ext(\a),\\
1+u+u^2+ \cdots & \text{if } i\in\Int(\a) \text{ and } i\not\in\Ext(\a),\\
1+v+v^2+ \cdots & \text{if } i\not\in\Int(\a) \text{ and } i\in\Ext(\a),\\
1 +( u + u^2 + \cdots) + (v+ v^2 + \cdots) & \text{if } i\in\Int(\a) \text{ and } i\in\Ext(\a).
\end{array}
\right.
$$
If we let $x=1+u+u^2+\cdots = 1/(1-u)$
and $y=1+v+v^2+\dots = 1/(1-v)$, then
$1+(u+u^2 + \cdots) + (v+ v^2 + \cdots) = x + y - 1$.
Thus, for all $\a\in P$, 
$$
\prod_{i=1}^n m_i(\a) = 
x^{|\Int(\a)\setminus\Ext(\a)|}\,
y^{|\Ext(\a)\setminus\Int(\a)|}\,
(x+y-1)^{|\Int(\a)\cap\Ext(\a)|},
$$
and $\wti \T_P(u,v)=\T_P(x,y)$.
\end{proof}

\begin{proof}[Proof of Theorem~\ref{th:order_independence}]
The order-independence ($S_n$-invariance) of the Tutte polynomial $\T_P(x,y)$ follows
from Theorem~\ref{th:tilde_Tutte}, because the formal power series
$\wti \T_P(u,v)$ is obviously independent of the ordering of the coordinates.
\end{proof}


\section{Polymatroid shadows}
\label{sec:polymatroid_shadows}

In this section, we give two additional expressions for the polymatroid Tutte polynomial, 
which are ``halfway between'' the forms in
Definitions~\ref{def:polymatroid_Tutte_poly} and~\ref{def:Tutte_uv}.


\begin{definition}
Let $P\subset \Z^n$ be a polymatroid.
The \emph{upper shadow} of $P$ is 
the Minkowski sum $\Sup(P):=P + \Z_{\geq 0}^n\subset\Z^n$.
The \emph{lower shadow} of $P$ is $\Slow(P):=P + \Z_{\leq 0}^n\subset\Z^n$.
\end{definition}


\begin{remark}
Lower and upper shadows are related to the matroidal notions of independent sets and spanning sets, respectively.
Precisely, for a matroid $M$ with ground set $[n]$, the independent (resp., spanning) sets of $M$ correspond to elements of $\{0,1\}^n$ lying in the lower (resp., upper) shadow of the polymatroid $P(M)$. 
\end{remark}


Recall that $\e_1,\dots,\e_n$ are the standard coordinate vectors in $\R^n$.
Also let $\e_0 = 0$.

\begin{definition} \label{def:activity-shadows}
Let $S\subset\Z^n$ be an upper or lower polymatroid shadow.
For $\a\in S$, an index $i\in [n]$ is called
\emph{internally active} if there is no $j\in\{0,1,\dots,i-1\}$ 
such that $\a-\e_i+\e_j\in S$.
Let $\tInt(\a)=\tInt_S(\a)\subseteq [n]$ denote the subset of 
internally active indices with respect to $\a\in S$.

For $\a\in S$, an index $i\in [n]$ is called
\emph{externally active} if there is no $j\in\{0,1,\dots,i-1\}$ 
such that $\a+\e_i-\e_j\in S$.
Let $\tExt(\a)=\tExt_S(\a)\subseteq [n]$ 
denote the subset of 
externally active indices with respect to $\a\in S$.
\end{definition}

In fact, elements of an upper polymatroid shadow $S=\Sup(P)$ have no externally active
indices. Indeed, for any $\a\in S$ and any $i\in[n]$, we have $0<i$ and $\a+\e_i=\a+\e_i-\e_0\in S$.
Likewise, elements of a lower polymatroid shadow have no internally 
active indices. Observe also that the index $1\in[n]$ is not necessarily active for an element of a polymatroid shadow.

\begin{remark}[Shadows as projections of M-convex sets]\label{rk:Shadow-are-projections}
Polymatroid shadows are actually the images of certain infinite M-convex sets in $\Z^{n+1}$ under the projection 
$$
\pi\colon \Z^{n+1}\to\Z^n
\quad\text{given by}\quad
\pi\colon (x_0,x_1,\dots,x_n)\mapsto(x_1,\dots,x_n).
$$
Moreover the activities of the shadows correspond, via the projection $\pi$, to the activities of these $M$-convex sets.
Explicitly, $\Sup(P)=\pi(\tSup(P))$, where $\tSup(P)=\wti P+ \wti U$ with 
$\wti P=\{\,(0,a_1,\ldots,a_n)\mid(a_1,\ldots,a_n)\in P\,\}$ and 
$$\wti U=\left\{(x_0,\ldots,x_n)\biggm|\sum_{i=0}^n x_i=0 \text{ and } (x_1,\ldots,x_n)\in \Z_{\geq 0}^n\right\}.$$
Similarly, $\Slow(P)=\pi(\tSlow(P))$, where $\tSlow(P)=\wti P+ \wti L$ and $\wti L=\{\,-\x\mid\x\in \wti U\,\}$. 
It is easy to see that for any $\wti \a$ in $\tSup(P)$ (resp., $\tSlow(P)$), an index $i\in [n]$ is internally/externally active for $\wti \a$ (in the sense of Definition~\ref{def:activity-polymatroid}) if and only if $i$ is internally/externally active for the element $\pi(\wti \a)$ of the shadow $\Sup(P)$ (resp., $\Slow(P)$) in the sense of Definition~\ref{def:activity-shadows}.
\end{remark}

For $\a\in\Z^n$, let 
$$
\level(\a):=a_1+\dots+a_n\in\Z.
$$
Recall that all the elements of a polymatroid $P$ have the same level, and that $\level(P)$ denotes the level of the elements of $P$. 

\begin{theorem}
\label{th:upper_lower_shadow}
Let $P\subset \Z^n$ be a polymatroid. 
Let $\Sup(P):=P+\Z_{\geq 0}^n$ and 
$\Slow(P):=P+\Z_{\leq 0}^n$ be
the upper and lower shadows of $P$.
The polymatroid Tutte polynomial $\T_P(x,y)$ 
satisfies (and is uniquely determined by any of) the following two formulas:
\begin{align*}
\T_P\left(x,\, \frac{1}{1-v}\right) & = \sum_{\a\in \Sup(P)} 
x^{|\tInt(\a)|}\,v^{\level(\a)-\level(P)},\\
\T_P\left(\frac{1}{1-u},\, y \right) & = \sum_{\a\in \Slow(P)} 
u^{\level(P)-\level(\a)}\,
y^{|\tExt(\a)|}.
\end{align*}
\end{theorem}

\begin{remark}
Theorem~\ref{th:upper_lower_shadow} gives the polymatroid analogues of the so-called ``independent sets expansion'' and ``spanning sets expansion'' of the matroidal Tutte polynomial; see~\cite{GT}.
\end{remark}

The proof of Theorem~\ref{th:upper_lower_shadow} relies on the 
construction of certain partitions of the lattice $\Z^n$ which are described in the next section.

\section{Shadow decompositions}\label{sec:shadows}
In this section we define, for an M-convex set $P\subset\Z^n$, two partitions of the lattice $\Z^n$ which are refinements of the cone decomposition presented in Section~\ref{sec:cone-decomposition}, and we prove Theorem~\ref{th:upper_lower_shadow}. 

Recall from Theorem~\ref{th:decomposition_into_cones} the cone decomposition associated to $P$:
$$
\Z^n = \biguplus_{\a\in P} C_P(\a).
$$ 
Let $\a\in P$. The cone $C_P(\a)$ can be written as the following Minkowski sum:
$$C_P(\a)=\a+\sum_{i\in \Int_P(\a)\cap \Ext_P(\a)}\Z\,\e_i+\sum_{i\in \Int_P(\a)\setminus \Ext_P(\a)}\Z_{\leq 0}\,\e_i+\sum_{i\in \Ext_P(\a)\setminus \Int_P(\a)}\Z_{\geq 0}\,\e_i.$$
We define two subcones:
$$ \Cup_P(\a):=\a+\sum_{i\in \Ext_P(\a)}\Z_{\geq 0}\,\e_i, ~\text{ and }~ \Clow_P(\a):=\a+\sum_{i\in \Int_P(\a)}\Z_{\leq 0}\,\e_i.$$
\begin{lemma}\label{lem:partition-shadows}
The cones $\Cup_P(\a)$ and $\Clow_P(\a)$ partition the upper and lower shadows of $P$, respectively: 
$$\Sup(P)=\biguplus_{\a\in P} \Cup_P(\a),\quad\text{and}\quad\Slow(P)=\biguplus_{\a\in P} \Clow_P(\a).$$
\end{lemma}
\begin{proof}
It is clear that $\biguplus_{\a\in P} \Cup_P(\a)\subseteq \Sup(P)$. 
To prove the converse inclusion, let us consider $\b$ in $\Sup(P)\cap C_P(\a)$. As $\b\in \Sup(P)$, we get $d_1^{>}(P,\b)=0$. Since $\b\in C_P(\a)$, Theorem~\ref{th:decomposition_into_cones}(2) and Lemma~\ref{lem:d_geq_leq} combine to give $d_1^{>}(\a,\b)=0$. Hence $\b\in \Cup_P(\a)$. 
Thus, $\Sup(P)=\biguplus_{\a\in P} \Cup_P(\a)$. Similarly, $\Slow(P)=\biguplus_{\a\in P} \Clow_P(\a)$.
\end{proof}

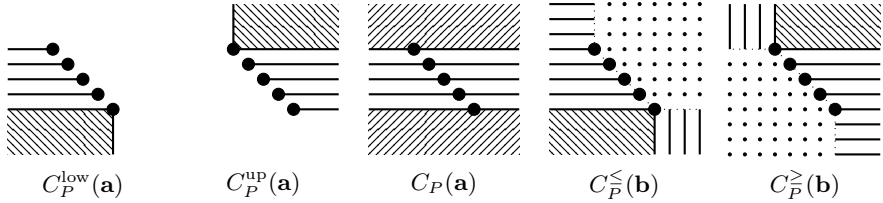
\begin{figure}[h]
\centering
\begin{tikzpicture}[scale=.2]

\begin{scope}[shift={(0,0)}]
\node at (2,-5) {\small $\Clow_P(\a)$};

\draw [fill] (0,4) circle [radius=.4];
\draw [fill] (1,3) circle [radius=.4];
\draw [fill] (2,2) circle [radius=.4];
\draw [fill] (3,1) circle [radius=.4];
\draw [fill] (4,0) circle [radius=.4];

\draw [thick] (-3,4) -- (0,4);
\draw [thick] (-3,3) -- (1,3);
\draw [thick] (-3,2) -- (2,2);
\draw [thick] (-3,1) -- (3,1);
\draw [thick] (-3,0) -- (4,0) -- (4,-3);

\path [pattern=north west lines] (-3,-3) rectangle (4,0);
\end{scope}

\node at (14,-5) {\small $\Cup_P(\a)$};

\begin{scope}[shift={(16,4)},yscale=-1,xscale=-1]

\draw [fill] (0,4) circle [radius=.4];
\draw [fill] (1,3) circle [radius=.4];
\draw [fill] (2,2) circle [radius=.4];
\draw [fill] (3,1) circle [radius=.4];
\draw [fill] (4,0) circle [radius=.4];

\draw [thick] (-3,4) -- (0,4);
\draw [thick] (-3,3) -- (1,3);
\draw [thick] (-3,2) -- (2,2);
\draw [thick] (-3,1) -- (3,1);
\draw [thick] (-3,0) -- (4,0) -- (4,-3);

\path [pattern=north west lines] (-3,-3) rectangle (4,0);
\end{scope}

\begin{scope}[shift={(24,0)}]
\node at (2,-5) {\small $C_P(\a)$};

\draw [fill] (0,4) circle [radius=.4];
\draw [fill] (1,3) circle [radius=.4];
\draw [fill] (2,2) circle [radius=.4];
\draw [fill] (3,1) circle [radius=.4];
\draw [fill] (4,0) circle [radius=.4];

\draw [thick] (-3,4) -- (7,4);
\draw [thick] (-3,3) -- (7,3);
\draw [thick] (-3,2) -- (7,2);
\draw [thick] (-3,1) -- (7,1);
\draw [thick] (-3,0) -- (7,0);

\path [pattern=north east lines] (-3,-3) rectangle (7,0);
\path [pattern=north east lines] (-3,4) rectangle (7,7);
\end{scope}


\begin{scope}[shift={(36,0)}]
\node at (2,-5) {\small $C^{\leq}_P(\b)$};

\draw [fill] (0,4) circle [radius=.4];
\draw [fill] (1,3) circle [radius=.4];
\draw [fill] (2,2) circle [radius=.4];
\draw [fill] (3,1) circle [radius=.4];
\draw [fill] (4,0) circle [radius=.4];

\draw [dotted] (0,7) -- (0,4) -- (4,0) -- (7,0);
\draw [thick] (-3,7) -- (0,7);
\draw [thick] (-3,6) -- (0,6);
\draw [thick] (-3,5) -- (0,5);
\draw [thick] (-3,4) -- (0,4);
\draw [thick] (-3,3) -- (1,3);
\draw [thick] (-3,2) -- (2,2);
\draw [thick] (-3,1) -- (3,1);
\draw [thick] (-3,0) -- (4,0) -- (4,-3);
\draw [thick] (5,0) -- (5,-3);
\draw [thick] (6,0) -- (6,-3);
\draw [thick] (7,0) -- (7,-3);

\draw [fill] (1,7) circle [radius=.1];
\draw [fill] (2,7) circle [radius=.1];
\draw [fill] (3,7) circle [radius=.1];
\draw [fill] (4,7) circle [radius=.1];
\draw [fill] (5,7) circle [radius=.1];
\draw [fill] (6,7) circle [radius=.1];
\draw [fill] (7,7) circle [radius=.1];
\draw [fill] (1,6) circle [radius=.1];
\draw [fill] (2,6) circle [radius=.1];
\draw [fill] (3,6) circle [radius=.1];
\draw [fill] (4,6) circle [radius=.1];
\draw [fill] (5,6) circle [radius=.1];
\draw [fill] (6,6) circle [radius=.1];
\draw [fill] (7,6) circle [radius=.1];
\draw [fill] (1,5) circle [radius=.1];
\draw [fill] (2,5) circle [radius=.1];
\draw [fill] (3,5) circle [radius=.1];
\draw [fill] (4,5) circle [radius=.1];
\draw [fill] (5,5) circle [radius=.1];
\draw [fill] (6,5) circle [radius=.1];
\draw [fill] (7,5) circle [radius=.1];
\draw [fill] (1,4) circle [radius=.1];
\draw [fill] (2,4) circle [radius=.1];
\draw [fill] (3,4) circle [radius=.1];
\draw [fill] (4,4) circle [radius=.1];
\draw [fill] (5,4) circle [radius=.1];
\draw [fill] (6,4) circle [radius=.1];
\draw [fill] (7,4) circle [radius=.1];
\draw [fill] (2,3) circle [radius=.1];
\draw [fill] (3,3) circle [radius=.1];
\draw [fill] (4,3) circle [radius=.1];
\draw [fill] (5,3) circle [radius=.1];
\draw [fill] (6,3) circle [radius=.1];
\draw [fill] (7,3) circle [radius=.1];
\draw [fill] (3,2) circle [radius=.1];
\draw [fill] (4,2) circle [radius=.1];
\draw [fill] (5,2) circle [radius=.1];
\draw [fill] (6,2) circle [radius=.1];
\draw [fill] (7,2) circle [radius=.1];
\draw [fill] (4,1) circle [radius=.1];
\draw [fill] (5,1) circle [radius=.1];
\draw [fill] (6,1) circle [radius=.1];
\draw [fill] (7,1) circle [radius=.1];

\path [pattern=north west lines] (-3,-3) rectangle (4,0);
\end{scope}

\node at (50,-5) {\small $C^{\geq}_P(\b)$};

\begin{scope}[shift={(52,4)},yscale=-1,xscale=-1]
\draw [fill] (0,4) circle [radius=.4];
\draw [fill] (1,3) circle [radius=.4];
\draw [fill] (2,2) circle [radius=.4];
\draw [fill] (3,1) circle [radius=.4];
\draw [fill] (4,0) circle [radius=.4];

\draw [dotted] (0,7) -- (0,4) -- (4,0) -- (7,0);
\draw [thick] (-3,7) -- (0,7);
\draw [thick] (-3,6) -- (0,6);
\draw [thick] (-3,5) -- (0,5);
\draw [thick] (-3,4) -- (0,4);
\draw [thick] (-3,3) -- (1,3);
\draw [thick] (-3,2) -- (2,2);
\draw [thick] (-3,1) -- (3,1);
\draw [thick] (-3,0) -- (4,0) -- (4,-3);
\draw [thick] (5,0) -- (5,-3);
\draw [thick] (6,0) -- (6,-3);
\draw [thick] (7,0) -- (7,-3);

\draw [fill] (1,7) circle [radius=.1];
\draw [fill] (2,7) circle [radius=.1];
\draw [fill] (3,7) circle [radius=.1];
\draw [fill] (4,7) circle [radius=.1];
\draw [fill] (5,7) circle [radius=.1];
\draw [fill] (6,7) circle [radius=.1];
\draw [fill] (7,7) circle [radius=.1];
\draw [fill] (1,6) circle [radius=.1];
\draw [fill] (2,6) circle [radius=.1];
\draw [fill] (3,6) circle [radius=.1];
\draw [fill] (4,6) circle [radius=.1];
\draw [fill] (5,6) circle [radius=.1];
\draw [fill] (6,6) circle [radius=.1];
\draw [fill] (7,6) circle [radius=.1];
\draw [fill] (1,5) circle [radius=.1];
\draw [fill] (2,5) circle [radius=.1];
\draw [fill] (3,5) circle [radius=.1];
\draw [fill] (4,5) circle [radius=.1];
\draw [fill] (5,5) circle [radius=.1];
\draw [fill] (6,5) circle [radius=.1];
\draw [fill] (7,5) circle [radius=.1];
\draw [fill] (1,4) circle [radius=.1];
\draw [fill] (2,4) circle [radius=.1];
\draw [fill] (3,4) circle [radius=.1];
\draw [fill] (4,4) circle [radius=.1];
\draw [fill] (5,4) circle [radius=.1];
\draw [fill] (6,4) circle [radius=.1];
\draw [fill] (7,4) circle [radius=.1];
\draw [fill] (2,3) circle [radius=.1];
\draw [fill] (3,3) circle [radius=.1];
\draw [fill] (4,3) circle [radius=.1];
\draw [fill] (5,3) circle [radius=.1];
\draw [fill] (6,3) circle [radius=.1];
\draw [fill] (7,3) circle [radius=.1];
\draw [fill] (3,2) circle [radius=.1];
\draw [fill] (4,2) circle [radius=.1];
\draw [fill] (5,2) circle [radius=.1];
\draw [fill] (6,2) circle [radius=.1];
\draw [fill] (7,2) circle [radius=.1];
\draw [fill] (4,1) circle [radius=.1];
\draw [fill] (5,1) circle [radius=.1];
\draw [fill] (6,1) circle [radius=.1];
\draw [fill] (7,1) circle [radius=.1];

\path [pattern=north west lines] (-3,-3) rectangle (4,0);
\end{scope}

\end{tikzpicture}
\caption{Various decompositions into cones for a polymatroid $P\subset \Z^2$. The elements of $P$ are indicated by large dots. 
The first panel shows the cone decomposition $\Slow(P)=\biguplus_{\a\in P}\Clow_P(\a)$ of the lower shadow. The second panel depicts the decomposition $\Sup(P)=\biguplus_{\a\in P}\Cup_P(\a)$ of the upper shadow. The third panel shows the decomposition $\Z^n=\biguplus_{\a\in P} C_P(\a)$ provided by Theorem~\ref{th:decomposition_into_cones}. The fourth (resp., fifth) panel is a depiction of the refined decomposition of $\Z^n$, given by Theorem~\ref{th:shadow_decomposition}(1), in terms of cones indexed by elements $\b$ of $\Sup(P)$ (resp., $\Slow(P)$). 
}
\label{fig:decompositions}
\end{figure}

The partitions of the shadows given by Lemma~\ref{lem:partition-shadows} are illustrated in the first two panels of Figure~\ref{fig:decompositions}.
We now consider some partitions of $\Z^n$ indexed by the points in the shadows.
For $\b\in\Sup(P)$ we define 
$$C^{\leq}_P(\b)= \b+\sum_{i \in \tInt(\b)}\Z_{\leq 0}\, \e_i.$$
Now we observe that for all $\b\in \Cup_P(\a)$, we have
$$\tInt(\b)=\Int_P(\a)\cap\{\,i\in[n]\mid a_i=b_i\,\}.$$ 
This implies 
\begin{equation}\label{eq:partition-cone+}
C_P(\a)=\biguplus_{\b\in \Cup_P(\a)} C^{\leq}_P(\b).
\end{equation}
Similarly, for $\b\in\Slow(P)$, we define $\ds C^{\geq}_P(\b)= \b+\sum_{i\in \tExt(\b)}\Z_{\geq 0}\, \e_i$ and get
\begin{equation}\label{eq:partition-cone-}
C_P(\a)=\biguplus_{\b\in \Clow_P(\a)} C^{\geq}_P(\b).
\end{equation}

We summarize our constructions:

\begin{theorem} \label{th:shadow_decomposition}
Let $P\subset\Z^n$ be an M-convex set. 
\begin{compactenum}
\item The cone decomposition of $\Z^n$ given by Theorem~\ref{th:decomposition_into_cones} has the following two refinements:
$$\ds \Z^n= \biguplus_{\a\in\Sup(P)} C^{\leq}_{P}(\b)=\biguplus_{\b\in\Slow(P)} C^{\geq}_{P}(\b).$$
\item Let $\a\in P$, let $\b\in \Cup_P(\a)$ and let $\b'\in \Clow_P(\a)$. 
If for all $i\in\Int_P(\a)\cap Ext_P(\a)$ either $a_i=b_i$ or $a_i=b'_i$ (or both), then $C^{\leq}_P(\b)\cap C^{\geq}_P(\b')=\{\b+\b'-\a\}$. Otherwise $C^{\leq}_P(\b)\cap C^{\geq}_P(\b')=\varnothing$.
\item For all $\b\in \Sup(P)$ and all $\c\in C^\leq_P(\b)$, one has 
$$d_1^{<}(P,\c)=d_1^{<}(P,\b)=\level(\b)-\level(P)~\text{ and }~d_1^{>}(P,\c)=d_1^{>}(\b,\c).$$
Symmetrically, for all $\b\in \Slow(P)$ and all $\c\in C^\geq_P(\b)$, one has 
$$d_1^{>}(P,\c)=d_1^{>}(P,\b)=\level(P)-\level(\b)~\text{ and }~d_1^{<}(P,\c)=d_1^{<}(\b,\c).$$
\end{compactenum}
\end{theorem}

The decompositions of $\Z^n$ given by Theorem~\ref{th:shadow_decomposition}(1) are illustrated in the last two panels of Figure~\ref{fig:decompositions}.

\begin{proof}
Property (1) is obtained by combining Lemma~\ref{lem:partition-shadows} with~\eqref{eq:partition-cone+} and~\eqref{eq:partition-cone-}. 

For Property (2), let us consider $\c\in C^{\leq}_P(\b)\cap C^{\geq}_P(\b')$. We observe that for all $i\notin \Int_P(\a)$ we have $i\notin \tInt(\b)$ whence $c_i=b_i$. Symmetrically, for all $i\notin \Ext_P(\a)$ we have $c_i=b_i'$. Lastly for $i\in \Int_P(\a)\cap \Ext_P(\a)$, if $b_i\neq a_i$ then $c_i=b_i>a_i$, and if $b_i'\neq a_i$ then $c_i=b_i'<a_i$. This easily implies Property (2). 

We now prove Property (3). Let $\b\in \Sup(P)$ and let $\c\in C^\leq_P(\b)$. Let $\a$ be the element of $P$ such that $\b\in C_P(\a)$. Theorem~\ref{th:decomposition_into_cones}(2) together with Lemma~\ref{lem:d_geq_leq} gives $d_1^{<}(P,\c)=d_1^{<}(\a,\c)$ and $d_1^{>}(P,\c)=d_1^{>}(\a,\c)$. Moreover, it is clear from the definitions that $d_1^{<}(\a,\c)=d_1^{<}(\a,\b)=\level(\b)-\level(\a)$ and $d_1^{>}(\a,\c)=d_1^{>}(\b,\c)$. This proves the first identity of Property (3), and the other identity is proved similarly.
\end{proof}

\begin{remark}[Shadow partitions as projections]\label{rk:Shadowpartitions-are-projections}
Recall from Remark~\ref{rk:Shadow-are-projections} that the shadows $\Sup(P)$ and $\Slow(P)$ are the projections of some M-convex sets $\tSup(P)$ and $\tSlow(P)$. It is not hard to check that the decomposition of $\Z^n$ given by Theorem~\ref{th:shadow_decomposition}(1) is the projection of the cone decompositions of $\Z^{n+1}$ obtained by applying Theorem~\ref{th:decomposition_into_cones} to the M-convex sets $\tSup(P)$ and $\tSlow(P)$.
\end{remark}

We now use Theorem~\ref{th:shadow_decomposition} to prove Theorem~\ref{th:upper_lower_shadow}. 

\begin{proof}[Proof of Theorem~\ref{th:upper_lower_shadow}] 
Using Theorem~\ref{th:tilde_Tutte} and Theorem~\ref{th:shadow_decomposition}(1), we obtain 
$$\T_P\left(\frac1{1-u},\frac{1}{1-v}\right) 
= \sum_{\b\in \Sup(P)}~\sum_{\c\in C^{\leq}_{P}(\b)} u^{d_1^{>}(P,\c)}\, v^{d_1^{<}(P,\c)}.$$
Using Theorem~\ref{th:shadow_decomposition}(3) then gives
\begin{eqnarray*}
\T_P\left(\frac1{1-u},\frac{1}{1-v}\right)&=& \sum_{\b\in \Sup(P)} v^{\level(\b)-\level(P)}\sum_{\c\in C^{\leq}_{P}(\b)}u^{d_1^{>}(\b,\c)}\\
&=&\sum_{\b\in \Sup(P)} v^{\level(\b)-\level(P)}\left(\frac{1}{1-u}\right)^{|\tInt(\b)|}.
\end{eqnarray*}
Replacing $\frac{1}{1-u}$ by $x$ gives the first needed formula. The proof of the second formula is similar.
\end{proof}

\section{Cameron-Fink polynomials and polymatroid neighborhoods} 
\label{sec:Cameron-Fink}

In this section we explain the relation between the polymatroid Tutte polynomial and the invariant defined by Cameron and Fink in~\cite{CF}. Let us first recall their polymatroid invariant $\QCF'_P(x,y)$.

\begin{definition}~\cite{CF} \ 
Let $P\subset\Z^n$ be a polymatroid and let $\mP=\conv(P)$ be its convex hull. 
Let $\Delta$ be the coordinate $(n-1)$-dimensional simplex 
$\Delta:=\conv(\e_1,\dots,\e_n)\subset\R^n$, and let $\nabla = -\Delta = 
\conv(-\e_1,\dots,-\e_n)$.
It is shown in~\cite{CF} that there exists a (unique) two-variable polynomial $\QCF_P$ such that 
for any positive integer values of $s$ and $t$, the value $\QCF_P(s,t)$ counts
the number of integer lattice points in the Minkowski sum
$\mP+s\nabla + t\Delta$:
$$
\QCF_P(s,t)= |(\mP+s\nabla + t\Delta)\cap \Z^n|.
$$
There are (unique) integer coefficients $c_{i,j}$ such that 
$$
\QCF_P(s,t) = \sum_{i,j \geq 0} c_{i,j}\,{s \choose i} {t \choose j},
$$
where ${x \choose i}:=\frac{x(x-1)\cdots(x-i+1)}{i!}$, and 
$\QCF'_P(x,y)$ is defined as the polynomial
$$
\QCF'_P(x,y) := \sum_{i,j\geq 0} c_{i,j} (x-1)^i\, (y-1)^j.
$$
\end{definition} 

The relation between $\QCF'_P(x,y)$ and $\T_P(x,y)$ takes the following simple form.

\begin{theorem}
\label{th:CF_Tutte}
The polynomial $\QCF'_P(x,y)$ equals $\T_P(x,y)/(x+y-1)$.
\end{theorem}

\begin{remark} In view of Theorem~\ref{th:CF_Tutte}, Proposition~\ref{prop:TP_properties}(h) can be rephrased as follows:
$$\QCF_P(0,-1)= |(\mP-\Delta)\cap \Z^n|\quad\text{and}\quad\QCF_P(-1,0)=|(\mP-\nabla)\cap \Z^n|.$$
\end{remark}
The rest of this section is dedicated to the proof of Theorem~\ref{th:CF_Tutte}.

\begin{definition}
For a polymatroid $P\subset \Z^n$ and non-negative integers $s,t\in\Z_{\geq 0}$,
the \emph{$(s,t)$-neighborhood} of $P$ is the set
$$
N(s,t,P):=\{\,\c\in\Z^n\mid d_1^{>}(P,\c)\leq s\text{ and } d_1^{<}(P,\c)\leq t\,\}.
$$
\end{definition}

Let $H_\ell\subset\Z^n$ be the affine hyperplane 
$
H_\ell:=\{\,\c\in\Z^n\mid \level(\c)=\ell\,\}$.

\begin{lemma} 
For a polymatroid $P\subset \Z^n$ and non-negative integers 
$s,t\in\Z_{\geq 0}$, 
the set of lattice points in the Minkowski sum $\mP+s\nabla + t\Delta$
is a hyperplane section of 
the $(s,t)$-neighborhood of $P$, namely
$$
(\mP+s\nabla + t\Delta)\cap \Z^n = 
N(s,t,P) \cap H_{\level(P)-s+t}.
$$
\end{lemma}


\begin{proof}
For any lattice point $\c\in (\mP+s\nabla + t\Delta)\cap \Z^n$,
we have $\level(\c) = \level(P)-s+t$.
We also have 
$\c=\a+\a'+\a''$, where $\a\in P$, $\a'\in s\nabla$, and $\a''\in t\Delta$.
So $d_1^{>}(P,\c) \leq d_1^{>}(\a,\c)\leq \sum_{i=1}^n |a_i'|= s$ and
$d_1^{<}(P,\c)\leq d_1^{<}(\a,\c) \leq \sum_{i=1}^n |a_i''|= t$. 
Thus $\c\in N(s,t,P)$.

On the other hand, for any $\c\in N(s,t,P)$ with $\level(\c)=\level(P)-s+t$, let
$\a\in P$ be any closest point of $P$ to $\c$, i.e., $d_1(P,\c)= d_1(\a,\c)$.
According to Lemma~\ref{lem:d_geq_leq}, we have
$d_1^{>}(\a,\c)=d_1^{>}(P,\c) \leq s$ and 
$d_1^{<}(\a,\c)=d_1^{<}(P,\c) \leq t$.
So $\c =\a + \a'+\a''$, where $\a'\in s'\nabla$ and $\a''\in t'\Delta$ with 
$s':= d_1^{>}(P,\c) \leq s$ and 
$t':= d_1^{<}(P,\c) \leq t$. 
Pick any $\a'''\in r\Delta$, where $r:=s-s'=t-t'$. 
Then we have $\c = \a + (\a'-\a''') + (\a'' + \a''')$, where $\a'-\a'''\in s\nabla$
and $\a''+\a'''\in t\Delta$. So $\c\in\mP+s\nabla + t\Delta$.
\end{proof}

\begin{proof}[Proof of Theorem~\ref{th:CF_Tutte}]
We have 
\begin{multline*}
\QCF_P(s,t)=|(\mP+s\nabla+t\Delta)\cap \Z^n| =
|N(s,t,P)\cap H_{\level(P)-s+t}|\\
=\sum_{\a\in P} |C_P(\a)\cap N(s,t,P)\cap H_{\level(P)-s+t}|.
\end{multline*}
The intersection 
$C_P(\a)\cap N(s,t,P)\cap H_{\level(P)-s+t}$ 
consists of all points $\a+\x\in\Z^n$ such that 
\begin{enumerate}
\item If $x_i<0$ then $i\in \Int_P(\a)$. 
If $x_j>0$ then $j\in \Ext_P(\a)$.
\item $\sum_{i\,:\,x_i <0} (-x_i) \leq s$ and 
$\sum_{i\,:\,x_i >0} x_i \leq t$.
\item $\sum_{i=1}^n x_i = -s + t$.
\end{enumerate}

Since we always have $1\in\Int(\a)\cap\Ext(\a)$, the above set
is in bijection with points $(x_2,\dots,x_n)\in\Z^{n-1}$ that 
satisfy only conditions (1) and (2).

For $\a+\x\in C_P(\a)\cap N(s,t,P)\cap H_{\level(P)-s+t} $, let $I(\x):=\{\,i\in [n]\setminus \{1\} \mid x_i <0\,\}$ and 
$J(\x):=\{\,j\in [n]\setminus \{1\} \mid x_j >0\,\}$. The pair $(I(\x),J(\x))$ belongs to the set $\mK(\a)$ of pairs $(I,J)$ of subsets of $[n]\setminus \{1\}$ such that $I\subset\Int(\a)$, $J\subset\Ext(\a)$, and $I\cap J = \varnothing$. 
Moreover for any $(I,J)\in\mK(\a)$ the number of points $\a+\x\in C_P(\a)\cap N(s,t,P)\cap H_{\level(P)-s+t}$ such that $I(\x)=I$ and $J(\x)=J$ is $ {s\choose |I|}\, {t\choose |J|}$ (because for $k\geq 0$ the number of tuples $(s_1,\ldots,s_k)\in Z_{>0}^k$ such that $\sum_{i=1}^k s_i\leq s$ is ${s\choose k}$). 
This gives
$$
\QCF_P(s,t) = \sum_{\a\in P} ~~\sum_{(I,J)\in\mK(\a)}
{s\choose |I|}\, {t\choose |J|}.
$$
Thus,
$$
\QCF_P'(x+1,y+1) = 
\sum_{\a\in P}~~ \sum_{(I,J)\in \mK(\a)} x^{|I|}\, y^{|J|}
=\sum_{\a\in P} (x+1)^{\oi(\a)}(y+1)^{\oe(\a)}(x+y+1)^{\ie(\a)-1},
$$
which is exactly $\T_P(x+1,y+1)/(x+y+1)$.
\end{proof}

\section{Background: Hypergraphical polymatroids}
\label{sec:hypergraphs}
In this section and the following ones, we discuss the polymatroids associated to hypergraphs \cite{Hel,tuttebook}. Some properties of the associated generalized permutohedra have been studied in~\cite{Pos,Kal, KP}.  
This class of polytopes extends two well known families of ``graphical polytopes'': base polytopes of graphical matroids and graphical zonotopes.


\subsection{Hypergraphs}
Roughly speaking, a \emph{hypergraph} over a set $V$ of ``vertices'' is a multiset of subsets of $V$ called ``hyperedges''. More formally, a \emph{hypergraph} with \emph{vertex set} $V$ and \emph{hyperedge set} $W$ is a function $\vec H$ from the set $W$ to the set $2^V$ of subsets of $V$.
Clearly, undirected graphs without loops identify with hypergraphs such that every hyperedge has cardinality $2$.
Next we recall that hypergraphs can be identified with bipartite graphs.

A \emph{bipartite graph} is a simple graph where every vertex is either called \emph{left vertex} or \emph{right vertex}, and every edge joins a left vertex to a right vertex. We adopt the notation $H = (V\sqcup W, E)$ to indicate that $H$ is a bipartite graph, $V$ is the set of left vertices, $W$ is the set of right vertices, and $E$ is the set of edges (note that our bipartite graphs come with a fixed bipartition of the vertices). To the bipartite graph $H$ we associate the hypergraph $\vec H$ with vertex set $V$ and hyperedge set $W$ defined by setting $\vec H(w)$ to be the set of neighbors of $w$ in $H$. Clearly, this is a bijection between bipartite graphs and hypergraphs.

The bipartite graph setting makes the symmetric role between vertices and hyperedges more apparent. For a bipartite graph $H= (V\sqcup W, E)$, we define the \emph{transpose bipartite graph} to be $H^\mir:=(W\sqcup V, E)$. We also say that the hypergraph $\vec H^\mir$ (associated to $H^\mir$) is the \emph{transpose} of $\vec H$.

\begin{example}
\label{ex:hypergraph}
\begin{figure}[h]
\centering
\begin{tikzpicture}[scale=.35]

\draw [thick] (12,4.5) -- (4,7.5);
\draw [thick] (12,4.5) -- (4,4.5);
\draw [thick] (12,4.5) -- (4,1.5);
\draw [thick] (12,1.5) -- (4,4.5);
\draw [thick] (12,1.5) -- (4,1.5);
\draw [thick] (12,1.5) -- (4,-1.5);
\draw [thick] (12,1.5) -- (4,-4.5);
\draw [thick] (12,-1.5) -- (4,7.5);
\draw [thick] (12,-1.5) -- (4,4.5);
\draw [thick] (12,-1.5) -- (4,1.5);
\draw [thick] (12,-1.5) -- (4,-1.5);
\draw [thick] (12,-1.5) -- (4,-4.5);

\draw [thick,fill=white] (12,4.5) circle [radius=.4];
\draw [thick,fill=white] (12,1.5) circle [radius=.4];
\draw [thick,fill=white] (12,-1.5) circle [radius=.4];

\draw [thick,fill=white] (4,7.5) circle [radius=.4];
\draw [thick,fill=white] (4,4.5) circle [radius=.4];
\draw [thick,fill=white] (4,1.5) circle [radius=.4];
\draw [thick,fill=white] (4,-1.5) circle [radius=.4];
\draw [thick,fill=white] (4,-4.5) circle [radius=.4];
\node at (4,7.5) {\tiny $
a$};
\node at (4,4.5) {\tiny $
b$};
\node at (4,1.5) {\tiny $
c$};
\node at (4,-1.5) {\tiny $
d$};
\node at (4,-4.5) {\tiny $
e$};

\node at (12,4.5) {\tiny $1$};
\node at (12,1.5) {\tiny $2$};
\node at (12,-1.5) {\tiny $3$};
\node at (1,1.5) {\small $V$};
\node at (15,1.5) {\small $W$};

\begin{scope}[shift={(25,0)}]
\draw [->] (0,-1) -- (0,7.2);
\draw [->] (0,-1) -- (-7,-4.5);
\draw [->] (0,-1) -- (7.2,-2.2);
\node at (-7.6,-4.8) {\small $1$};
\node at (8.1,-2.35) {\small $2$};
\node at (0,8) {\small $3$};
\draw [help lines] (0,6) -- (-6,-4) -- (6,-2) -- cycle;
\draw [fill=lightgray,lightgray] (0,-3) circle [radius=.3];
\draw [fill=lightgray,lightgray] (3,-2.5) circle [radius=.3];
\draw [fill=lightgray,lightgray] (-1.5,-1) circle [radius=.3];
\draw [fill=lightgray,lightgray] (1.5,-.5) circle [radius=.3];
\draw [fill=lightgray,lightgray] (4.5,0) circle [radius=.3];
\draw [fill=lightgray,lightgray] (-3,1) circle [radius=.3];
\draw [fill=white,white] (0,1.5) circle [radius=.4];
\draw [fill=lightgray,lightgray] (0,1.5) circle [radius=.3];
\draw [fill=lightgray,lightgray] (3,2) circle [radius=.3];
\draw [fill=lightgray,lightgray] (-1.5,3.5) circle [radius=.3];
\draw [fill=lightgray,lightgray] (1.5,4) circle [radius=.3];
\draw [fill=lightgray,lightgray] (0,6) circle [radius=.3];

\end{scope}

\end{tikzpicture}
\caption{A bipartite graph $H$ and its hypergraphical polymatroid $P_H$ (which is the same polymatroid as the one in Example \ref{ex:polymatroid}).}
\label{fig:hypergraph}
\end{figure}
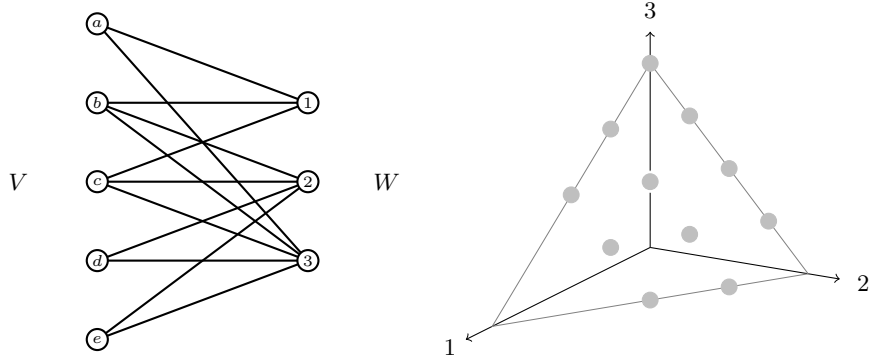

Let $H$ be the bipartite graph shown in Figure~\ref{fig:hypergraph} whose
left and right vertex sets are $V=\{a,b,c,d,e\}$ and $W=\{1,2,3\}$ respectively.
The hypergraph $\vec H$ has vertex set $V$, hyperedge set $W$ and is given by $\vec H(1)=\{a,b,c\}$, $\vec H(2)=\{b,c,d,e\}$ and $\vec H(3)=\{a,b,c,d,e\}$.
The transpose hypergraph $\vec H^\mir$ has the vertex set $W$, hyperedge set $V$ and is given by $\vec H^\mir(a)=\{1,3\}$, $\vec H^\mir(b)=\vec H^\mir(c)=\{1,2,3\}$ and $\vec H^\mir(d)=\vec H^\mir(e)=\{2,3\}$.
\end{example}

In the next subsections we will define the \emph{hypergraphical polymatroid} $P_H\subset \Z^W$ associated to the bipartite graph $H= (V\sqcup W, E)$.
There are two different (but equivalent) constructions of $P_H$: one in terms of the hypergraph $\vec H$ (via spanning hypertrees) and one in terms of the transpose hypergraph $\vec H^\mir$ (via Minkowski sums of simplices).

\subsection{Hypergraphical polymatroid in terms of hypertrees}
\label{sec:PH-hypertrees}

We generalize the notion of spanning trees from graphs to hypergraphs as follows.

\begin{definition}\cite[Section~12]{Pos}, \cite{Kal, KP}.
Let $H = (V\sqcup W, E)$ be a connected bipartite graph. 
A \emph{spanning hypertree} of the hypergraph $\vec H$ is 
a non-negative integer vector $\a\in\Z^W_{\geq 0}$ for which there exists a usual spanning
tree $T$ of the bipartite graph $H$ such that $a_w = \deg_T(w)-1$ for every right vertex $w\in W$.
Here $\deg_T(w)$ denotes the degree of the vertex $w$ in the tree $T$.
\end{definition}

\begin{remark}\label{rk:hypergraph-fromG}
Let $G=(V,E)$ be a graph. Let $\vec H(G)$ be the corresponding hyper\-graph, which is to say, $\vec H(G)$ is the hypergraph associated to the bipartite graph $H(G)=(V\sqcup W,E')$ obtained from $G$ by adding a right vertex $w_e$ in the middle of each edge $e$ of $G$ (see Figure~\ref{fig:hypergraph-fromG}).
Then the spanning hypertrees of $\vec H(G)$ are in bijection with the spanning trees of $G$ (the bijection associates to a spanning tree $T\subseteq E$ of $G$ the hypertree in $\{0,1\}^W$ given by the indicator function of the set $T$).
\end{remark}

\begin{figure}[!ht]
\begin{center}\includegraphics[width=.6\linewidth]{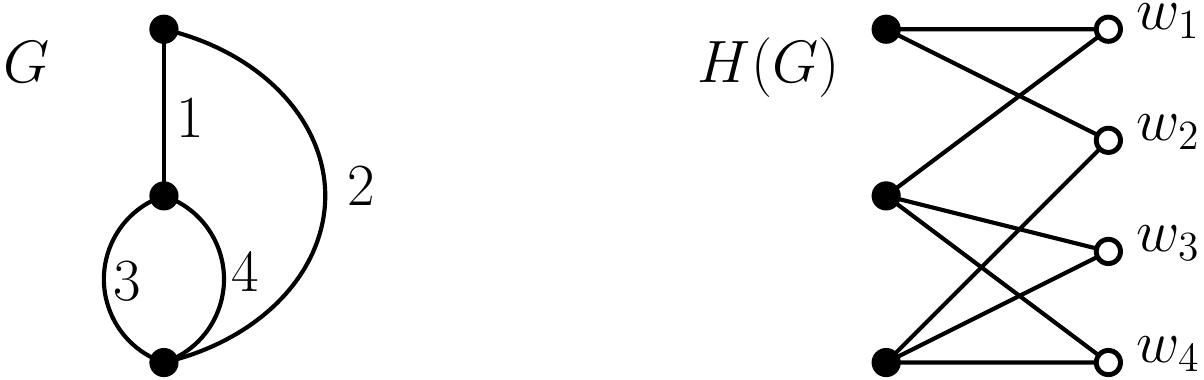}\end{center}
\caption{A graph $G$ and the associated hypergraph $\vec H(G)$. The 
latter has five spanning hypertrees: $(1,1,0,0)$, $(1,0,1,0)$, $(1,0,0,1)$, $(0,1,1,0)$ and $(0,1,0,1)$.}
\label{fig:hypergraph-fromG}
\end{figure}

The notion of spanning hypertree first appeared in~\cite{Pos} under the
name of \emph{right degree vectors} of $H$, and also as \emph{$H$-draconian
sequences}. 
The term \emph{hypertree} was used
in~\cite{Kal, KP}. 
In \cite[Proposition~5.4]{Pos} several other equivalent ways to characterize spanning hypertrees are given. One characterization is in terms of the \emph{dragon marriage problem}~\cite{Pos} --- a variant of Hall's marriage problem that involves brides, grooms
and one dragon. Another characterization from \cite{Pos}  is given below. 

\begin{proposition}~\cite{Pos}
\label{prop:hypertree_ineq}
Let $H = (V\sqcup W, E)$ be a connected bipartite graph.
The vector $\a\in\Z^W$ is a spanning hypertree of the hypergraph $\vec H$ if and only if $\a$ satisfies 
the following conditions:
\begin{compactenum}
\item[(i)] $\sum_{w\in W} a_w = |V|-1$, and
\item[(ii)] for all $S\subset W$ we have $\ds \sum_{w\in S} a_w \leq \left|\bigcup_{w\in S} \vec H(w)\right|-1$. 
\end{compactenum}
\end{proposition}

\begin{remark}
\label{rmk:complete_dragon}
We point out that (by Proposition~\ref{prop:hypertree_ineq}) the \emph{$n$-draconian sequences} of Definition~\ref{def:dragon-poly} coincide with the spanning hypertrees of the \emph{complete hypergraph} $\vec H_n$ having vertex set $[n]$ and one hyperedge for each non-empty subset of $[n]$. 
\end{remark}

\begin{remark}\label{rk:Pos-to-Kal}
It is easy to see that the right-hand side in Condition (ii) of Proposition \ref{prop:hypertree_ineq} can equivalently be replaced with 
\begin{equation}\label{eq:rank-PH}
f(S):=\left|\bigcup_{w\in S} \vec H(w)\right|-\widetilde{c}(S),
\end{equation}
where $\widetilde{c}(S)$ is the number of connected components of the subgraph of $H$ induced by $S\uplus \bigcup_{w\in S} \vec H(w)$. 
Moreover, it is not hard to check that $f\colon2^W\to \Z$ is submodular \cite[Proposition~4.7]{Kal}. 
Thus the set of spanning hypertrees of the hypergraph $\vec H$ is a polymatroid with rank function $f$. 

\OB{I added the following:}
Moreover it is clear that 
\begin{equation}\label{eq:rank-PH2}
f(S)=|V|-c(S),
\end{equation}
where $c(S)$ is the number of connected components of the subgraph of $H$ induced by $S\uplus V$ (because the vertices in $V\setminus \bigcup_{w\in S} \vec H(w)$ are isolated in this induced subgraph). The expression \eqref{eq:rank-PH2} of the rank function $f$ of $P_H$ coincides with the definition given in \cite[Section~1.7]{Hel}.
\end{remark}

\begin{definition}\cite{Pos,Kal,Hel} \label{def:PH-hypertrees}
Let $H = (V\sqcup W, E)$ be a connected bipartite graph.
We define the \emph{hypergraphical polymatroid} $P_H\subset\Z^W$ as the set of all spanning hypertrees of the hypergraph $\vec H$. \OB{I added the following:}
It has rank function given by \eqref{eq:rank-PH}, or equivalently \eqref{eq:rank-PH2}.
\end{definition}


Note that, for notational convenience, our hypergraphical polymatroids are defined to be subsets of $\Z^W$ instead of $\Z^n$ for $n=|W|$. From now on we adopt this relaxed definition of polymatroids. 

\begin{remark} Observe from Remark \ref{rk:hypergraph-fromG} that for the hypergraph $\vec H(G)$ corresponding to a graph $G$, the  polymatroid $P_{H(G)}$ is the base polytope of $G$ (which is the polymatroid associated with the graphical matroid of $G$). 
\end{remark}

\subsection{Hypergraphical polymatroid as a Minkowski sum of simplices}\label{sec:hypergraph-as-Minkowski}

Recall that, for two subsets $A,B\subset\R^W$, their 
\emph{Minkowski sum} is $A+B:=\{\,\a+\b\mid \a\in A,\, \b\in B\,\}$
and their \emph{Minkowski difference} is $A-B:=\{\,\c\in\R^W \mid \{\c\}+B \subseteq A\,\}$.

\begin{definition}
\cite[Definition~11.2]{Pos}
Define the polytope $\mP_H\subset \R^W$ in terms of the hyperedges $\vec H^\mir(v)$ of the transpose hypergraph $\vec H^\mir$ as the following Minkowski sum and difference of polytopes:
$$
\mP_H:= \left(\sum_{v\in V} \Delta_{\vec H^\mir(v)}\right) - \Delta_W \subset \R^W,
$$
where the polytopes $\Delta_I=\conv(\e_i\mid i\in I)\subset \R^W$ are the coordinate simplices associated with subsets $I\subseteq W$.
\end{definition}

The two constructions above are actually equivalent, due to~\cite[Theorem~12.9]{Pos} (which is a consequence of Proposition~\ref{prop:hypertree_ineq}).

\begin{proposition}~\cite{Pos}
Let $H=(V\sqcup W, E)$ be a connected bipartite graph.
Then the set $P_H \subset\Z^W$ of spanning hypertrees of
the hypergraph $\vec H$ is a polymatroid,
and the polytope $\mP_H\subset\R^W$ defined as a Minkowski sum 
over the hyperedges of the transpose hypergraph $\vec H^\mir$ is the generalized permutohedron
associated with 
$P_H$.

In other words, $P_H$ 
is the set of lattice points of $\mP_H$, and $\mP_H = \conv(P_H)$.
\end{proposition}

\begin{example}
For the  bipartite graph $H$ of Example~\ref{ex:hypergraph}, the polymatroid $P_H\subset\Z^3$ has $11$ elements as shown in Figure~\ref{fig:hypergraph}. It coincides with the polymatroid of Example~\ref{ex:polymatroid} and it can be described either as the $11$ spanning hypertrees of the right hypergraph, or, equivalently, as the $11$ lattice points of
the Minkowski sum of simplices
$$
\mP_H = \Delta_{13} + 2 \Delta_{23} + (2-1)\Delta_{123}.
$$
\end{example}


Suppose that the bipartite graph $H=(V\sqcup W, E)$ has a left vertex $v_0\in V$ adjacent to every right vertex. Then the polytope $\mP_H$ can be expressed as
a Minkowski sum of simplices (without using a Minkowski difference): 
$$
\mP_H:= 
\sum_{v\in V,\, v\ne v_0} \Delta_{\vec H^\mir(v)}.
$$
This follows from the identity $(A+B)-B=A$ for Minkowski sums and differences.
(A little care is needed when using Minkowski differences, 
because it is not true in general that $(A-B)+B = A$.) 
The Tutte polynomial of this type of polymatroid will be studied in Section~\ref{sec:Tutte-hypergraph}.
As we now explain, graphical zonotopes (studied in Section~\ref{sec:zonotopes}) are an even more special case.
 
\OB{Added the following.}
\begin{remark}\label{rk:zonotope-is-Hmirror}
Let $G=(W,E)$ be a graph and let $H_G^+=(V^+\sqcup W,E')$ be the bipartite graph obtained from $G$ by adding a left vertex $v_e\in V$ in the middle of each edge $e\in E$ and adding an additional left vertex $v_0$ adjacent to every right vertex. This is represented in Figure \ref{fig:zonotope}.  Then, $\mP_{H_G^+}$ is exactly the \emph{graphical zonotope} $\mZ(G)$ of the graph $G$:
$$ \mP_{H_G^+} = \mZ(G) = \sum_{\{i,j\} \text{  edge of } G} [\e_i, \e_j].$$
\end{remark}

\begin{figure}[!ht]
\begin{center}\includegraphics[width=.9\linewidth]{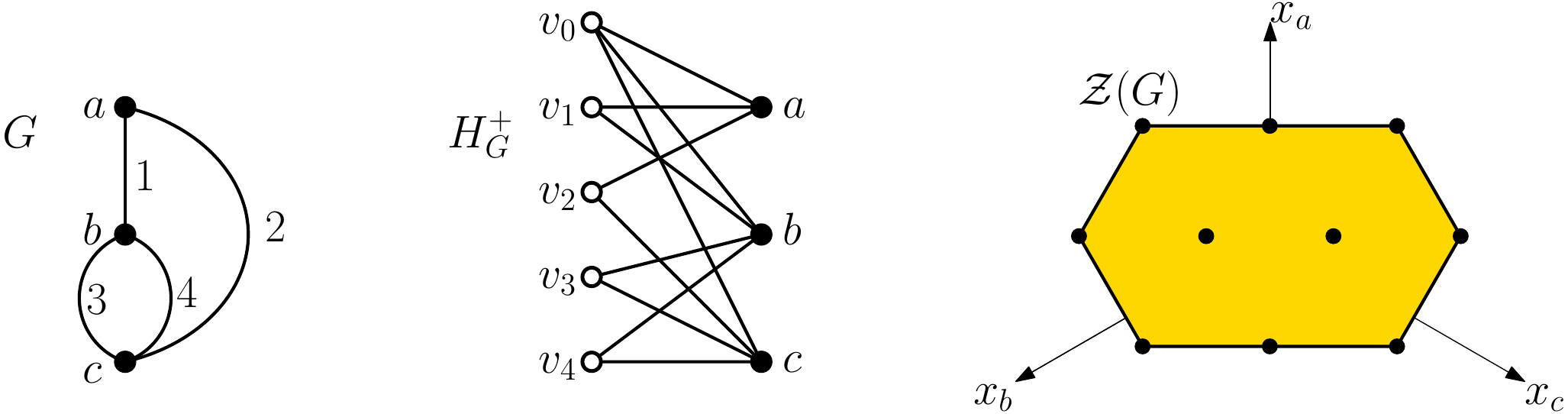}\end{center}
\caption{Left: a graph $G$. Middle: the bipartite graph $H_G^+$. Right: the zonotope $\mZ(G)=P_{H_G^+}=[\e_a,\e_b]+[\e_a,\e_c]+2[\e_b,\e_c]$.}
\label{fig:zonotope}
\end{figure}

\subsection{Direct and aggregated sums of hypergraphs}

The notion of hypergraphical polymatroid can be extended to \emph{non-connected} bipartite graphs $H = (V\sqcup W, E)$.
For this, in the definition of $P_H$ in terms of spanning hypertrees, one just needs to replace the notion of spanning tree used in Definition~\ref{def:PH-hypertrees} by the notion of maximal spanning forest of $H$. Equivalently, one can consider the connected components $H_1$,\ldots, $H_k$ of $H$ and define $P_H$ as the direct sum of polymatroids $P_{H_1}\oplus \cdots \oplus P_{H_k}$. Note that by Proposition~\ref{prop:TP_properties}(e), this gives
$$\T_{P_H}(x,y)=\T_{P_{H_1}}(x,y)\times \cdots \times \T_{P_{H_k}}(x,y).$$

We now describe a construction on hypergraphs corresponding to the \emph{aggregated sums} of polymatroids (as defined just before Proposition~\ref{prop:TP_properties}). Let $H_1=(V_1\sqcup W_1,E_1)$ and $H_2=(V_2\sqcup W_2,E_2)$ be (disjoint) bipartite graphs. Let us form $H=(V\sqcup W,E)$ by choosing arbitrary edges $e_1=(v_1,w_1)\in E_1$ and $e_2=(v_2,w_2)\in E_2$, and identifying $v_1$ and $v_2$ to a single vertex $v\in V$, also $w_1$ and $w_2$ to a single vertex $w\in W$, as well as identifying $e_1$ and $e_2$ to a single edge $e\in E$. It is not hard to see that the polymatroid $P_H$ is equal to the aggregated sum of $P_{H_1}$ and $P_{H_2}$. Precisely, for $\a\in P_{H_1}$ and $\b\in P_{H_2}$, let us define the vector $\c:=\a\boxplus\b\in \Z^W$ by $c_u=a_u$ for $u\in W_1\setminus \{w_1\}$, $c_u=b_u$ for $u\in W_2\setminus \{w_2\}$, and $c_w=a_{w_1}+b_{w_2}$. Then it is not hard to see that $\c$ is in $P_H$ and that the above gives a bijection between $P_{H_1}\boxplus P_{H_2}$ and $P_H$ (see~\cite[Theorem 6.7]{Kal} for details).

Thus by Proposition~\ref{prop:TP_properties}(e), we get 
\begin{equation}
\T_{P_H}(x,y)=\frac{\T_{P_{H_1}}(x,y)\, \T_{P_{H_2}}(x,y)}{x+y-1}.
\end{equation}

\begin{example}
Let $H$ be the bipartite graph which is the 4-cycle. Then $\vec H=\vec H(G)$ is the hypergraph associated to the graph $G$ having two vertices and two edges between them. One has $T_G(x,y)=x+y$, hence by Theorem~\ref{thm:oldandnew},  $\mT_{P_H}=(x+y-1)(x+y)$. Therefore the polymatroid Tutte polynomial of both (isomorphic) hypergraphs derived from the bipartite graph 
\begin{tikzpicture}[baseline=2pt,scale=.2]
\draw [fill] (0,0) circle [radius=.4];
\draw [fill] (2,0) circle [radius=.4];
\draw [fill] (4,0) circle [radius=.4];
\draw [fill] (0,2) circle [radius=.4];
\draw [fill] (2,2) circle [radius=.4];
\draw [fill] (4,2) circle [radius=.4];
\draw [thick] (0,0) rectangle (4,2);
\draw [thick] (2,0) -- (2,2);
\end{tikzpicture}
is $(x+y)^2(x+y-1)$.
\end{example}

\subsection{Polymatroid duality for planar hypergraphs}

We now recall the connection between duality of plane bipartite graphs 
and 
their polymatroids. 
Let $H=(V\sqcup W,E)$ be a 
connected bipartite graph properly embedded in the plane, where this time we do allow multiple edges between elements of $V$ and $W$. 
Then there is a notion of a dual bipartite graph $H^*=(V'\sqcup W,E')$, obtained as follows: 
\begin{compactitem}
\item Place a vertex in each face of $H$; this gives the vertex set $V'$.
\item For every $v'\in V'$, create an edge from $v'$ to each vertex $w$ of $W$ incident to the face of $H$ containing $v'$ (more precisely, an edge is created between $v'$ and $w$ for each incidence of $w$ with the face containing $v'$); 
this gives the edge set $E'$. 
\end{compactitem}
This construction is illustrated in Figure \ref{fig:duality}. It is easy to see that $(H^*)^*=H$ for all~$H$.

We stress again that the order of $V$ and $W$ is part of the structure of $H$, which makes $H$ equivalent to the hypergraph $\vec H$. 
 By iterating the operations of transposition and duality, a total of six hypergraphs can be generated from any connected plane bipartite graph. (This idea goes back to Cori and Penaud \cite{cp}.) They form a structure called a \emph{trinity} 
that was first studied by Tutte \cite{TTT}.

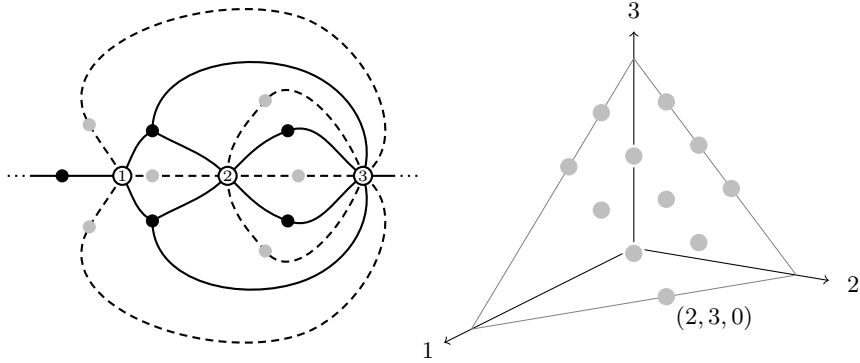
\begin{figure}[h]
\centering
\begin{tikzpicture}[scale=.20]


\draw [thick] (2,3) -- (0,3);
\draw [thick,dotted] (0,3) -- (-1.8,3);
\draw [thick] (22,3) -- (24,3);
\draw [thick,dotted] (24,3) -- (25.8,3);
\draw [thick] (2,3) -- (6,3);
\draw [thick] (6,3) to [out=60,in=-150] (8,6);
\draw [thick] (6,3) to [out=-60,in=150] (8,0);
\draw [thick] (8,6) to [out=-30,in=135] (13,3);
\draw [thick] (8,0) to [out=30,in=-135] (13,3);
\draw [thick] (13,3) to [out=45,in=-160] (17,6);
\draw [thick] (13,3) to [out=-45,in=160] (17,0);
\draw [thick] (17,6) to [out=20,in=144] (22,3);
\draw [thick] (17,0) to [out=-20,in=-144] (22,3);
\draw [thick] (8,6) to [out=90,in=160] (18,10) to [out=-20,in=72] (22,3);
\draw [thick] (8,0) to [out=-90,in=-160] (18,-4) to [out=20,in=-72] (22,3);

\draw [thick,densely dashed] (6,3) -- (13,3);
\draw [thick,densely dashed] (13,3) -- (22,3);
\draw [thick,densely dashed] (13,3) to [out=90,in=-130] (15.5,8) to [out=50,in=108] (22,3);
\draw [thick,densely dashed] (13,3) to [out=-90,in=130] (15.5,-2) to [out=-50,in=-108] (22,3);
\draw [thick,densely dashed] (6,3) to [out=120,in=-60] (3.8,6.4) to [out=120,in=160] (18,13.5) to [out=-20,in=36] (22,3);
\draw [thick,densely dashed] (6,3) to [out=-120,in=60] (3.8,-.4) to [out=-120,in=-160] (18,-7.5) to [out=20,in=-36] (22,3);


\draw [fill] (2,3) circle [radius=.4];
\draw [fill] (8,0) circle [radius=.4];
\draw [fill] (8,6) circle [radius=.4];
\draw [fill] (17,0) circle [radius=.4];
\draw [fill] (17,6) circle [radius=.4];
\draw [thick,fill=white] (6,3) circle [radius=.6];
\draw [thick,fill=white] (13,3) circle [radius=.6];
\draw [thick,fill=white] (22,3) circle [radius=.6];
\draw [lightgray,fill=lightgray] (8,3) circle [radius=.4];
\draw [lightgray,fill=lightgray] (17.7,3) circle [radius=.4];
\draw [lightgray,fill=lightgray] (3.8,6.4) circle [radius=.4];
\draw [lightgray,fill=lightgray] (3.8,-.4) circle [radius=.4];
\draw [lightgray,fill=lightgray] (15.5,-2) circle [radius=.4];
\draw [lightgray,fill=lightgray] (15.5,8) circle [radius=.4];

\node at (6,3) {\tiny $1$};
\node at (13,3) {\tiny $2$};
\node at (22,3) {\tiny $3$};

\begin{scope}[shift={(40,0)},scale=1.8]
\draw [->] (0,-1) -- (0,7);
\draw [->] (0,-1) -- (-7,-4.5);
\draw [->] (0,-1) -- (7.2,-2.2);
\node at (-7.6,-4.8) {\small $1$};
\node at (8.1,-2.35) {\small $2$};
\node at (0,7.8) {\small $3$};
\draw [help lines] (0,6) -- (-6,-4) -- (6,-2) -- cycle;
\draw [fill=lightgray,lightgray] (1.2,-2.8) circle [radius=.3];
\draw [fill=white,white] (0,-1.2) circle [radius=.4];
\draw [fill=lightgray,lightgray] (0,-1.2) circle [radius=.3];
\draw [fill=lightgray,lightgray] (2.4,-.8) circle [radius=.3];
\draw [fill=lightgray,lightgray] (-1.2,.4) circle [radius=.3];
\draw [fill=lightgray,lightgray] (1.2,.8) circle [radius=.3];
\draw [fill=lightgray,lightgray] (3.6,1.2) circle [radius=.3];
\draw [fill=lightgray,lightgray] (-2.4,2) circle [radius=.3];
\draw [fill=white,white] (0,2.4) circle [radius=.4];
\draw [fill=lightgray,lightgray] (0,2.4) circle [radius=.3];
\draw [fill=lightgray,lightgray] (2.4,2.8) circle [radius=.3];
\draw [fill=lightgray,lightgray] (-1.2,4) circle [radius=.3];
\draw [fill=lightgray,lightgray] (1.2,4.4) circle [radius=.3];

\node [below right] at (1.2,-2.8) {\small $(2,3,0)$};
\end{scope}

\end{tikzpicture}
\caption{A pair of dual hypergraphs $\vec H$, $\vec H^*$ with hyperedge set $W=[3]$. The bipartite graph $H$, drawn solid, is isomorphic to that of Figure \ref{fig:hypergraph}. The bipartite graph $H^*$ is drawn dashed, and the associated polymatroid $P_{H^*}$ is depicted on the right.}
\label{fig:duality}
\end{figure}

It is known~\cite[Thm.\ 8.3]{Kal} that the spanning hypertrees of $\vec H$ and $\vec H^*$ are in a one-to-one correspondence. The canonical bijection associates to a spanning hypertree $\mathbf t=(t_w)_{w\in W}$ of $\vec H$ the spanning hypertree $\mathbf t^*=(\deg(w)-1-t_w)_{w\in W}$, where $\deg(w)$ denotes the degree of the vertex $w$ in $H$. This shows that the polymatroids $P_H$ and $P_{H^*}$ are dual (in the sense of Section~\ref{sec:duality}) to each other modulo a linear translation. More precisely, $P_{H^*}=-P_H+(\deg(w)-1)_{w\in W}$. Via Lemma~\ref{lem:duality}, this implies the following.

\begin{proposition}
For any pair of planar dual hypergraphs $\vec H$ and $\vec H^*$, we have \[\T_{P_H}(x,y)=\T_{P_{H^*}}(y,x).\]
\end{proposition}

\begin{example}
\label{ex:duality}
Figure~\ref{fig:duality} shows two dual hypergraphs $\vec H$, $\vec H^*$ (with $H$ as in Figure \ref{fig:hypergraph}). The reader can check that the polymatroid $P_{H^*}$ is the one represented at the right of Figure\ref{fig:duality}, and that   $P_{H^*}=-P_H+(2,3,4)$.
\end{example}

\section{Tutte polynomials of hypergraphs}
\label{sec:Tutte-hypergraph}

In this section we give a formula for the 
Tutte polynomial of a large family of hypergraphs. More precisely, given a simple bipartite graph $H=(V\sqcup W,E)$, we give a formula for the Tutte polynomial $\T_{P_{H^+}}(x,y)$, where $H^+=(V^+\sqcup W,E^+)$ is the bipartite graph obtained from $H$ by adding a 
vertex $v^+$ to $V$ that is adjacent to every vertex in $W$. 


We first need some notation and definitions.
For $v\in V$ and $F\subseteq E$, we let
$$N_F(v)=\{\,w\in W\mid\{v,w\}\in F\,\},$$
and we also let $N(v)=N_E(v)$ (which is also the set we denoted by $\vec H^\mir(v)$ earlier). 

\begin{definition}\label{def:draconian-set}
We say that a subset of edges $F\subseteq E$ is a \emph{draconian set} if there is no vertex $v\in V$ such that $N_F(v)=W$, and for any collection $\{I_1,\ldots,I_k\}$ of non-empty subsets of $W$, 
\begin{equation}\label{eq:draconian-set}
\#\{\,v\in V\mid N_F(v)\in \{I_1,\ldots,I_k\}\,\}<|I_1\cup \cdots \cup I_k|.
\end{equation}
\end{definition}
\TK{Can we add an example and/or some motivation of these definitions?} \OB{The motivation is simply to write the theorem. I am not sure we need an example.} \TK{The ``draconian'' terminology is certainly out of the blue...} \OB{OK:I added the following remark.}
\begin{remark} The term ``draconian'' in Definition \ref{def:draconian-set}, comes from the relation with Definition \ref{def:dragon-poly} of $n$-draconian sequences. 
Indeed, consider a subset of edges $F\subseteq E$  such that  there is no vertex $v\in V$ with $N_F(v)=W$. We associate to $F$ the function $g_F: 2^W\setminus \{\varnothing\}\to \Z^{\geq 0}$ defined by 
 $$\forall \varnothing \neq I\subsetneq W,~ g_F(I)=\#\{v\in V~|~N_F(v)=I\},$$
and $g_F(W)=|W|-1-\sum_{ \varnothing \neq I\subsetneq W} g_F(I)$. 
Then, $F$ is a draconian set if and only if $g_F$ is an $n$-draconian sequence (when identifying $W$ with $[n]$).  
\end{remark}

Lastly, for $v\in V$, and an ordered set partition $B=(B_1,\ldots,B_\ell)$ of $W$, we denote by $B(v)\subseteq W$ the block $B_m$, where $m=\max\{\,k\in [\ell]\mid N(v)\cap B_k\neq\varnothing\,\}$.

\begin{theorem} \label{thm:tutte-hypergraphs}
Let $H=(V\sqcup W,E)$ be a simple bipartite graph without isolated vertex in $V$. Let $k\geq 0$ be the number of vertices in $V$ that are adjacent to every vertex of $W$. 
Let $\preceq$ be a total order on $W$. Then, with the above notation,
$$\frac{\T_{P_{H^+}}(x,y)}{(x+y-1)}=\sum_{\substack{\text{draconian}\\ \text{sets }F\subseteq E}}\,\sum_{\substack{B\in\mB_W\\ F\text{-compatible}}}\!\!{k\one_{\ell(B)=1}-1\choose |W|-|V(F)|-1}(-1)^{\ell(B)-1}x^{\lr(B)-1}y^{\rl(B)-1},$$
where 
\begin{compactitem}
\item $V(F)=\{\,v\in V\mid N_F(v)\neq \varnothing\,\}$;  
\item $\mB_W$ is the set of ordered set partitions of $W$;
\item $B\in\mB_W$ is \emph{$F$-compatible} if for all $v\in V(F)$, we have $N_F(v)=N(v)\cap B(v)$;
\item for $B\in\mB_W$, the value $\ell(B)$ is the number of blocks of $B$, and $\lr(B)$ (resp., $\rl(B)$) is the number of left-to-right (resp., right-to-left) minima of $B$ when elements of $W$ are compared using the fixed linear order $\preceq$. (The definition of these minima was given in Section~\ref{sec:formula_universal_Tutte}.)
\end{compactitem}
\end{theorem}

Note that the binomial coefficient in Theorem~\ref{thm:tutte-hypergraphs} is equal to $(-1)^{|W|-|V(F)|-1}$ except for the partition $B^\circ=(W)$ with a single block. Observe all that the total contribution of the partition $B^\circ$ is $-\T_{P_{H^+}}(0,0)$, which (according to Proposition~\ref{prop:TP_properties}(g)) is equal to the number of interior lattice points of $\mP_{H^+}$. 

\begin{remark}
In Theorem~\ref{thm:tutte-hypergraphs} we do not assume that $H$ is connected. Indeed $H^+$ is always connected since $H$ has no isolated vertices in $V$.
As explained in Section~\ref{sec:hypergraph-as-Minkowski}, we have
$$\mP_{H^+}=\sum_{\varnothing \neq I\subseteq W}m_I\Delta_I,$$
where $m_I=\#\{\,v\in V\mid N(v)=I\,\}$ and $\Delta_I=\conv(\e_i\mid i\in W)$. Hence Theorem~\ref{thm:tutte-hypergraphs} can equivalently be stated as a result for polymatroids which are Minkowski sums of simplices.
\end{remark}

\OB{A question for us: I don't know of a similar formula for $\T_{P_H}(x,y)$ when there is no vertex $v\in V$ adjacent to all $W$. In particular I cannot use Theorem~\ref{thm:tutte-hypergraphs} to give a new formula for the Tutte polynomial of a graph. It would be interesting however to see what it gives for $\T_{P_{H^+}}(x,y)$ when $H$ corresponds to a graph (i.e. every right vertex $w\in W$ has degree 2). The case of zonotopes in the next section corresponds to computing $\T_{P_{H^+}}(x,y)$ when $H^\mir$ corresponds to a graph (i.e. every left vertex $v\in V$ has degree 2).}

\begin{proof}
Let us identify the ordered set $(W,\preceq)$ with $[n]$, where $n=|W|$.
As $\mP_{H^+}=\sum_{I\subseteq [n]}m_I\Delta_I$, where $m_I=\#\{\,v\in V\mid N(v)=I\,\}$, 
Theorem~\ref{thm:formulaTn} (together with Remark~\ref{cor:formulaTn}) gives 
\begin{eqnarray}
\frac{\T_{P_{H^+}}(x,y)}{x+y-1}&=&\sum_{B\in \mB_n} D_n((m_I^B))\,(-1)^{\ell(B)-1}\,x^{\lr(B)-1}\,y^{\rl(B)-1}\nonumber \\
&=& \sum_{B\in \mB_n}~\sum_{g\,:\,n\text{-draconian}} {m_{[n]}^B-1\choose g([n])}\prod_{\varnothing \neq I\subsetneq [n]}{m_I^B \choose g(I)}\, \weight(B),\label{eq:Tn2}
\end{eqnarray}
where $\m=(m_I^B)$ is given by~\eqref{eq:mB} and $\weight(B)=(-1)^{\ell(B)-1}\,x^{\lr(B)-1}\,y^{\rl(B)-1}$. 
Observe that for all $\varnothing \neq I\subseteq [n]$
$$m_I^B=\#\{\,v\in V\mid N(v)\cap B(v)=I\,\}.$$
We now interpret the product of binomial coefficients in \eqref{eq:Tn2} combinatorially. 
For $B\in \mB_n$, we denote by $\Omega(B)$ the set of tuples $(V_I)_{\varnothing \neq I\subsetneq [n]}$ such that $V_I\subseteq\{\,v\in V\mid N(v)\cap B(v)=I\,\}$.
For an $n$-draconian function $g$, we say that the tuple $(V_I)\in \Omega(B)$ is $g$\emph{-draconian} if $|V_I|=g(I)$ for all non-empty subsets $I\subsetneq[n]$.
Then the product 
$$\prod_{\varnothing \neq I\subsetneq [n]}{m_I^B \choose g(I)}$$
clearly equals the number of $g$-draconian tuples $(V_I)$ in $\Omega(B)$.
For any tuple $\V=(V_I)\in \Om(B)$, the subsets $V_I$ are necessarily all disjoint. We can therefore associate to $\V$ a subset of edges 
$$F_\V:=\biguplus_{I\subsetneq [n]}\{\,\{v,w\}\mid v\in V_I,~w\in I\,\}\subseteq E.$$
It is easy to see that $F_\V$ is a draconian set if and only if $\V$ is  \emph{draconian}, that is,  $g$-draconian for some $n$-draconian function $g$. Furthermore the mapping $\V\mapsto F_V$ is clearly a bijection between the set of draconian tuples $\V$ in $\Om(B)$, and the set of draconian sets $F\subseteq E$ such that $B$ is $F$-compatible. 
Thus, 
\begin{eqnarray*}
\frac{\T_{P_{H^+}}(x,y)}{x+y-1}&=& \sum_{B\in \mB_n}\,\sum_{g\,:\,n\text{-draconian}}\,\sum_{\substack{\V\in \Om(B)\\ g\text{-draconian} }}{m_{[n]}^B-1\choose g([n])}\,\weight(B)\\
&=& \sum_{B\in \mB_n}\,\sum_{\substack{F\subseteq E\,:\,\text{draconian set}\\ \text{such that\ }B\text{ is }F\text{-compatible}}}{m_{[n]}^B-1\choose n-|V(F)|-1}\,\weight(B),
\end{eqnarray*}
where the second equality comes from the fact that a $g$-draconian tuple $\V$ satisfies
$$g([n])=n-1-\sum_{\varnothing \neq I\subsetneq [n]}g(I)=n-|V(F_\V)|-1.$$
Reordering the sums, and observing $m_{[n]}^B=k\one_{\ell(B)=1}$, gives the theorem.
\end{proof}


\section{Tutte polynomials of graphical zonotopes}\label{sec:zonotopes}

In this section we use the formula of the universal Tutte polynomial (precisely, Theorem \ref{thm:tutte-hypergraphs}) in order to compute the polymatroid Tutte polynomials of graphical zonotopes. 
It turns out that for any graph $G$, the polymatroid Tutte polynomial of the zonotope of $G$ is a specialization of the classical Tutte polynomial of $G$.

Recall that for a loopless graph $G=(W,E)$ (with at least one edge), the zonotope of $G$ is the polytope $\mZ(G)\subset \R^W$ obtained as the following Minkowski sum of line segments:
$$\mZ(G)=\sum_{\{i,j\}\subseteq W}m_{\{i,j\}} \,[\e_i,\e_j],$$
where $m_{\{i,j\}}$ is the number of edges of $G$ with endpoints $\{i,j\}$. 
Recall from Remark~\ref{rk:zonotope-is-Hmirror} that $Z(G):=\mZ(G)\cap \Z^W$ is a hypergraphical matroid.


\begin{theorem}\label{thm:Tuttezonotope}
Let $G=(W,E)$ be a loopless graph with $k$ connected components, and at least one edge. Let $\mZ(G)$ be the zonotope associated to $G$.
The polymatroid Tutte polynomial of $Z(G):=\mZ(G)\cap \Z^W$ is given in terms of the classical Tutte polynomial of the graph $G$ by
$$\T_{Z(G)}(x,y)=(x+y-1)^{k}~T_G(x+y,1).$$
Equivalently,
$$\T_{Z(G)}(x,y)=\sum_{F\,:\,\text{spanning forest of }G}(x+y-1)^{\#\text{connected components of }F},$$
where the sum is over the subset of edges $F\subseteq E$ such that the subgraph $G_F=(W,F)$ has no cycle.
\end{theorem}

\begin{example} The classical permutohedron $\Pi_n=\conv((w(1),\ldots,w(n))\mid w\in S_n)$ is a graphical zonotope. More precisely, $\Pi_n=(1,1,\ldots,1)+\mZ(K_n)$, where $K_n$ is the complete graph on $n$ vertices. Hence, by Theorem~\ref{thm:Tuttezonotope},
$$\T_{\Pi_n\cap \Z^n}(x,y)=\sum_{F}(x+y-1)^{\#\text{ connected components of }F},$$
where the sum is over forests (i.e. acyclic graphs) with vertex set $[n]$. 
For instance, $\T_{\Pi_3\cap \Z^3}(x,y)=(x+y-1)^3+3(x+y-1)^2+3(x+y-1)$, and 
$$\T_{\Pi_4\cap \Z^4}(x,y)=(x+y-1)^4+6(x+y-1)^3+15(x+y-1)^2+16(x+y-1).$$
\end{example}

\begin{remark}
Specializing Theorem~\ref{thm:Tuttezonotope} at $x=y=1$ gives $|Z(G)|=T_G(2,1)$. 
The reader will recognize that this is a well known specialization of the Tutte polynomial~\cite{Sta} upon remembering that the points of $Z(G)$ are in clear bijection with the outdegree sequences of the orientations of $G$. 
The specialization $x=0, y=1$ also gives a classical result via Proposition~\ref{prop:TP_properties}(h), namely that the specialization $T_G(1,1)$ counts root-connected outdegree sequences of $G$. 
As we explain in Section~\ref{sec:mirror-symmetry}, the case $y=1$ of Theorem~\ref{thm:Tuttezonotope} can in fact be deduced from a relation proven in~\cite{KP} between the interior polynomial of a hypergraph and its transpose.
\end{remark}

\begin{proof} 
The case $|W|=2$ can be treated by hand, and we now assume $|W|>2$.
Let $H_G=(V\sqcup W,E')$ be the bipartite graph obtained from $G=(W,E)$ by inserting a vertex $v_e\in V$ in the middle of each edge $e\in E$ (in other words, each edge $e\in E$ is replaced by a path of length 2). By Remark \ref{rk:zonotope-is-Hmirror}, $Z(G)=P_{H_G^+}$. Hence Theorem~\ref{thm:tutte-hypergraphs} gives
$$\frac{\T_{Z(G)}(x,y)}{(x+y-1)}=\!\sum_{\substack{F'\subseteq E'\\ \text{draconian set }}}\sum_{\substack{B\in\mB_W\\ F'\text{-compatible}}}\!\!\!\!(-1)^{|W|-|V(F')|-1}\,(-1)^{\ell(B)-1}\,x^{\lr(B)-1}\,y^{\rl(B)-1},$$
where we have used $k=0$ (which holds because $|W|>2$).

Next we observe that the draconian sets $F'\subseteq E'$ are in bijection with the spanning forests of $G$. 
Indeed, observe that if $F'\subseteq E'$ is a draconian set, then any vertex $v\in V$ is incident to either $0$ or $2$ edges from $F'$ (because if we had $N_{F'}(v)=\{w\}$, then the draconian condition \eqref{eq:draconian-set} would be violated for $I_1=\{w\}$). Hence, to any draconian set $F'\subseteq E'$, we can bijectively associate the subset of edges $F=\{\,e\in E\mid N_{F'}(v_e)=2\,\}$. Furthermore, it is not hard to see that $F'\subset E'$ is draconian if and only $F$ is a forest (because if $F$ had a cycle then $F'$ would violate the draconian condition \eqref{eq:draconian-set}). 

We also observe that a partition $B\in\mB_W$ is $F'$-compatible if and only if for all $e\in F$, the endpoints of $e$ are in the same block of $B$. In other words, denoting by $C_F$ the unordered partition of $W$ whose blocks correspond to the connected components of $F$, we see that $B\in\mB_W$ is $F'$-compatible if and only if $B$ is a coarsening of $C_F$. 

Hence we get
$$\T_{Z(G)}(x,y)= \sum_{\substack{F\subseteq E\\ \text{forest }}}\!\!(-1)^{k(F)}\sum_{\substack{B\in\mB_W\\ \text{coarsening of } C_F}}(-1)^{\ell(B)}(x+y-1)\,x^{\lr(B)-1}\,y^{\rl(B)-1},$$
where $k(F)=|W|-|V(F')|$ is the number of connected components of $F$. 
Moreover it is clear that the inner sum only depends on $k(F)$, and is equal to 
$$\sum_{B\in\mB_{k(F)}}(-1)^{\ell(B)}(x+y-1)\,x^{\lr(B)-1}\,y^{\rl(B)-1}.$$
We now show that this sum is equal to $(1-x-y)^{k(F)}$ using some well-known combinatorial tricks (see for instance~\cite[Chapter 1]{Sta2}). Let $\mC_k$ be the set of unordered set partitions of $[k]$. First, let us compute the contribution of the ordered set partitions $B\in\mB_k$ corresponding to a given unordered partition $C\in \mC_k$ with $\ell$ blocks.
For a permutation $w\in S_\ell$, we denote by $\lr(w)$ (resp., by $\rl(w)$) the number of left-to-right (resp., right-to-left) minima of $w$:
$$\lr(w)=\#\{\,i\in[\ell]\mid\forall j<i,w(i)<w(j)\,\}\text{ and }\rl(w)=\#\{\,i\in[\ell]\mid\forall j>i,w(i)<w(j)\,\}.$$ 
It is easy to see that 
$$\sum_{\substack{B\in\mB_k \\ \text{ ordering of }C}}(x+y-1) x^{\lr(B)-1}y^{\rl(B)-1}=\sum_{w \in S_\ell} (x+y-1)x^{\lr(w)-1}y^{\rl(w)-1}=(x+y-1)^{(\ell)},$$
where $u^{(\ell)}:=u(u+1)\cdots(u+\ell-1)$ (the second equality is justified by considering the insertion process producing the permutation $w$ seen as a word of $\ell$ letters). 
Moreover we have
$$\sum_{C\in\mC_k}(-1)^{\ell(C)}u^{(\ell(C))}=\sum_{C\in\mC_k}(-u)_{(\ell(C))}=(-u)^{k},$$
where $u_{(\ell)}:=u(u-1)\cdots(u-\ell+1)$ (the second equality is justified by recognizing that for $u\in \Z_{<0}$, both sides count the colorings of the set $[-u]$ by $k$ colors). 
Hence 
$$\sum_{B\in\mB_k}\!(-1)^{\ell(B)}(x+y-1)\,x^{\lr(B)-1}\,y^{\rl(B)-1}=\sum_{C\in\mC_k}\!(-1)^{\ell(C)}(x+y-1)^{(\ell(C))}=(1-x-y)^{k},$$
as claimed. This completes the proof.
\end{proof}

\section{Boxed Tutte polynomials of hypergraphs and boxed polymatroids}\label{sec:boxed}

Most polymatroids that one encounters (for instance, usual matroids and hypergraphical
polymatroids) satisfy natural lower and upper bounds on the components of their elements (e.g.,
non-negativity). We now incorporate this in our theory by defining boxed
polymatroids and their Tutte polynomials.

\subsection{Boxed polymatroids and boxed Tutte polynomials}
A \emph{box} for a polymatroid $P\subset \Z^n$ is a pair of tuples $\bal=(\al_1,\ldots,\al_n)$ and $\bbe=(\be_1,\ldots,\be_n)\in \Z^n$ such that every point $\a$ in $P$ satisfies $\bal\leq \a\leq \bbe$, where inequalities are taken componentwise. In other words, $P$ is contained in the product of intervals $\prod_{i}[\al_i,\be_i]$.
A \emph{boxed polymatroid} is a triple $(P,\bal,\bbe)$, where $P$ is a polymatroid and $(\bal,\bbe)$ is a box for $P$. 

We now define the boxed Tutte polynomial of the boxed polymatroid $(P,\bal,\bbe)$. Recall from Section~\ref{sec:order-invariance} the notation $\wt(\c)=u^{d_1^{>}(P,\c)}\, v^{d_1^{<}(P,\c)}$ for $\c\in \Z^n$. The \emph{boxed Tutte polynomial} of $(P,\bal,\bbe)$ is then defined as 
\begin{equation}\label{eq:def-Tutte-boxed}
\wti \T_{(P,\bal,\bbe)}(u,v)=\sum_{\c\in \Z^n,~\bal\leq \c\leq \bbe}\wt(\c).
\end{equation}
observe that the invariant $\wti \T_{P}(u,v)=\sum_{\c\in \Z^n}\wt(\c)$ of Section~\ref{sec:order-invariance} is the limit (in the sense of formal power series in $u,v$) of $\wti \T_{(P,\bal,\bbe)}(u,v)$ as $\al_1,\ldots,\al_n\to -\infty$ and $\be_1,\ldots,\be_n\to +\infty$. 

It is clear that the boxed Tutte polynomial satisfies the \emph{translation invariance}:
$$\text{for all }\c\in\Z^n,~~ \wti \T_{(P+\c,\bal+\c,\bbe+\c)}(u,v)=\wti \T_{(P,\bal,\bbe)}(u,v),$$
 the \emph{$S_n$-invariance}: 
$$\text{for all }w\in S_n,~~ \wti \T_{(w(P),w(\bal),w(\bbe))}(u,v)=\wti \T_{(P,\bal,\bbe)}(u,v),$$
and the \emph{duality relation}:
$$\wti \T_{(-P,-\bbe,-\bal)}(u,v)=\wti \T_{(P,\bal,\bbe)}(v,u).$$

The boxed Tutte polynomial is closely related to the classical Tutte polynomial of matroids.
By Remark~\ref{rk:mat-are-polymat}, matroids on the ground set $[n]$ can be identified with boxed polymatroids of the form $(P,0^n,1^n)$, where $0^n:=(0,\ldots,0)$ and $1^n:=(1,\ldots,1)$. Moreover, it follows from Remark~\ref{rk:corank-null} (together with the corank-nullity formula~\eqref{eq:corank-nullity-mat} for the Tutte polynomial of $M$) that 
$$T_M(x,y)=\wti \T_{(P(M),0^n,1^n)}(x-1,y-1).$$

Next, we express the boxed Tutte polynomial $\wti \T_{(P,\bal,\bbe)}$ in terms of internal and external activities.

\begin{proposition}
For any boxed polymatroid $(P,\bal,\bbe)$ in $\Z^n$, 
\begin{multline*}
\wti \T_{(P,\bal,\bbe)}(u,v)
=\sum_{\a\in P}~\prod_{i=1}^n\left(1+\one_{i\in \Int(\a)}\frac{u(1-u^{a_i-\al_i})}{(1-u)}+\one_{i\in \Ext(\a)}\frac{v(1-v^{\be_i-a_i})}{(1-v)}\right).
\end{multline*}
where $\one_{condition}$ denotes the characteristic function, which is equal to 1 if the condition is true and 0 otherwise.
\end{proposition}


\begin{proof}
This is analogous to the proof of Theorem~\ref{th:tilde_Tutte}.
The decomposition $\Z^n=\biguplus_{\a\in P} C_P(\a)$ from Theorem~\ref{th:decomposition_into_cones} gives 
$$
\wti \T_{(P,\bal,\bbe)}(u,v) = \sum_{\a\in P} ~~\sum_{\c\in C_P(\a),~\bal\leq \c\leq \bbe} \wt_P(\c).
$$
Moreover, Theorem~\ref{th:decomposition_into_cones}(2)
together with Lemma~\ref{lem:d_geq_leq} yield
$$
\sum_{\c\in C_P(\a),\,\bal\leq \c\leq \bbe} \wt_P(\c) = 
\prod_{i=1}^n \tilde m_i(\a),
$$
where 
$$
\tilde m_i(\a) = 
\left\{
\begin{array}{cl}
1 & \text{if } i\not\in\Int(\a)\cup \Ext(\a),\\
\sum_{k=0}^{a_i-\al_i}u^k & \text{if } i\in\Int(\a)\setminus \Ext(\a),\\
\sum_{k=0}^{\be_i-a_i}v^k & \text{if } i\in\Ext(\a)\setminus \Int(\a),\\
1+\sum_{k=1}^{a_i-\al_i}u^k+\sum_{k=1}^{\be_i-a_i}v^k & \text{if } i\in\Int(\a)\cap \Ext(\a).
\end{array}
\right.
$$
Finally, observing $\ds\sum_{k=0}^{a_i-\al_i}u^k=1+\frac{u(1-u^{a_i-\al_i})}{1-u}$ and $\ds \sum_{k=0}^{\be_i-a_i}v^k=1+\frac{v(1-v^{\be_i-a_i})}{1-v}$ proves the stated relation.
\end{proof}
 
We can also derive expressions for the boxed Tutte polynomial in the spirit of Theorem~\ref{th:upper_lower_shadow}. Indeed, using the shadow decomposition of Theorem~\ref{th:shadow_decomposition} and proceeding as in the proof of Theorem~\ref{th:upper_lower_shadow}, one gets: 

\begin{proposition}\label{prop:shadows-Tutte-hyper}
For any boxed polymatroid $(P,\bal,\bbe)$ in $\Z^n$, 
$$
\wti \T_{(P,\bal,\bbe)}(u,v)= \sum_{\a\in S^{\textrm{up}}(P),\,\bal\leq \a\leq \bbe~} 
v^{\level(\a)-\level(P)}\prod_{i\in [n]}\left(1+\one_{i\in\tInt(\a)}\frac{u(1-u^{a_i-\al_i})}{1-u}\right)
$$
and
$$
\wti \T_{(P,\bal,\bbe)}(u,v)= \sum_{\a\in S^{\textrm{low}}(P),\,\bal\leq \a\leq \bbe~} 
u^{\level(P)-\level(\a)}\prod_{i\in [n]}\left(1+\one_{i\in\tExt(\a)}\frac{v(1-v^{\be_i-a_i})}{1-v}\right).
$$
\end{proposition}

\subsection{Boxed Tutte polynomials of hypergraphs}
\label{sec:boxed-hypergraphs}
Let $H=(V\sqcup W,E)$ be a connected bipartite graph, and let $\vec{H}$ be the corresponding hypergraph. 
We first recall and generalize the notion of spanning hypertrees of $\vec{H}$.

Let $\mS$ be the set of subgraphs of $H$ which contain every vertex of $H$ and at least one edge incident to each right vertex $w\in W$.
For a subgraph $S\in\mS$ we denote by $\c_S\in \Z^W$ the tuple $(c_w)_{w\in W}$ where for all right vertex $w\in W$, $c_w$ is equal to the number of edges of $S$ incident to the vertex $w$ \emph{minus 1}. 
Recall from Section~\ref{sec:hypergraphs} that a \emph{spanning hypertree} of $\vec H$ is a point of the form $\c_T\in \Z^W$ for a spanning tree $T\in \mS$ of $H$. 
Similarly, we call \emph{spanning hypergraph} (resp. \emph{spanning hyperforest}, \emph{spanning hyperconnex}) of $\vec H$ a point of the form $\c_S\in \Z^W$ for a subgraph (resp. acyclic subgraph, connected subgraph) $S\in \mS$.

\begin{remark}
\label{rk:hypergraph-ga}
Suppose that the bipartite graph $H$ is such that every right vertex $w\in W$ has degree 2. 
In this case we can identify the hypergraph $\vec H$ with the graph $G=(V,E')$ obtained by replacing every right vertex $w\in W$ and the incident edges by a single edge joining the vertices adjacent to $w$. 
Observe that in this context, the set of spanning hypersubgraphs (resp.\ hyperforests, hyperconnexes, hypertrees) of $\vec{H}$ is in obvious bijection with the set of spanning subgraphs (resp. forests, connected subgraphs, trees) of $G$. 
\end{remark}
 
As in Section~\ref{sec:hypergraphs}, we denote by $P_H$ the polymatroid whose elements are the spanning hypertrees of $\vec {H}$. In symbols,
$$P_H:=\{\,\c_T\mid T\in \mT\,\},$$
where $\mT\subset \mS$ is the set of spanning trees of $H$.
Clearly the pair $(0^W,\c_H)$ is a box for the polymatroid $P_H$. In fact, the set 
$$\hypbox_H:=\{\c\in\Z^W~\mid~0^W\leq \c\leq \c_H\}$$
is clearly equal to the set of spanning hypersubgraphs of $\vec H$. Moreover, $\Slow(P_H)\cap \hypbox_H$ is the set of spanning hyperforests of $\vec {H}$, and $\Sup(P_H)\cap \hypbox_H$ is the set of spanning hyperconnexes.
In particular, we obtain the following evaluations of the boxed Tutte polynomial:
\begin{eqnarray*}
\wti \T_{(P_H,0^W,\c_H)}(0,0)&=&\# \text{spanning hypertrees},\\
\wti \T_{(P_H,0^W,\c_H)}(0,1)&=&\# \text{spanning hyperconnexes},\\
\wti \T_{(P_H,0^W,\c_H)}(1,0)&=&\# \text{spanning hyperforests},\\
\wti \T_{(P_H,0^W,\c_H)}(1,1)&=&\# \text{spanning hypersubgraphs}= \prod_{w\in W}\deg(w). 
\end{eqnarray*}
In fact, Proposition~\ref{prop:shadows-Tutte-hyper} further gives
$$\wti \T_{(P_H,0^W,\c_H)}(0,v)=\sum_{\c\,\text{spanning hyperconnex}}\!\!v^{\level(\c)-\level(P)},$$
and 
$$\wti \T_{(P_H,0^W,\c_H)}(u,0)=\sum_{\c\,\text{spanning hyperforest}}\!\!u^{\level(P)-\level(\c)}.$$

We now relate the weight $\wt_P(\c)$ to the corank and nullity of the subgraphs of $H$. 
Consider the matroid $M_H$ (with ground set $E$) associated to the graph $H$: the bases of $M_H$ are the spanning trees of $H$ (considered as subsets of $E$). Every spanning subgraph $S\in\mS$ can be identified with a subset of $E$, whose corank and nullity in the matroid $M_H$ are given by
$$\cork(S)=\min\{\,|T\setminus S|\mid T\in\mT\,\}\quad\text{and}\quad\nul(S)=\min\{\,|S\setminus T|\mid T\in\mT\,\},$$
respectively.


\begin{lemma}\label{lem:cork-null-hyper}
For all $\c\in\hypbox_H$ we have
$$d_1^>(P,\c)=\min\{\,\cork(S)\mid S\in \mS,~\c_S=\c\,\},$$ 
and dually
$$d_1^<(P,\c)=\min\{\,\nul(S)\mid S\in \mS,~\c_S=\c\,\}.$$
\end{lemma}

\begin{proof}
We prove the identity for $d_1^>(P,\c)$, the other one being symmetric. 
Let $m=\min\{\,\cork(S)\mid S\in \mS,~\c_S=\c\,\}$. 
We first show $d_1^>(P,\c)\leq m$. 
Let $S\in \mS$ be such that $\c_S=\c$ and let $T\in\mT$.
Clearly $d_1^>(P,\c)\leq d_1^>(\c_T,\c_S)\leq |T\setminus S|$.
Since by definition, $m=\min\{\,|T\setminus S|\mid T\in \mT,~S\in \mS,~\c_S=\c\,\}$, this gives $d_1^>(P,\c)\leq m$. 

We now show $d_1^>(P,\c)\geq m$.
Let $T\in\mT$ be such that $d_1^>(P,\c)=d_1^>(\c_T,\c)$. It is clearly possible to remove $d_1^>(\c_T,\c)$ edges from $T$ and then add $d_1^>(\c_T,\c)$ edges 
so as to obtain a subgraph $S\in\mS$ such that $\c=\c_{S}$. Moreover $m\leq \cork(S)\leq |T\setminus S|=d_1^>(\c_T,\c_S)=d_1^>(P,\c)$.
\end{proof}

Lemma~\ref{lem:cork-null-hyper} allows us to rewrite~\eqref{eq:def-Tutte-boxed} in the context of hypergraphs:
\begin{equation}\label{eq:Tutte-hyper}
\wti \T_{(P_H,0^W,\c_H)}(u,v)=\sum_{\substack{\c\text{ spanning }\\\textrm{hypersubgraph of }H}}\!\!\!\!\!\!\! u^{\min\{\cork(S)\mid S\in\mS,\,\c_S=\c\}}\,v^{\min\{\nul(S)\mid S\in\mS,\,\c_S=\c\}}.
\end{equation}

\begin{remark}
\label{rk:hypergraph-ga2}
In the case when the hypergraph $H$ identifies with a graph $G=(V,E')$ in the sense of Remark~\ref{rk:hypergraph-ga}, both exponents in~\eqref{eq:Tutte-hyper} are minima of singletons and~\eqref{eq:Tutte-hyper} reduces to a familiar expression for the Tutte polynomial of $G$:
\begin{equation*}
\wti \T_{(P_H,0^W,\c_H)}(u,v)=\sum_{S'\subseteq E'}u^{\cork(S')}\,v^{\nul(S')}=T_G(u+1,v+1).
\end{equation*}
\end{remark}

Remember now the definition of duality of planar hypergraphs as defined in Section~\ref{sec:hypergraphs}. Let $H$ be a plane hypergraph and let $H^*$ be its dual. Since $P_{H^*}=-P_H+\c_H$, we know that
\begin{equation*}
\wti \T_{(P_{H^*},0^W,\c_H)}(u,v)=\wti \T_{(P_{H},0^W,\c_H)}(v,u).
\end{equation*}
In fact, the mapping $\phi\colon \hypbox_H\to \hypbox_H$ defined by $\phi(\c)=\c_H-\c$ is a bijection between the hypersubgraphs of $H$ and those of $H^*$ exchanging $d_1^<$ and $d_1^>$. Denoting by $\mS^*$, $\cork^*$, and $\nul^*$, respectively, the analogues of $\mS$, $\cork$, and $\nul$ for the hypergraph $H^*$, Lemma~\ref{lem:cork-null-hyper} gives, for all $\c\in\hypbox_H$,
$$\min\{\,\cork(S)\mid S\in\mS,\,\c_S=\c\,\}=\min\{\,\nul^*(S)\mid S\in\mS^*,\,\c_S=\c_H-\c\,\},$$ 
and dually
$$
\min\{\,\nul(S)\mid S\in\mS,\,\c_S=\c\,\}=\min\{\,\cork^*(S)\mid S\in\mS^*,\,\c_S=\c_H-\c\,\}.$$


\section{Concluding remarks and questions}\label{sec:conclusion}
We conclude with some open questions about the polymatroid Tutte polynomial and some directions for future research.

\subsection{Trivariate Tutte polynomial and mixed Tutte polynomials}

One consequence of the existence of the universal Tutte polynomial $\T_n$ is that for every polymatroid $P$, there exists a three-variable polynomial $\T_P(x,y;u)$ such that for every non-negative integer $k$ we have
$$T_P(x,y;k)=T_{k P}(x,y),$$
where $k P=\{\,(k a_1,\ldots,k a_n)\mid\a\in P\,\}$ denotes the $k$-dilation of $P$. Indeed, if $P$ has rank function $f$, then $k P$ has rank function $k f$. Hence
the polynomial $\T_P(x,y;u)$ is obtained from $\T_n(x,y,(z_I))$ by substituting each variable $z_I$ by $uf(I)$. Note that this trivariate invariant inherits from $\T_n$ the translation invariance property $\T_{\c+P}(x,y;u)=\T_P(x,y;u)$, the $S_n$-invariance property $\T_{w(P)}(x,y;u)=\T_P(x,y;u)$, and the duality relation $\T_{-P}(x,y;u)=\T_P(y,x;u)$. One might wonder about other properties of this invariant. For instance, are there nice combinatorial-reciprocity properties when evaluating the variable $u$ at negative integers?
We can also ask what information is captured by the trivariate polynomial $\T_P(x,y;u)$, compared to $\T_P(x,y)=\T_P(x,y;1)$ (in particular when $P=P(M)$ is associated to a matroid $M$). 

One can generalize the preceding idea to tuples of polymatroids. For polymatroids $P_1,\ldots,P_d$, there exists a unique $(k+2)$-variate \emph{mixed Tutte polynomial} $\T_{P_1,\ldots,P_d}(x,y;u_1,\ldots,u_d)$ such that for any tuple $(k_1,\ldots,k_d)$ of positive integers, 
$$\T_{P_1,\ldots,P_d}(x,y;k_1,\ldots,k_d)=\T_{k_1P_1+\cdots+k_d P_d}(x,y).$$
Indeed, remembering the identity $P_{f}+P_{g}=P_{f+g}$, we see that the above polynomial is given by
$$\T_{P_1,\ldots,P_d}(x,y;u_1,\ldots,u_d)=\T_n(x,y,(u_1f_1(I)+\cdots + u_d f_d(I))),$$
where $f_1,\ldots,f_d$ are the rank functions of $P_1,\ldots,P_d$, respectively. Note that the Cameron--Fink polynomial $\QCF_P(x,y)$ is equal to $\T_{P,\nabla,\Delta}(1,1;1,x,y)$. Hence by Theorem~\ref{th:CF_Tutte}, the specializations $\T_{P,\nabla,\Delta}(x,y;1,0,0)$ and $\T_{P,\nabla,\Delta}(1,1;1,u,v)$ capture the same information about $P$. 
Again it would be interesting to study the properties of the mixed Tutte polynomials in general, or in the restricted context of matroids.


\subsection{Polymatroid analogues for properties of the classical Tutte polynomial}
\label{subsec:deletion-contraction}
The classical Tutte polynomial is known to enjoy many remarkable properties, and it is natural to wonder if these properties generalize to the polymatroid setting. For instance, does the convolution formula of \cite{KRS} admit a generalization to the polymatroid setting? 
Also, can we consider more general notions of activities for the bases of a polymatroid in the spirit of \cite{GM}?
 
Let us now focus on the deletion-contraction formula. Recall that for a matroid $M$ on the ground set $[n]$, and an element $i\in[n]$, one has
$T_M(x,y)=y T_{M\setminus i}(x,y)$ if $i$ is a loop, 
$T_M(x,y)=y T_{M/ i}(x,y)$ if $i$ is a coloop, 
and 
$$T_M(x,y)=T_{M\setminus i}(x,y)+T_{M/i}(x,y)$$ 
if $i$ is neither a loop nor a coloop of $M$. These relations clearly determine $T_M(x,y)$, and they play an important role in the theory of the Tutte polynomial because the Tutte polynomial is in fact ``universal'' among matroid invariants satisfying linear deletion-contraction relations (see, e.g., \cite{Bollobas}). 
We also note that the ``multivariate'' Tutte polynomials of \cite{BR,Zas,ET,Sok} are defined by adding more parameters to the deletion-contraction recurrence.

As we now explain, Property~\ref{prop:TP_properties}(f) is a partial generalization of the deletion-contraction formula to the polymatroid setting. 
Let $f$ let the rank function of the polymatroid $P(M)\subset \Z^n$. For every element $i\in[n]$, the quantity $r_i:=f(\{i\})+f([n]\setminus \{i\})-f([n])$ is equal to $0$ if $i$ is a loop or coloop of $M$, and equal to 1 otherwise. Furthermore it is easy to check, via Theorem~\ref{thm:oldandnew}, that the relations \eqref{eq:deletion-coloop} and \eqref{eq:deletion-contraction} for $\T_{P(M)}$ specialize to the deletion-contraction formulas for $T_M$.

The main question is whether there exists a more general deletion-contraction relation (true for arbitrary values of $r_i$) at the polymatroid level. 
Let $P$ be a polymatroid on the ground set $[n]$ and let $i\in[n]$. The $i$'th coordinate of a base of $P$ can take any value between $\al_i=f([n])-f([n]\setminus \{i\})$ and $\be_i=f(\{i\})$. Both the ``bottom face'' $\underline P$ consisting of the bases with $i$'th coordinate equal to $\al_i$ and the ``top face'' $\overline P$ consisting of the bases with $i$'th coordinate equal to $\be_i$ are polymatroids. However the other ``slices'' of $P$ are not necessarily polymatroids, and it is not clear how to express $P\setminus(\underline P\cup\overline{P})$ as a union of polymatroids when $r_i=\be_i-\al_i>1$. Still, when, say, $P\setminus \underline P$ is a polymatroid, is it possible to express $\T_P$ in terms of $\T_{\underline P}$ and $\T_{P\setminus \underline P}$? If not, which specializations of $\T_P$ can be expressed in this manner? (Note that $\T_P(1,1)$ always works.) 

In a different direction, suppose that $I\subset [n]$ is such that $r_I:=f(I)+f([n]\setminus I)-f([n])=1$. Can we express $\T_P$ in terms of $\T_{\underline P}$ and $\T_{\overline P}$, where now $\underline P=\{\,\a\in P\mid\sum_{i\in I}a_I=f(I)-1\,\}$ and $\overline P=\{\,\a\in P\mid\sum_{i\in I}a_I=f(I)\,\}$?

Lastly, focusing on hypergraphs, can we find a generalization of the notions of deletion and contraction of edges which would lead to a recurrence for the polymatroid Tutte polynomial of hypergraphs? For instance, for a hypergraph $H$ and a hyperedge~$e$, can we express $\T_{P_H}$ in terms of the polynomials $\T_{P_{H(B)}}$ for all the hypergraphs $H(B)$ obtained from $H$ by choosing a set partition $B$ of $e$, identifying the vertices in each block of $B$, and deleting $e$? Or, if this is not the right approach, then can we extend Kato's formula \cite[Corollary 1.3]{kato} that provides a recurrence relation for the interior polynomial?

\subsection{Relation to other hypergraph and polymatroid invariants}\label{sec:rel-other-invariants}
\OB{I rewrote this subsection.\\}
Many invariants of polymatroids related to the Tutte polynomial have been considered in the literature. We wonder if the polymatroid Tutte polynomial introduced in the present paper is related to those invariants. For the reader's convenience, we now do a quick review of those invariants.

In  \cite{Hel}, Helgason defines the following \emph{Poincar\'e polynomial} for a polymatroid $P\subset\Z^n$ having rank function $f$:
\begin{equation}\label{eq:Poincare-poly}
\scP_P(x,y)=\sum_{I\subseteq [n]}(y-1)^{|I|}x^{f([n])-f(I)}.
\end{equation}
In words, the Poincar\'e polynomial counts subsets of $[n]$ according to their cardinality and rank. When specialized to matroid, this invariant is clearly equivalent (i.e. equal up to a change of variables) to the Tutte polynomial as defined by \eqref{eq:corank-nullity-mat}. The specialization $\scP_P(x,0)$ is called the \emph{characteristic polynomial} of $P$.

Although both $\scP_P$ and $\T_P$ are generalizations of the Tutte polynomial of matroids to the polymatroid setting, it seems unlikely they are equivalent. However, we wonder if they have non-trivial common specializations. For instance, can the characteristic polynomial be obtained as a specialization of $\T_P$. In the other direction, can the number of bases of $P$ (given by $\T_P(1,1)$) be obtained as a specialization of $\scP_P(x,y)$?

Let us now turn to hypergraph invariants. For a bipartite graph $H=(V\sqcup W,E)$ one can consider the hypergraph $\vec H$, and the polymatroid $P_H$ whose rank function is given by \eqref{eq:rank-PH2}.
The Poincar\'e polynomial of $P_H$ is then
\begin{equation}\label{eq:Poincare-PH}
\scP_{P_H}(x,y)=\sum_{I\subseteq W}(y-1)^{|I|}x^{c(I)-c(W)},
\end{equation}
where $c(I)$ is the number of connected components of the subgraph of $H$ induced by the set $V\sqcup I$ of vertices.
Now recall that the \emph{chromatic polynomial}  $\chi_H(q)$ of $\vec{H}$ (as defined in e.g.  \cite{Ber}) is the unique polynomial such that for every positive integer $k$, the evaluation $\chi_H(k)$ counts the number of $k$-colorings of the vertices (i.e. functions $V\to [k]$) without monochromatic hyperedge.
It was shown in \cite{Hel} (see also \cite{Whi}), that the chromatic polynomial $\vec{H}$ is related to the characteristic polynomial of $P_H$ by
$$ \chi_H(q)=q^{c(W)}\scP_{P_H}(q,0)= \sum_{I\subseteq W}(-1)^{|I|}q^{c(I)}.$$
More generally, it is shown in \cite[Section 3.3]{Sta3} that  $\chi_H(q,y)=q^{c(W)}\scP_{P_H}(q,y)$ counts the colorings of $H$ according to the number of monochromatic hyperedges. Precisely, for every positive integer $k$,
$$k^{c(W)}\scP_{P_H}(k,y)=\sum_{k-\textrm{coloring of vertices}}y^{\#\textrm{ monochromatic hyperedges}}.$$
This is a generalization to hypergraphs of the classical relation between the Tutte polynomial and Potts model.  
We also mention another invariant of hypergraphs, which was defined in \cite{Ath} as 
$$\scA_{H}(x,y)=\sum_{I\subseteq W}(y-1)^{|I|}x^{d(I)},$$
where $d(I)=|V|-\min\{|J|~|~J\subseteq W \textrm{ equivalent to } I\}$, where two sets of hyperedges $I,J$ are called \emph{equivalent} if the subgraphs of $H$ induced by $V\sqcup I$ and $V\sqcup J$ have the same connected components. 
The invariants $\chi_{H}(x,y)$ and $\scA_H(x,y)$ coincide for hypergraphs $\vec H$ which are graphs (every hyperedge of cardinality 2). Both are generalizations of the Tutte polynomial of graphs, in the sense that when restricted to graphs, these invariants are equal to the Tutte polynomial up to a change of variables.

More recently, Aval et al. \cite{AKT} studied an invariant of hypergraph $\widetilde{\chi}_H(q)$, which is a different generalization of the chromatic polynomial of (loopless) graphs. Namely,  $\widetilde{\chi}_H(q)$ is the unique polynomial such that for every positive integer $k$, the evaluation $\widetilde{\chi}_H(k)$ counts the number of $k$-colorings of the vertices such that for every hyperedge the maximum color appearing in the hyperedge only appears once in the hyperedge.  When restricted to loopless graphs the invariant $\widetilde{\chi}_H(q)$ is equal to the chromatic polynomial  $\chi_H(q)$.  One could also consider the two-variable extension  $\widetilde{\chi}_H(q,y)$ counting the colorings of $\vec{H}$ according to the number of hyperedges violating the above coloring condition -- and this would again give a  generalization of the Tutte polynomial of graphs.

We wonder again whether  $\widetilde{\chi}_H(q)$ can be obtained as a specialization of  $\T_{P_H}$.
The invariant  $\widetilde{\chi}_H(q)$ can be defined using the Hopf monoid framework of Aguiar and Ardila \cite{AA}. 
Can the polymatroid Tutte polynomial be also described in this manner (using a character of the Hopf monoid of hypergraphs from \cite{AA})? 

Lastly a generalization of the Tutte polynomial from the graph setting to the directed graph setting is defined in \cite{AB}. Since hypergraphs can be seen as a subclass of directed graphs (namely the class of directed graphs such that every vertex is either a source or a sink), it is natural to ask whether there is any relation with the invariant defined in \cite{AB} (which is actually a trivariate polynomial). While we do not know a satisfying answer to this question, we note that the invariant obtained from \cite{AB} is \emph{mirror invariant} (that is, it takes the same value for $H$ and $H^\mir$). This leads us to our next topic.

\subsection{Mirror symmetry} 
\label{sec:mirror-symmetry}

Let us recall the relation established in~\cite{KP} between the interior polynomial of a hypergraph and its transpose. 
Let $H = (V\sqcup W, E)$ be a connected bipartite graph. Then it is proved in~\cite{KP} that 
\begin{equation}\label{eq:mirror}
x^{-|W|}\T_{P_H}(x,1)=x^{-|V|}\T_{P_{H^\mir}}(x,1).
\end{equation}
However the proof of \eqref{eq:mirror} in \cite{KP} is quite complicated. Therefore it would be nice if such a formula could be derived from the results of the present paper, for instance using the explicit expression of $\T_n$ established in Section~\ref{sec:formula_universal_Tutte}.

Another question concerns a possible extension of \eqref{eq:mirror}. As we now explain, the Tutte polynomial of a hypergraph does not determine the Tutte polynomial of its transpose. This can be seen from the case of graphical zonotopes treated in Section~\ref{sec:zonotopes}. 
Let $G=(W,E)$ be a connected graph, and let $Z(G)\subset \Z^W$ be the corresponding zonotope (seen as a polymatroid). Recall from Remark~\ref{rk:zonotope-is-Hmirror} that $Z(G)$ is a graphical polymatroid. Precisely, let $H=(V^+\sqcup W,E')$ be the bipartite graph obtained from $G=(W,E)$ by inserting a vertex $v_e\in V^+$ in the middle of each edge $e\in E$, and adding a vertex $v_0\in V^+$ joined to every vertex in $W$. Then $Z(G)=P_{H}$, and Theorem~\ref{thm:Tuttezonotope} reads
\begin{equation*}
\T_{P_{H}}(x,y)=(x+y-1)\,T_G(x+y,1).
\end{equation*}
Now let us consider the transpose hypergraph $\vec H^\mir:=(\vec H_G^+)^\mir$. It can be shown that 
\begin{equation}\label{eq:transpose-zonotope}
\T_{P_{H^\mir}}(x,y)=x^{|E|-|W|+1}y^{|W|-1}(x+y-1)\,T_G\left(\frac{x+y}{y},\frac{x+y-1}{x}\right).
\end{equation}
Note that $\T_{P_{H}}(x,1)=x T_G(x,1)=x^{|W|-|E|-1}\T_{P_{H^\mir}}(x,1)$, in accordance with \eqref{eq:mirror}.
The proof of \eqref{eq:transpose-zonotope} is left to the reader\footnote{Hint: the spanning hypertrees of $H^\mir$ are in bijection with the spanning forests of $G$ and if we take an order on the set $V^+$ such that $v_0$ is the minimal element, then the external activities of the spanning hypertrees (in the sense of the present paper) can be related to the external activities of the corresponding forests (in the sense of \cite{GT}).}. 
Hence the Tutte polynomial of $\vec H$ is equivalent to $T_G(x,1)$ while the Tutte polynomial of the transpose hypergraph $\vec H^\mir$ is equivalent to the full Tutte polynomial $T_G(x,y)$. While this rules out a direct two-variable extension of \eqref{eq:mirror}, it raises the question of finding a three-variable invariant $\mR_{\vec H}(x,y,y')$ of hypergraphs such that for any hypergraph $\vec H$, we have $\mR_{\vec H}(x,y,y')=\mR_{\vec H^\mir}(x,y',y)$.



\subsection{Zonotopes and $x,y$ symmetry}
In Theorem~\ref{thm:Tuttezonotope}, we give an expression for the polymatroid Tutte polynomial $\T_{Z(G)}(x,y)$ of a graphical zonotope $Z(G)$. Could this expression be obtained directly without using the explicit formula for the universal Tutte polynomial $\T_n$? (the formula of $\T_n$ secretly captures the fact that a hypergraph and its mirror have the same number of spanning hypertrees).

Observe that the polymatroid Tutte polynomial $\T_{Z(G)}(x,y)$ of any graphical zonotope $Z(G)$ is a polynomial in $x+y$. Is it true that $\T_P(x,y)$ is a polynomial in $x+y$ if and only if $P$ is a zonotope? 

Observe more generally that whenever $-P$ is a translate of $P$, we get $T_P(x,y)=T_{-P}(y,x)=T_P(y,x)$. Hence in this case $T_P(x,y)$ is a symmetric polynomial in $x$ and $y$. Are there nice interpretations of the coefficients of $\T_P$ in some basis of the symmetric functions?

\subsection{Miscellaneous}

We leave the reader with a few additional questions.

In Section~\ref{sec:formula_universal_Tutte} we give an explicit formula for the universal Tutte polynomial $\T_n$ based on a formula from \cite{Pos} for the number of points in polymatroids. There are other point-counting formulas in \cite{Pos}, one of which is based on a result of Brion \cite{Brion}. Could these alternative formulas be used to give alternative expressions of $\T_n$, which reveal other properties of this polynomial?

In Proposition~\ref{prop:TP_properties}, we provide interpretations of some evaluations of $\T_P(x,y)$. Is it possible to give combinatorial interpretations of some other specializations (possibly restricting our attention to hypergraphical polymatroids)? We wonder, in particular, if the evaluations $\T_P(1,2)$, $\T_P(2,1)$, and $\frac{\T_P(x,y)}{x+y-1}\bigg|_{x=1/2,~y=1/2}$ can be given a combinatorial meaning.

Lastly, let us recall that the original motivation for introducing the interior and exterior polynomials in \cite{Kal} was to give a combinatorial interpretation for some of the coefficients of the HOMFLY polynomials of knots associated to plane bipartite graphs. Can the polymatroid Tutte polynomial be interpreted as a knot invariant, or can it at least be used to further our understanding of (quantum) knot invariants?\\

\medskip

\noindent {\bf Acknowledgments.} We thank Spencer Backman for stimulating discussions about generalizing Crapo’s interval decomposition to the polymatroid setting. We thank  Gleb Koshevoy, Klas Markstr\"om, Richard Stanley, and Lorenzo Traldi for pointing out relevant references.

\end{document}